\documentclass[letterpaper, 11pt,  reqno]{amsart}

\usepackage[margin=1.2in,marginparwidth=1.5cm, marginparsep=0.5cm]{geometry}

\usepackage{amsmath, mathrsfs, amssymb,amscd,amsthm,amsxtra, esint}
\usepackage{tikz} 
\usepackage{color}
\usepackage{xcolor}

\usepackage[implicit=true]{hyperref}

\usepackage{cases}%%%

\usepackage{upgreek}

\allowdisplaybreaks[2]

\definecolor{gr}{rgb}   {0.,   0.69,   0.23 }
\definecolor{bl}{rgb}   {0.,   0.5,   1. }
\definecolor{mg}{rgb}   {0.85,  0.,    0.85}
\definecolor{yl}{rgb}   {0.8,  0.7,   0.}
\definecolor{or}{rgb}  {0.7,0.2,0.2}

\usetikzlibrary{shapes.misc}
\usetikzlibrary{shapes.symbols}
\usetikzlibrary{shapes.geometric}

\tikzset{
	ddot/.style={circle,fill=white,draw=black,inner sep=0pt,minimum size=0.8mm},
	>=stealth,
	}

\tikzset{
	ddot2/.style={circle,fill=black,draw=black,inner sep=0pt,minimum size=0.8mm},
	>=stealth,
	}

\newtheorem{theorem}{Theorem} [section]

\newtheorem{lemma}[theorem]{Lemma}
\newtheorem{proposition}[theorem]{Proposition}
\newtheorem{remark}[theorem]{Remark}

\newtheorem{definition}[theorem]{Definition}

\newtheorem{oldtheorem}{Theorem}

\DeclareMathOperator*{\intt}{\int}

\DeclareMathOperator*{\supp}{supp}
\DeclareMathOperator{\med}{med}

\DeclareMathOperator{\HS}{HS}

\DeclareMathOperator{\Id}{Id}
\DeclareMathOperator{\sgn}{sgn}

\DeclareMathOperator{\Ker}{Ker}

\newcommand{\I}{\mathcal{I}}

\newcommand{\noi}{\noindent}
\newcommand{\Z}{\mathbb{Z}}
\newcommand{\R}{\mathbb{R}}

\newcommand{\T}{\mathbb{T}}
\newcommand{\bul}{\bullet}

\let\Re=\undefined\DeclareMathOperator*{\Re}{Re}
\let\Im=\undefined\DeclareMathOperator*{\Im}{Im}

\let\P= \undefined
\newcommand{\P}{\mathbf{P}}

\newcommand{\PP}{\mathbb{P}}

\newcommand{\E}{\mathbb{E}}

\renewcommand{\H}{\mathcal{H}}

\newcommand{\CC}{\mathcal{C}}
\newcommand{\D}{\mathcal{D}}

\renewcommand{\L}{\mathcal{L}}

\newcommand{\RR}{\mathbf{R}}

\newcommand{\F}{\mathcal{F}}

\newcommand{\al}{\alpha}
\newcommand{\be}{\beta}
\newcommand{\dl}{\delta}
\newcommand{\updl}{\updelta}

\newcommand{\nb}{\nabla}

\newcommand{\Dl}{\Delta}
\newcommand{\eps}{\varepsilon}
\newcommand{\kk}{\kappa}
\newcommand{\g}{\gamma}
\newcommand{\G}{\Gamma}
\newcommand{\ld}{\lambda}
\newcommand{\Ld}{\Lambda}
\newcommand{\s}{\sigma}

\newcommand{\ft}{\widehat}
\newcommand{\Ft}{{\mathcal{F}}}
\newcommand{\wt}{\widetilde}
\newcommand{\cj}{\overline}
\newcommand{\dx}{\partial_x}
\newcommand{\dt}{\partial_t}
\newcommand{\dd}{\partial}

\newcommand{\ta}{\theta}

\renewcommand{\l}{\ell}
\renewcommand{\o}{\omega}

\newcommand{\Gdl}{\mathcal{G}_{\dl} }

\newcommand{\les}{\lesssim}
\newcommand{\ges}{\gtrsim}

\newcommand{\jb}[1]
{\langle #1 \rangle}

\newcommand{\ind}{\mathbf 1}

\renewcommand{\S}{\mathcal{S}}

\newcommand{\M}{\mathcal{M}}

\newcommand{\N}{\mathbb{N}}

\newcommand{\NN}{\mathcal{N}}
\newcommand{\cL}{\mathcal{L}}
\newcommand{\cC}{\mathcal{C}}

\newcommand{\cX}{\mathcal{X}}
\newcommand{\cY}{\mathcal{Y}}

\newcommand{\W}{\mathcal{W}}

\newcommand{\Lip}{\mathrm{Lip}}
\newcommand{\uw}{U^w}
\newcommand{\sw}{S^w}

\newcommand{\uu}{\mathbf{u}}
\newcommand{\vv}{\mathbf{v}}

\newcommand{\z}{\zeta}

\newcommand{\Ta}{\Theta}

\newcommand{\BO}{\text{\rm BO} }
\newcommand{\KDV}{\text{\rm KdV} }

\newcommand{\too}{\longrightarrow}

\usepackage[most]{tcolorbox}

\newcommand{\XXb}{\mathbb{X}}
\newcommand{\YYb}{\mathbb{Y}}

\newcommand{\Yb}{\YY^{\mathfrak{b}}}

\newcommand{\Yg}{\YY^{\mathfrak{g}}}

\newcommand{\PPsi}{\pmb{\Psi}}

\newcommand{\Hs}{\mathscr{H}}

\newcommand{\HH}{\mathcal{H}}

\newcommand{\Sym}{\textup{\texttt{Sym}}}
\newcommand{\Om}{\Omega}

\newcommand{\LOP}{\mathcal{L}}

\newcommand{\XX}{\mathbf{X}}
\newcommand{\YY}{\mathbf{Y}}

\newcommand{\hf}{\mathfrak{h}}
\newcommand{\ff}{\mathfrak{g}}

\def\doublestroke#1{\pdfliteral{1 Tr .35 w}#1\pdfliteral{0 Tr 0 w}}

\newcommand{\Taa}{\doublestroke{\Theta}}
\newtheorem*{ackno}{Acknowledgements}

\numberwithin{equation}{section}
\numberwithin{theorem}{section}

\makeatletter
\@namedef{subjclassname@2020}{\textup{2020} Mathematics Subject Classification}
\makeatother

\begin{document}
\baselineskip = 14pt

\title[Nonlinear PDEs with modulated dispersion III]
{Nonlinear PDEs with modulated dispersion III:\\
multiplicative noises}

\author[A.~Chapouto, M.~Gubinelli, G.~Li, J.~Li, and T.~Oh]
{Andreia Chapouto, Massimiliano Gubinelli, Guopeng Li, Jiawei Li,\\
 and Tadahiro Oh}

\address{
Andreia Chapouto\\
CNRS, Laboratoire de math\'ematiques de Versailles, UVSQ, Universit\'e Paris-Saclay, CNRS, 45 avenue des 
\'Etats-Unis, 78035 Versailles Cedex, France, 
and School of Mathematics, Monash University, VIC 3800, Australia}

\email{andreia.chapouto@monash.edu}

\address{Massimiliano Gubinelli\\
Mathematical Institute\\ University of Oxford\\ 
Woodstock Road,\\
 Oxford,  OX2 6GG, United Kingdom}

\email{gubinelli@maths.ox.ac.uk}

\address{Guopeng Li, 
School of Mathematics and Statistics, Beijing Institute of Technology, Beijing 100081, China}

\email{guopeng.li@bit.edu.cn}

\address{Jiawei Li, School of Mathematics\\
The University of Edinburgh\\
and The Maxwell Institute for the Mathematical Sciences\\
James Clerk Maxwell Building\\
The King's Buildings\\
Peter Guthrie Tait Road\\
Edinburgh\\ 
EH9 3FD\\
 United Kingdom}

\email{jiawei.li@ed.ac.uk}

\address{
Tadahiro Oh, 
School of Mathematics\\
The University of Edinburgh\\
and The Maxwell Institute for the Mathematical Sciences\\
James Clerk Maxwell Building\\
The King's Buildings\\
Peter Guthrie Tait Road\\
Edinburgh\\
EH9 3FD\\
 United Kingdom,
and School of Mathematics and Statistics, Beijing Institute of Technology, Beijing 100081, China}

%%
%\address{
%Tadahiro Oh, 
%School of Mathematics and Statistics, Beijing Institute of Technology, Beijing 100081, China\\
%and
%School of Mathematics\\
%The University of Edinburgh\\
%and The Maxwell Institute for the Mathematical Sciences\\
%James Clerk Maxwell Building\\
%The King's Buildings\\
%Peter Guthrie Tait Road\\
%Edinburgh\\
%EH9 3FD\\
% United Kingdom}
%

\email{hiro.oh@ed.ac.uk}

% 65M12}


\begin{abstract}

We investigate  pathwise  well-posedness of the stochastic modulated Korteweg-de~Vries equation (KdV) 
on the circle with a {\it multiplicative} noise, where a time non-homogeneous modulation acts on the linear dispersion term. (i)~In the Young case (= fractional-in-time case with Hurst parameter greater than~$\frac 12$), we establish a new regularization-by-noise phenomenon on the stochastic convolution in a pathwise manner, where a gain of spatial regularity  becomes (arbitrarily) larger for more irregular modulations. We then prove that, given {\it any} $s \in \mathbb R$ and any  multiplicative Young noise, {\it however rough it is in space}, the stochastic modulated KdV is pathwise locally well-posed in $H^s(\mathbb T)$, provided that the modulation is sufficiently irregular. (ii)~In the rough case (= white-in-time case), irregularity of the modulation does not induce any smoothing on the stochastic convolution, and in fact, there is a slight loss in the spatial regularity. In this case, by slightly regularizing the multiplicative noise term, we prove pathwise local well-posedness in $H^s(\mathbb T)$ for any given $s \in \mathbb R$, provided that  the noise is sufficiently smooth in space. We achieve these goals by combining (i) the sewing lemma approach to the nonlinear Young integration theory, introduced by Chouk and the second author (2014), and (ii) the pathwise construction of stochastic convolutions as Young or rough integrals via the random tensor estimate and the sewing lemma, introduced by the first, fourth, and fifth authors (2026). In the appendix, we also present an example of regularization by noise for a stochastic modulated Schr\"odinger equation with a multiplicative Young noise. 

\end{abstract}

%\date{\today}
%%
%
\maketitle

\tableofcontents

\section{Introduction}\label{SEC:1}
\subsection{Stochastic modulated dispersive equations with multiplicative noises}
In this paper, we study  a stochastic modulated dispersive 
partial differential equation (PDE)
with a multiplicative noise, 
posed on $  \T^d = \R/ (2\pi \Z)^d$:\footnote{By convention, we endow
$\T^d$ with the normalized Lebesgue measure $ dx_{\T^d} =  (2\pi)^{-d}dx$
such that we do not need to carry factors involving $2\pi$.}
\begin{align}
%\begin{cases}
\dt u +  L u \cdot \dt w=  \NN(u)  + u \phi\zeta, 
%u|_{t = 0} = u_0,
%\end{cases}
\qquad ( t, x) \in \R_+ \times \T^d,
\label{ME1}
\end{align}

\noi
where
 $\NN(u)$ denotes the nonlinearity
 and 
 $w:\R_+\to\R$ is a continuous function of time, 
called a {\it modulation}, 
 acting on the linear dispersion term $Lu$.
Here, 
$\phi$ is a Hilbert-Schmidt operator from $L^2(\T^d)$ to % = \P_{\ne 0}L^2(\T)$ to 
$H^\s(\T^d)$ 
for some $\s \in \R$  
(such that\footnote{More precisely, $\phi W^\be(t)$ has spatial regularity $\s$, 
where $W^\be$ is as in \eqref{W1}.}
 the noise $\phi \zeta$ has spatial regularity $\s$)
and $\zeta$ denotes a fractional-in-time\,/\,white-in-time and white-in-space noise.
Heuristically, one may think of $\zeta$ as 
\begin{align}
\text{``}\,\z = \jb{\dt}^{-\al}\xi\,\text{''}
\label{W0a}
\end{align}
for some $0 \le \al <  \frac 12$, where 
$\jb{\,\cdot\,} = (1+ |\cdot|^2)^\frac12$ and 
$\xi$ denotes 
a (Gaussian) space-time white noise on $\R_+ \times \T^d$
whose space-time covariance is (formally) given by 
\begin{align*}
 \E[ \xi(t_1, x_1)\xi(t_2, x_2) ] = \dl(t_1 - t_2) \dl (x_1 - x_2)
\end{align*} 

\noi
for $t_1, t_2 \in \R_+$  and $x_1, x_2 \in \T^d$
with $\dl$ denoting the Dirac delta function.
See \eqref{ze1}
for the precise meaning of $\z$.

We say that $u$ is a solution to \eqref{ME1}
with initial data $u|_{t = 0} = u_0$
if $u$ satisfies the following
 Duhamel formulation (= mild formulation):
\begin{align}
\begin{split}
u(t) 
& = U^w(t) u_0 +  U^w({t}) \int_0^t U^w({t'})^{-1}\NN( u(t') ) dt' + \Psi(u) (t), 
\end{split}
\label{mild1}
\end{align}

\noi
where 
\begin{align}
U^w(t) = e^{- w(t) L }
\label{lin1}
\end{align}

\noi
 denotes the modulated linear propagator
such that\footnote{By setting $w(0) = 0$, we can guarantee
  $U^w(0) = \Id$.
 Note that 
 this normalization 
 is not an additional restriction since only the time derivative $\dt w$
 appears in \eqref{ME1} and \eqref{ME2}.}
  $U^w(0) = \Id$
and $\Psi(u)$ denotes the stochastic convolution,\footnote{In the current modulated setting, 
due to the lack of the group property (i.e.~$\uw(t)\uw(t')^{-1} \ne \uw(t - t')$ in general), $\Psi(u)$
in \eqref{psi1} does not have a `convolution' structure.
Nonetheless, we refer to $\Psi(u)$ as a stochastic convolution 
in this paper (just as in the unmodulated setting).} 
representing the effect of the stochastic forcing in \eqref{ME1}, 
defined by 
\begin{align}
\Psi(u)(t) = 
U^w(t) \int_0^t U^w(t')^{-1}
\big[u (t') \phi dW^\be(t')\big].
\label{psi1}
\end{align}

\noi
Here, 
$W^\be$ denotes the cylindrical fractional Wiener process on $L^2(\T^d)$ given by 
\begin{align}
W^\be(t) 
 = \sum_{n\in \Z^d}  
B_n(t) e_n, 
\label{W1}
\end{align}

 \noi
where $e_n(x) = e^{in\cdot x}$
and $\{B_n \}_{n \in \Z^d}$ is a family of independent 
complex-valued 
fractional Brownian motions with Hurst parameter\footnote{Note that $\al$ in \eqref{W0a}
is given by $\al = \be -\frac 12$.} $\frac 12 \le \be < 1$, 
where,  in the real-valued setting, 
we additionally impose that\footnote{In particular, in the real-valued setting, 
the family $\{B_n\}_{(\Z^d)_{+,0}}$ is independent, where 
$(\Z^d)_{+,0}$ is as in~\eqref{fBM0}.
} 
\begin{align}
B_{-n} = \cj{B_n}, \quad n \in \Z^d.
\label{W2}
\end{align}

\noi
Namely, we have
\begin{align}
\z = \dt W^\be.
\label{ze1}
\end{align}

\noi
%Moreover, in the real-valued setting, 
%$B_n$'s are 
%conditioned that 
%\begin{align}
%B_{-n} = \cj{B_n}, \quad n \in \Z^d.
%\label{W2}
%\end{align}
When $\be = \frac 12$, 
the sequence $\{B_n \}_{n \in \Z^d}$ 
reduces to a family of independent 
standard
complex-valued 
 Brownian motions (satisfying \eqref{W2} in the real-valued case)
 and thus $W^{ \frac 12}$ corresponds 
 to an $L^2$-cylindrical Wiener process.
See Subsection \ref{SUBSEC:FBM}
for a review on fractional Brownian motions and 
stochastic integrals with respect to them.

Our main goal in this paper is two-fold:
(i)  
We build
{\it pathwise} local well-posedness theory for stochastic modulated
dispersive PDEs.
The main strategy is to combine

\smallskip
\begin{itemize}
\item
the sewing lemma approach to 
the nonlinear Young integration theory, 
introduced by Chouk and the second author \cite{CG1, CGLLO1}, and 

\smallskip
\item 
the pathwise construction of stochastic convolutions
as Young or rough integrals, 
introduced by the first, fourth, and fifth authors 
in the unmodulated setting
\cite{CLO2, CLO3, COZ, CLOO}, 
via 
the random tensor estimate
(Lemma \ref{LEM:RT})
and the sewing lemma
(Lemma~\ref{LEM:sew}).

\end{itemize}
\smallskip

\noi
(ii)~We
 present novel {\it regularization-by-noise} phenomena
on the stochastic convolution $\Psi (u)$ in \eqref{psi1}
by exploiting the nonlinear interaction between the unknown $u$ and the noise $\phi \zeta$.
In particular, 
we show that 
{\it irregularity} of the modulation function~$w$ (see Definition~\ref{DEF:ir}) induces
smoothing on the stochastic convolution $\Psi(u)$ 
in a {\it pathwise} manner, 
where 
a gain of spatial regularity  becomes (arbitrarily) larger for more irregular modulations.

See Subsections \ref{SUBSEC:main1}
and \ref{SUBSEC:main2}
for the statements of our main results.
We emphasize that the pathwise approach we develop
in this paper is easily applicable to a general
class of stochastic modulated dispersive PDEs.

\subsection{Background}

Let us first   consider  the case $\phi = 0$.
In this case,  the equation~\eqref{ME1} reduces
to the following
(deterministic) modulated dispersive equation:
\begin{align}
%\begin{cases}
\dt u +  L u \cdot \dt w=  \NN(u)
\label{ME2}
\end{align}

\noi
which naturally appears
in various physical contexts.
For example, 
the  modulated 
 Korteweg-de Vries equation (KdV):
\begin{equation}
\label{kdv1}
\dt u+  \dx^3 u \cdot \dt w =\dx u^2
\end{equation}

\noi
appears 
as a model 
for
weakly nonlinear long waves in an inhomogeneous waveguide;
see \cite{CMG, HZ}, where 
the modulation $w$ is taken to be periodic but not differentiable.
See also~\cite{DD, CG1}
for the references therein
on the physical background of the modulated nonlinear Schr\"odinger equation.

Over the last fifteen years, 
the modulated dispersive equation \eqref{ME2} has been studied extensively
from the viewpoints of both stochastic analysis and rough analysis.
For a specific modulation such as that given by a Brownian motion, 
stochastic calculus can be used to study 
the equation~\eqref{ME1};  see \cite{DD, DT, DR2, Ste}.
For a general modulation $w$, however, 
such an approach based on stochastic calculus
is not available.

In \cite{CG1},   %, CGLLO1},  
Chouk and the second author 
proposed 
  a novel  pathwise approach 
to study the modulated dispersive equation~\eqref{ME1}
from the viewpoint of rough analysis.
By writing~\eqref{ME2} in the  Duhamel formulation, we have
\begin{align}
u(t) 
& = \uw(t) u_0 +  U^w({t}) \int_0^t U^w({t'})^{-1}\NN( u(t') ) dt', 
\label{mild1a}
\end{align}

\noi
where 
$\uw(t)$ is as in \eqref{lin1}.
Let $\uu$ be 
 the 
 modulated interaction representation
 of the unknown~$u$, defined by 
\begin{align}
\uu(t)=\uw(t)^{-1}u(t).
\label{int1} 
\end{align}

\noi
Then, \eqref{mild1a} becomes
\begin{equation}
\uu(t) = u_0 + \int_0 ^t \uw(t')^{-1}\NN(\uw(t') \uu(t'))dt'.
\label{mild2}
\end{equation}

\noi
In \cite{CG1},   %, CGLLO1},  
Chouk and the second author 
first  made sense of the integral term in \eqref{mild2}
as a nonlinear Young integral
$\I^\XX(\uu)$ with  the driver $\XX$ associated with \eqref{ME2}
(see, for example,~\eqref{K1x}
for the modulated KdV)
and then studied the 
 resulting nonlinear Young differential equation (YDE):
\begin{equation}
\uu = u_0 + \I^\XX(\uu)
\label{YDE0a}
\end{equation}

\noi
by a contraction argument. 
This approach was further developed in \cite{CGLLO1};
see \cite[Section~3]{CGLLO1}
for a general review on this approach via the 
nonlinear Young integration. %sewing lemma. % due to the second author \cite{Gub04}.
See also Section~\ref{SEC:Young} below.
In the subsequent work~\cite{CGLLO1}, 
Chouk and the last four authors 
exhibited striking regularization-by-noise phenomena
for the modulated 
KdV \eqref{kdv1} and other modulated dispersive PDEs
 on both the circle and the real line.
For example, 
they  proved that, given {\it any} $s \in \mathbb R$, the modulated KdV~\eqref{kdv1} on $\T$ with a sufficiently irregular modulation is locally well-posed in $H^s(\T)$;
see Theorem~\ref{THM:B} below.
In the same paper, %~\cite{CGLLO1}, 
they also established 
other forms of regularization by noise
such as 
semilinearization and nonlinear smoothing
(of arbitrary order).
See~\cite{CGLLO1}
for a further discussion on various examples.
We also mention 
recent works
\cite{Tanaka, Robert1, Robert2, Robert3, Robert4}
on pathwise well-posedness of  various modulated dispersive equations;
see \cite[Remark~1.14]{CGLLO1}
for a discussion on these works.

\medskip

In the works \cite{CG1, CGLLO1} mentioned above, 
the following   notion of irregularity of the modulation function~$w$, 
introduced in~\cite{CatellierGubinelli, CG1}, 
played a fundamental role.

\begin{definition}
\label{DEF:ir}
\rm

Let $\rho>0$ and  $0 < \g < 1$.
Given $T > 0$, we say that a function $w\in C([0,T];\R)$ is $(\rho,\g)$-irregular 
on the time interval $[0, T]$  if we have
\begin{align}
\|\Phi^w\|_{  \W^{\rho,\g}_T} 
:= \sup_{a\in \R} \sup_{0\leq r < t\leq T} \langle a \rangle^\rho \frac{|\Phi^w_{t,r}(a)|}{|t-r|^\g} 
< \infty, 
\label{rho1}
\end{align}

\noi
where 
\begin{align}
\Phi^w_{t,r}(a)
=\int_r^t e^{i  a w(t') } d t'.
\label{rho2}
\end{align}

\noi
We say that $w$ is 
$(\rho,\g)$-irregular 
on $\R_+$  if it is 
$(\rho,\g)$-irregular 
on $[0, T]$ for each finite $T > 0$.

\end{definition}

For further studies on the notion of irregularity, 
we refer readers to 
 recent works \cite{GG, RT}.
The following results from \cite{CG1, GG}
provide important examples of $(\rho, \g)$-irregular paths.
The term   ``almost every''
in Theorem \ref{THM:A}\,(ii)
 is to be understood according to the notion of {\it prevalence}; see~\cite{GG} for details and the  references therein.

\begin{oldtheorem}
\label{THM:A}
\textup{(i)}
Let $\{W_t\}_{t\in \R_+}$ be a fractional Brownian motion 
of Hurst index $\be\in(0,1)$.
Then, 
for any $\rho < \frac{1}{2\be}$,  
there exists $\frac 12 < \g < 1$
such that,  with probability one,  the sample paths of 
$W$ are $(\rho,\g)$-irregular on $\R_+$.

\smallskip

\noi
\textup{(ii)}
Let  $d \ge 1$.
Given  any $0 < \delta< 1$, almost every $\dl$-H\"older continuous  function $w \in C^\delta([0,1];\R^d)$ 
 is $(\rho,\g)$-irregular for any $\rho < \frac1{2\delta}$ with  some $\gamma = \gamma(\rho) \in (\frac 12,1)$.

\end{oldtheorem}

With Definition \ref{DEF:ir} in hand, 
we now recall the local well-posedness
and nonlinear smoothing results for  the modulated KdV \eqref{kdv1} on the circle;
see 
 \cite[Theorems 1.3 and 1.15]{CGLLO1}.

\begin{oldtheorem}\label{THM:B}
Given $\rho \ge\frac 12$,  $\frac12< \g < 1$, and $T> 0$, 
let  $w$ be $(\rho,\g)$-irregular on $[0, T]$ in the sense of Definition~\ref{DEF:ir}.

\smallskip

\noi
\textup{(i) (local well-posedness).}
Suppose that $\rho \ge \frac 12$ and $s\in \R$
satisfy one of the following conditions\textup{:}
\begin{align}
\begin{split}
\textup{(i.a)} &\ \  \tfrac 12 \le \rho \le \tfrac 34 \quad \text{and} \quad 
s > \tfrac 32 - 3 \rho, \\
\textup{(i.b)} &\ \  \rho > \tfrac 34\quad  \text{and} \quad  s \ge - \rho.
\end{split}
\label{reg1}
\end{align}

\noi
Then, the modulated KdV equation \eqref{kdv1}
on  $\T$
is locally well-posed in $H^s(\T)$.

\smallskip

\noi
\textup{(ii) (nonlinear smoothing).}
Suppose that $\rho \ge \frac 12$ and $s\in \R$
satisfy 
 \eqref{reg1}.
In addition, 
suppose that $s_0 > s$ satisfies 
 one of the following conditions\textup{:}
\begin{align}
\begin{split}
\textup{(ii.a)} &\ \  
0 \le s + \rho \le  \tfrac 12
 \quad \text{and} \quad 
s_0 < 2s +  3\rho - \tfrac 32, \\
\textup{(ii.b)} &\ \  
s + \rho >  \tfrac 12 
\quad  \text{and} \quad s_0 \le s + 2\rho - 1.
\end{split}
\label{reg3}
\end{align}

\noi
Given $u_0 \in H^s(\T)$, 
let $u
\in C([0, \tau]; H^s(\T))$ be the  solution 
to the modulated KdV equation~\eqref{kdv1}
on $\T$
with $u|_{t = 0} = u_0$, 
where 
$\tau \in (0,  T]$ denotes the local existence time.
Then, we have 
\begin{align*}
 u - e^{-w(t) \dx^3}u_0 \in C([0, \tau]; H^{s_0}(\T)).
\end{align*}

\end{oldtheorem}

The proof of Theorem \ref{THM:B} in \cite{CGLLO1} is based on 
making sense of the second term 
on the right-hand side of  \eqref{mild2}
as the nonlinear Young integral 
 $\I^\XX(\uu)$
associated  with 
the 
modulated KdV  driver $\XX = \XX^w$ defined by 
 \begin{align}
\XX_{t,r}(f_1,f_2)
=\int_r^t  \uw(t')^{-1}  \partial_x
\big( (\uw(t')  f_1)( \uw(t') f_2 )\big) dt'
\label{K1x}
\end{align}

\noi
for $0 \le r < t \le T$, where  $f_1$ and $f_2$ are functions on $\T$, 
and then studying local well-posedness
of the resulting 
 nonlinear YDE 
\eqref{YDE0a}.
See Subsection \ref{SUBSEC:Y2}
for a review of the nonlinear Young integration theory.
Then, the matter  reduces
to studying 
 regularity properties of the  driver~$\XX$
by making use of the $(\rho, \g)$-irregularity 
of the modulation function~$w$;
see 
Lemma~\ref{LEM:kdv1}.

\medskip

Theorem \ref{THM:B}
exhibits regularization by noise
in several respects.
First of all, 
given {\it any} $s \in \R$, 
the modulated KdV \eqref{kdv1} on $\T$
is locally well-posed
in $H^s(\T)$
by taking $\rho \gg1$ sufficiently large, 
whereas 
 the (unmodulated) KdV equation on $\T$:
\begin{equation}
\dt u+  \dx^3 u  =\dx u^2
\label{kdv2}
\end{equation}

\noi
is known to be ill-posed in $H^s(\T)$ for $s < -1$;
see
\cite{KPV96, CKSTT03, 
KT1, M12,  KV19}
for the known well-\,/\,ill-posedness
results for the (unmodulated) KdV \eqref{kdv2} on $\T$.
Furthermore, 
the gain of spatial regularity 
on the nonlinear part $ u - e^{-w(t) \dx^3}u_0 $
can be arbitrarily large by taking $\rho \gg1$
sufficiently large.
See \cite{CGLLO1}
for a further discussion on other aspects of regularization by noise.

In the context of stochastic differential equations (SDEs) 
and stochastic parabolic PDEs, 
there  have been intensive research activities
on regularization by noise 
since the 70's;
see 
 \cite{FGP, GG0, G22, RT, CGLLO1} and the  references therein.
See also   survey works \cite{Flan, Gess}.
Prior to the work~\cite{CGLLO1}, 
however, 
regularization by noise 
in the context of 
dispersive PDEs
was known only 
for probabilistic well-posedness with  random initial data and\,/\,or additive noises of super-critical regularity
(such as \cite{BO96, BT08, CO, OP, GKO2, DNY2, DNY3, OOT1, Bring, OOT2, BDNY});
see \cite[Remark~1.8]{CGLLO1}.
In this paper, we 
take 
 the  stochastic modulated
KdV (SmKdV) on $\T$ with a multiplicative noise:
\begin{align}
\dt u+  \dx^3 u \cdot \dt w = \dx u^2 + u \phi \zeta
\label{skdv1}
\end{align}

\noi
as our primary example, 
and establish its pathwise local well-posedness
in both the fractional-in-time case (= Young case: $\frac 12 < \be < 1$, 
where $\be$ denotes the Hurst parameter of 
 the noise $\phi \z$)
and the white-in-time case (= rough case: $\be = \frac 12$).
Furthermore, 
in  the Young case, 
we 
 present
a  novel regularization-by-noise phenomenon
by pathwise analysis.
In Appendix~\ref{SEC:NLS},
we also present an example of regularization by noise
by studying a stochastic modulated
Schr\"odinger equation
with a multiplicative Young noise.

\begin{remark}\label{REM:NF}\rm 
In \cite{CGLLO1}, 
 the (sufficient) temporal regularity of
the modulated interaction representation~$\uu$
was essential 
in applying the sewing lemma (Lemma \ref{LEM:sew})
to construct
a nonlinear Young integral
$\I^\XX(\uu)$ in \eqref{YDE0a}.
Our approach to study the stochastic modulated 
dispersive PDE \eqref{ME1} with a multiplicative noise
is also based on the sewing lemma, 
thus requiring  sufficient temporal regularity 
on the modulated interaction representation~$\uu$ of the unknown.
In a recent work
\cite{GLLO}, a sequel to the current paper, 
the last four authors 
developed a normal form approach 
to study the modulated dispersive PDE \eqref{ME2}
{\it without}
imposing any 
 positive temporal regularity
 on the modulated interaction representation $\uu$, 
 thus providing a significant improvement
 in terms of temporal regularity.
% as compared to the known results.
See~\cite{GLLO}
for further details.

\end{remark}

\begin{remark}\rm

Theorem \ref{THM:A} 
implies that Theorem \ref{THM:B}
applies 
to the case
where a modulation is given by a fractional Brownian motion
or a generic  $\dl$-H\"older continuous function.
See, for example, \cite[Corollary~1.5]{CGLLO1}.
A similar comment applies to our main results 
presented in Subsections \ref{SUBSEC:main1}
and  \ref{SUBSEC:main2}
and Appendix \ref{SEC:NLS}
for stochastic modulated dispersive PDEs.

\end{remark}

\subsection{Main results, Part 1: Young case} 
\label{SUBSEC:main1}

In this section and the next, 
we state our main results on 
pathwise local well-posedness
of SmKdV~\eqref{skdv1}.
Before proceeding further, let us introduce a simplification 
of the problem.
Recall from~\cite{CGLLO1}
that it is essential to impose 
the mean-zero assumption
on the unknown $u$
in order to have a (useful) nonlinear estimate
for the modulated KdV \eqref{kdv1} on $\T$.\footnote{The same comment also applies
to the unmodulated KdV \eqref{kdv2} on $\T$; see \cite{BO93, KPV96, CKSTT03}.}
For the (deterministic) modulated KdV on $\T$, 
if initial data has non-zero mean $\al_0$, then the following Galilean transformation:
\begin{align}
u(t, x)\mapsto u (t, x - 2\al_0t)-\al_0,
\label{gauge1} 
\end{align}
along with the conservation of the spatial mean (which can be verified
by arguing as in \cite[Section 6]{CGLLO1}), 
transforms
the equation with a non-zero mean into the mean-zero one.
See~\cite{Oh2} for such a reduction 
in the presence of an additive noise (in the unmodulated setting).

Let us turn our attention to \eqref{skdv1}.
Suppose that
initial data has non-zero mean $\al_0$.
Then, the Galilean transformation \eqref{gauge1}
converts \eqref{skdv1} into the following equation:
\begin{align}
\dt u+  \dx^3 u \cdot \dt w = \dx u^2 + 
\al_0  \phi \wt \zeta
+ u \phi \wt \zeta, 
\label{skdv1a}
\end{align}

\noi
where $\wt \z(t, x) = \z(t, x - 2\al_0 t)$,
such that the initial data now has mean zero.
Unfortunately,~\eqref{skdv1a} does not preserve the spatial mean of a solution and thus 
a further reduction is required to 
reduce the equation \eqref{skdv1a}
into a system of the (unmodulated) SDE for the spatial mean $\al(t)$
and the (mean-preserving) stochastic modulated KdV with an additive noise
and a multiplicative noise for the mean-zero part;
see \cite{CCO} for such a reduction in the unmodulated setting.\footnote{If we assume that 
the noise and a solution have sufficient spatial regularities, 
then we do not need to impose the mean-zero assumption; see \cite{Tsutsumi}.}
While it is possible to carry out our analysis on such a system, 
it would blur the focus of our main result.
Therefore, 
for simplicity of presentation, 
we choose to study the following slightly simplified model on $\T$:
\begin{align}
\begin{cases}
\dt u+  \dx^3 u \cdot \dt w = \dx u^2 + \P_{\ne 0}[u \phi \zeta]\\
u|_{t = 0} = u_0,
\end{cases}
\label{skdv2}
\end{align}

\noi
which preserves the spatial mean of a solution.
Here, $\P_{\ne 0}$ denotes
the projection onto non-zero (spatial) frequencies.
In the following, 
we assume that  the initial data $u_0$ (and hence a solution $u$) has  spatial mean zero.
As a further  simplification,  we only consider the case when 
$\phi$ is spatially homogeneous, namely 
$\phi$ is a Fourier multiplier operator
given by 
\begin{align}
\phi (e_n) = \phi_n e_n.
\label{phi1}
\end{align}

\noi
Under this assumption, we have 
\begin{align}
\| \phi \|_{\HS(L^2; H^\s)}= \|\jb{n}^\s\phi_n\|_{\l^2_n}.
\label{phi1a}
\end{align}

By writing \eqref{skdv2} 
in the Duhamel formulation
(at the level of the modulated interaction representation
$\uu$ defined in \eqref{int1}), 
we have 
\begin{equation}
\begin{split}
\uu(t) 
& = u_0 + \int_0 ^t \uw(t')^{-1}
\dx\big( (\uw(t') \uu(t'))^2\big)dt'
+ \PPsi(\uu) (t), 
\end{split}
\label{mild1b}
\end{equation}

\noi
where 
$\uw (t)=e^{-   w(t)\dx^3}  $
denotes the modulated linear propagator for \eqref{skdv2}
and $\PPsi(\uu)$ is given by 
\begin{align}
\begin{split}
 \PPsi(\uu) (t) 
 &  =  \P_{\ne 0} \int_0^t
\uw(t')^{-1} \big[ (\uw(t')\uu(t') )
\phi dW^\be(t') \big]\\
&  = \sum_{n\in \Z_*} e_n 
 \int_0^t \sum_{n_1,  n_2  \in \Z}  %_{\substack{n_1, n_2  \in \Z\\n_2 \ne 0}} 
 \ind_{n = n_1+n_2}
\cdot  e^{- i w(t') (n^3 - n_2^3)}\ft \uu(t', n_2) \phi_{n_1}    d B_{n_1}(t').
\end{split}
\label{psi2}
\end{align}

\noi
Here, $\Z_*= \Z\setminus \{0\}$ and $\{B_n\}_{n \in \Z}$ is as in \eqref{W1}, 
satisfying \eqref{W2}.
Our main goal in this paper is then to study pathwise local well-posedness of~\eqref{mild1b}
by giving a proper meaning to 
$\PPsi (\uu)$, depending on the value of the Hurst parameter $\be \in [\frac 12 , 1)$:

\smallskip

\begin{itemize}
\item[(i)] 
Young regime:
The noise in \eqref{skdv2} is 
fractional in time  with Hurst parameter $\frac 12 < \be < 1$.
In this case, we construct $\PPsi(\uu)$
as  a Young integral.

\medskip

\item[(ii)]
Rough regime:
The noise in \eqref{skdv2} is 
white in time (with Hurst parameter $\be = \frac 12$).
In this case, we construct $\PPsi(\uu)$
as  a rough integral.

\end{itemize}

\medskip

In this subsection, we restrict our attention
to the Young regime 
($\frac 12 < \be < 1$);
see the next subsection for the rough case.
%Given $\rho \ge\frac 12$, $\frac 12 < \g < 1$, and $T> 0$, 
%fix a $(\rho,\g)$-irregular function $w$ on $[0, T]$ in the sense of Definition~\ref{DEF:ir}.
Our first goal is to construct the stochastic term $\PPsi(\uu)$ in~\eqref{psi2}
as the (linear) Young integral $\I^{\YY}(\uu)$ associated 
with the following driver:
\begin{align}
\begin{split}
\YY_{t, r}(f)
&  = \P_{\ne 0} \int_r^t
\uw(t')^{-1} \big[ (\uw(t') f)
\phi dW^\be(t') \big]
\end{split}
\label{sto1x}
\end{align}

\noi
for $0 \le r < t \le T$, 
 where  $f$ is a function on $\T$.
The following theorem summarizes
the construction and 
a smoothing property of the stochastic term $\PPsi(\uu)$ in \eqref{psi2}.

\begin{theorem}\label{THM:1}

Given 
$\rho > 0$, 
$\frac 12 < \g, \be < 1$, $ \s \in \R$, and $T > 0$, 
let 
$w$ be 
 a $(\rho,\g)$-irregular function on $[0, T]$ in the sense of Definition~\ref{DEF:ir}
 and 
$\phi \in \HS(L^2(\T); H^\s(\T))$, satisfying
\eqref{phi1} and 
\begin{align}
\phi_0 = 0.
\label{phi2}
\end{align}

\noi
Suppose that 
$ \g_0 > 0 $ and $s, s_0 \in \R$ satisfy 
\begin{align}
\g_0 & <  \frac 12 + \frac 16 (2\be - 1)(2\g-1), 
\label{YT1}\\
s + \s & > -   \rho (2\be -1 ), 
\label{YT2}\\
\begin{split}
s_0 & < \min \bigg(s + \frac 23 \rho(2\be -1) + \Big(\s + \frac 13 \rho(2\be -1)\Big) \wedge 0, \\ 
& \hphantom{iiXXXX}s\wedge0 + \s +  \rho (2\be -1 )\bigg), 
\end{split}
\label{YT3}
\end{align}

\noi
where $a\wedge b = \min(a, b)$.
Then, given any $\uu \in \CC^\al([0, T]; H^s(\T))$, 
where
 $0 < \al < 1$ satisfies
$\g_0 + \al > 1$, 
the stochastic term $\PPsi (\uu)$ in \eqref{psi2} can be constructed as 
the Young integral
$\I^\YY(\uu)$ 
associated with the driver $\YY = \YY^{w, \phi}$ in \eqref{sto1x}
such that 
\[\PPsi(\uu) = \I^\YY(\uu)
 \in \CC^{\g_0}([0, T]; H^{s_0}(\T)),\] 

\noi
almost surely.
 As a  consequence, with $u(t) = \uw(t)\uu(t)$, 
 the stochastic convolution $\Psi(u)(t) = \uw(t)\PPsi(\uu)(t)$
 in~\eqref{psi1}
belongs to 
$C([0, T]; H^{s_0}(\T))$, almost surely.

\end{theorem}

There are several temporal regularities in the statement.
For readers' convenience, we summarize them here:
\begin{itemize}
\item 
$\g$ denotes the ``temporal regularity'' of the modulation function $w$
in the sense of Definition \ref{DEF:ir},\footnote{Indeed, 
$\g$ 
corresponds to the temporal regularity of the nonlinear Young integral 
$\I^\XX(\uu)$,  appearing in~\eqref{YDE0b},  
with the bilinear driver $\XX$ in~\eqref{K1x}.}

\smallskip
\item 
$\be$ denotes the Hurst parameter of the noise $\phi W^\be$ in \eqref{sto1x}
(which has temporal regularity $\be - \eps$ for any $\eps > 0$, almost surely),\footnote{In the remaining part
of the paper, we use $\eps > 0$ to denote an arbitrarily small constant.
Various constants depend on $\eps$ but,
for simplicity of notation,  we suppress such $\eps$-dependence.}

\smallskip
\item 
$\g_0$ denotes the temporal regularity of the stochastic term $\PPsi (\uu)$, 
while 
$\al$ denotes the temporal regularity of the input function $\uu$.

\end{itemize}

\smallskip

\noi
We will stick to these notations in the remaining part of the paper.

We first note that the condition \eqref{YT1}
allows us to take $\g_0 > \frac 12$ since 
$\frac 12 < \g, \be < 1$.
We also note that the right-hand side of \eqref{YT1}
is bounded by $\min(\g, \be)$.
The condition \eqref{YT2}
states that, by taking $\rho \gg 1$ sufficiently large, 
we can take 
 the  spatial regularity $s$ of the input function $\uu$
and the spatial regularity $\s$ of the noise $\phi \z$
to be arbitrarily low.
The condition~\eqref{YT3}
states that, by taking $\rho \gg 1$ sufficiently large,  
 the output spatial regularity $s_0$ of $\PPsi(\uu)$
 can be  arbitrarily high (recall $\be>  \frac 12$), 
 exhibiting a new type of {\it regularization by noise}
 for the stochastic convolution $\Psi(u)(t) = \uw(t)\PPsi(\uu)(t)$
 in~\eqref{psi1}.
This regularization by noise
on the stochastic term $\PPsi(\uu)$
comes from  the 
 interaction of a  modulated wave and the noise.

\begin{remark}\label{REM:n0}\rm 
(i)  We point out that 
 the regularization by noise 
 on $\PPsi(\uu)$ stated in Theorem~\ref{THM:1}
 fails
unless the noise $\phi \z$ has spatial mean zero (namely, 
the condition \eqref{phi2}, 
leading to $n_1 \ne 0$ in \eqref{psi2}).
See Remark \ref{REM:YG1}\,(i) below.

We also note that 
 the construction of the stochastic term $\PPsi(\uu)$ in Theorem~\ref{THM:1}
  also holds without the projector $\P_{\ne 0}$ 
(corresponding  to the condition $n \ne 0$ in \eqref{psi2})
which we inserted for the purpose of local well-posedness
of SmKdV \eqref{skdv2} (Theorem \ref{THM:2}).

\smallskip

\noi
(ii) This new regularization-by-noise phenomenon
occurs only in the Young case.
The same comment applies to Theorem \ref{THM:2} 
on pathwise local well-posedness of \eqref{skdv2} stated below.
In the rough case ($\be = \frac 12$), 
there is no such 
regularization by noise.
See Remark \ref{REM:rough1}.

\end{remark}

We prove 
 Theorem \ref{THM:1}
 by constructing
 $\PPsi(\uu)$
 as the (linear) Young integral
$\I^\YY(\uu)$:
\begin{align}
\PPsi (\uu) = \I^\YY(\uu)
\label{sto1y}
\end{align}

\noi
with 
the Young driver
$\YY$
defined in \eqref{sto1x}.
In recent works \cite{CLO2, CLO3, COZ, CLOO}, 
the first, fourth, and fifth authors
developed
an approach to proving pathwise local well-posedness
of stochastic (unmodulated) dispersive PDEs with multiplicative noises, 
in particular in the low regularity setting.
We adapt the approach in 
\cite{CLO2, CLO3, COZ, CLOO}
to the current modulated setting.
More precisely, 
we   establish almost sure mapping  properties
(in particular, the smoothing property)
of the Young driver~$\YY$ in~\eqref{sto1x}
by using
the random tensor estimate
for stochastic integrals
with respect to fractional Brownian motions 
(Lemma \ref{LEM:RT}), 
developed in~\cite{CLO2}.
Here,   the nonlinear interaction between the input function $\uu$
and the noise $\phi\z$ 
together with 
 the irregularity of the modulation function $w$
induces a strong smoothing property;
see Proposition \ref{PROP:drive2}.
Once this is achieved, 
Theorem~\ref{THM:1}
follows from the standard construction 
of a Young integral 
(see
Lemma~\ref{LEM:int1})
via the sewing lemma, 
and thus we omit details of a proof of 
Theorem~\ref{THM:1}.
Under the hypothesis of Theorem \ref{THM:1}, 
the stochastic term $\PPsi(\uu)$
in \eqref{psi2} and \eqref{sto1y}
is given 
as the unique limit of 
Riemann-Stieltjes type sums:
\begin{align}
\PPsi(\uu)(t)
& = \I^{\YY}(\uu)(t) = \lim_{|\Pi([0,t])|\to 0} 
\sum^{n-1}_{j=0} \YY_{t_j,t_{j+1}}(\uu_{t_{j+1}}), 
\label{sto1z}
\end{align} 

\noi
 where 
 the limit is taken over any partition
 $\Pi([0,t])$  
 of  $[0,t]$\textup{:} 
\[\Pi ([0,t]) = \{0 = t_{n} < \dots < t_1 <  t_0 = t\}\]
whose mesh size 
$|\Pi([0,t])| = \max_{j} |t_j-t_{j+1}|$ 
tends to $0$.

\medskip

We now state 
a  pathwise local well-posedness
result
for the (mean-zero) SmKdV \eqref{skdv2} on~$\T$ with a multiplicative 
fractional-in-time noise (with Hurst parameter $\frac 12 < \be < 1$), 
which essentially 
follows as a corollary to Theorem \ref{THM:B} 
(more precisely, Lemma \ref{LEM:kdv1}) and Theorem~\ref{THM:1}
(more precisely, Proposition \ref{PROP:drive2}).
Given $s \in \R$, we set 
\begin{align}
 H_0^s(\T)= \P_{\ne 0}H^s(\T), 
 \label{hs1}
\end{align}

\noi
where $\P_{\ne 0}$ denotes the projection onto non-zero frequencies.
Namely, 
 $H_0^s(\T)$
is a subspace of  $H^s(\T)$, 
consisting of mean-zero elements.

\begin{theorem}\label{THM:2}

Given 
$\rho \ge\frac 12$,
$\frac 12 < \g, \be < 1$, $ \s \in \R$, $T > 0$, 
let 
$w$ be 
 a $(\rho,\g)$-irregular function on $[0, T]$ in the sense of Definition~\ref{DEF:ir}
 and 
$\phi \in \HS(L^2(\T); H^\s(\T))$, satisfying 
\eqref{phi1} and~\eqref{phi2}.

\smallskip

\noi
\textup{(i) (local well-posedness).}
Suppose that
the parameters satisfy
 \eqref{reg1}, 
\eqref{YT2}, and
\begin{align}
s  < 
 \s +  \rho (2\be -1 ).
\label{YT4}
\end{align}

\noi
Then, the stochastic modulated KdV equation \eqref{skdv2}
on  $\T$
with a multiplicative fractional-in-time noise 
with the Hurst parameter $\be \in \big(\frac 12 , 1\big)$
is pathwise locally well-posed in $H^s_0(\T)$.
Moreover, the modulated interaction representation~$\uu$
of the solution almost surely belongs
to $\CC^{\g_0}([0, \tau]; H^s_0(\T))$
for any $ \g_0 \in \big(\frac 12 ,  \g\big] $ satisfying \eqref{YT1}, 
where 
 $\tau \in [0, T]$ denotes the almost surely positive  time of local existence.

\medskip

\noi
\textup{(ii) (nonlinear smoothing).}
In addition, 
suppose that $s_0 > s$ satisfies \eqref{reg3}
and \eqref{YT3}.
Given $u_0 \in H^s_0(\T)$, 
let $u
\in C([0, \tau]; H^s_0(\T))$ be the  solution 
to the stochastic  modulated KdV equation~\eqref{skdv2}
on $\T$
with $u|_{t = 0} = u_0$ constructed in Part (i).
Then, we have\footnote{While the stochastic term
$\PPsi(\uu)$ is linear in $u$, it is nonlinear in the noise
in view of the Picard iteration
for \eqref{mild1b}.
For this reason, we do not remove
$\Psi(u)(t) = \uw(t)\PPsi(\uu)(t)$ in \eqref{YT5}.} 
\begin{align}
 u - e^{-w(t) \dx^3}u_0 \in C([0, \tau]; H^{s_0}_0(\T)).
\label{YT5}
\end{align}

\end{theorem}

Theorem \ref{THM:2} 
(and Theorem \ref{THM:5} below)
establishes
the first {\it pathwise} local well-posedness 
result for stochastic modulated dispersive PDEs
with multiplicative noises.
Moreover, 
Theorem \ref{THM:2} provides
the first regularization-by-noise
result for these equations;
noting that $\be > \frac 12$, 
given $\s \in \R$, 
it follows from 
 \eqref{reg1}, 
\eqref{YT2}, and
\eqref{YT4} that we can take
the (spatial) regularity~$s$
of initial data $u_0$ (and a solution $u$) 
to be arbitrarily low
by taking $\rho \gg1$ sufficiently large.
Furthermore, given the regularity
$s \in \R$ of initial data $u_0$, 
it follows
from 
 \eqref{reg3} and~\eqref{YT3} that 
the regularity $s_0$ of the nonlinear part
$ u - e^{-w(t) \dx^3}u_0$
can be  arbitrarily high
by taking $\rho \gg1$ sufficiently large.

We prove Theorem \ref{THM:2} by 
interpreting 
\eqref{mild1b}
as the following nonlinear YDE
with a Young perturbation:
\begin{equation}
\uu = u_0 + \I^\XX(\uu) + \I^\YY(\uu), 
\label{YDE0b}
\end{equation}

\noi
where 

\smallskip
\begin{itemize}
\item
$\I^\XX(\uu)$ is the nonlinear Young integral 
with the bilinear driver $\XX$ in \eqref{K1x}, 
constructed in \cite{CGLLO1}
(see Lemmas \ref{LEM:int1} and \ref{LEM:kdv1}), representing
the second term on the right-hand side of \eqref{mild1b}, and

\smallskip
\item 
$\I^\YY(\uu)$
denotes the stochastic term $\PPsi(\uu)$
constructed in Theorem \ref{THM:1}
as the (linear) Young integral
with the driver $\YY$  in 
\eqref{sto1x}. 

\end{itemize}

\smallskip

\noi
In  Proposition~\ref{PROP:main}
(see also Remark \ref{REM:mean0}), 
 we present a  general local well-posedness result
for the nonlinear YDE 
with a Young perturbation 
of the form  \eqref{YDE0b}.
Then, Theorem~\ref{THM:2} follows
as a consequence of 
Proposition~\ref{PROP:main}
with 
 Lemma \ref{LEM:kdv1} and 
 Proposition \ref{PROP:drive2}
 on the regularity properties
 of the drivers $\XX$ and $\YY$
 in \eqref{K1x} and \eqref{sto1x}, respectively, 
 and thus we omit details of a proof of Theorem \ref{THM:2}.
We only note that 
 the  condition \eqref{YT4}
 comes from 
setting $s_0 = s$ in~\eqref{YT3}
 (where we also use \eqref{YT2} when $s < 0$).

\begin{remark}\rm

In \cite{GT10}, 
Tindel and the second author
developed the convolution Young and rough integration theory, 
adapted to an analytic semigroup, 
and 
studied pathwise well-posedness
of the stochastic heat equation with a multiplicative noise 
in both the Young case ($\frac 12 < \be < 1$) and 
the rough case ($\be = \frac 12$).
See also 
\cite{GLT, GH} %for the case of a finite-dimensional noise
and \cite[Exercises 4.16 and~4.17]{FH20}.

Our pathwise approach in both the Young and rough cases, based
on the strategy introduced in \cite{CLO2, CLO3, COZ, CLOO}
for  stochastic 
(unmodulated)
dispersive PDEs with multiplicative noises, 
is strongly inspired by the argument in \cite{GT10}, but
with two notable differences.
For dispersive PDEs, 
we can work with interaction representations
(the modulated interaction representation $\uu$ in~\eqref{int1} in our case), 
which leads to a certain simplification.
Another key difference comes from 
the random tensor estimate (Lemma \ref{LEM:RT})
for multiple stochastic integrals
with respect to (fractional) Brownian motions,
developed in \cite{OWZ,  CLO2}, 
allowing us to estimate operator norms
of random operators in a more effective manner
as compared to~\cite{GT10}, 
where operator norms were crudely bounded by 
Hilbert-Schmidt norms
which are 
 more amenable to stochastic analysis
(in particular, the Wiener-Ito isometry).
In a recent work \cite{SLO}, 
the fourth and fifth authors with Y.\,Shao
combined
the convolution Young and rough integration theory developed in \cite{GT10} 
with  the random tensor estimate approach
used 
in this paper (and in \cite{CLO2, CLO3, COZ, CLOO})
and studied
pathwise well-posedness of the stochastic heat equation with a multiplicative noise, 
improving the results in \cite{GT10}. 
We note that the presentation in \cite{SLO} is restricted
to the one-dimensional case but it can be easily adapted to any dimension.
Moreover, 
 in the rough case ($\be = \frac 12$), 
the result in \cite{GT10} covers the case of almost space-time white noise
(on $\R_+ \times \T$), 
thus establishing an optimal result
within the framework of one-parameter rough paths.
See \cite{SLO} for a further discussion.

\end{remark}

\begin{remark}\rm

(i)
Given $N \in \N$, consider the following Galerkin approximation to \eqref{skdv2}:
\begin{align}
\begin{cases}
\dt u_N+  \dx^3 u_N \cdot \dt w = \dx \P_N\big((\P_N u_N)^2 \big)+ \P_{\ne 0}
\P_N [(\P_N u_N) \P_N \phi \zeta]\\
u_N|_{t = 0} = \P_N u_0,
\end{cases}
\label{skdv2x}
\end{align}

\noi
where $\P_N$
 denotes  the Dirichlet projector onto the (spatial) frequencies 
$\{|n| \leq N\}$.
Let $\YY^N$ denote the driver associated with the noise 
$\P_{\ne 0}
\P_N [(\P_N u_N) \P_N \phi \zeta]$ in \eqref{skdv2x},
obtained by appropriately
inserting $\P_N$ in \eqref{sto1x}.
Then, it follows from a slight modification of the proof of Proposition \ref{PROP:drive2}
that $\YY^N$ converges almost surely to $\YY$ in 
$\cX^{s, s_0, \g_0}_{1, 0} ([0, T]\times \T)$, 
where 
$\cX^{s, s_0, \g_0}_{1, 0} ([0, T]\times \T)$ is as in \eqref{X1y}.
Putting this observation together with Lemma~\ref{LEM:kdv1}\,(iv), 
we conclude that the solution $u_N$
of \eqref{skdv2x}
converges almost surely to the solution $u$ of \eqref{skdv2}
in $C([0, \tau]; H^s_0(\T))$
for some almost surely positive $\tau$,  
provided that  
\eqref{reg3a}, \eqref{YT2}, and \eqref{YT4}
hold.
We omit details.

Lemma~\ref{LEM:kdv1}\,(ii) yields
 persistence of regularity 
for the (deterministic) modulated KdV~\eqref{kdv1};
see \cite[Remark 1.4]{CGLLO1}.
In the stochastic setting, such 
persistence of regularity is false in general;
even if initial data is smooth, 
the spatial regularity of a solution to~\eqref{skdv2}
constructed in Theorem \ref{THM:2}
is 
bounded above by $s_0$ satisfying \eqref{YT3}
(which depends on the value of $\s$).

%Given the parameters $\rho, \g, \be$, and $\s$, 
%let $s \in \R$ be as in Theorem \ref{THM:2}
%
%
%the modulated interaction representation 
%$\uu$ of a solution 
%to \eqref{skdv2}
%constructed 
%
% and $s_0 > 0$.
%Then, we have
%$\CC^{ \g_0}([0, \tau]; H^{s_0}_0(\T))$
%if $s_0$ additionally satisfies \eqref{YT3}.

These 
comments also apply to  the rough case ($\be = \frac 12$).

\smallskip

\noi
(ii)
We also note that 
our pathwise approach via the sewing lemma
allows for a straightforward Euler approximation scheme
for time discretization; see, for example,  \cite[Subsection 2.3]{CG1}
and \cite[Subsection 3.3]{Gub12}.
See also 
 \cite[Remark 1.17]{GLLO}
and \cite{CFO}.

\end{remark}

\begin{remark}\label{REM:deriv}\rm

The stochastic KdV equation appears in various fields of mathematical physics
(\cite{BKKS, Herman, Gar}), sometimes with a damping term, 
and 
the noise of the form $u \phi \z$ in \eqref{skdv2} %(i.e.~$\kk_1 = \kk_2 = 0$)
 corresponds to dissipation,
while the noise of the form
$(\dx u) \phi \z$ 
corresponds to velocity fluctuations; see \cite[pp.\,84-85]{BKKS}.

Given 
$\rho \ge\frac 12$,
$\frac 12 < \g, \be < 1$, $ \s \in \R$, $T > 0$, 
let 
$w$ be 
 a $(\rho,\g)$-irregular function on $[0, T]$ in the sense of Definition~\ref{DEF:ir}
 and 
$\phi \in \HS(L^2(\T); H^\s(\T))$, satisfying 
\eqref{phi1} and~\eqref{phi2}.
Then, given any $s \in \R$, 
a straightforward modification of the proof of Theorem \ref{THM:2}
yields pathwise local well-posedness 
in $H^s_0(\T)$ of the following SmKdV
on $\T$
with a multiplicative Young noise ($\frac 12 < \be < 1$):
\begin{align*}
\dt u+  \dx^3 u \cdot \dt w = \dx u^2 + \P_{\ne 0}\dx^{\kk_1}[(\dx^{\kk_2}u) \phi \zeta]
\end{align*}

\noi
for any $\kk_1, \kk_2 \in \N$, 
provided that $\rho = \rho(\be,  s, \s, \kk_1, \kk_2)\gg1$
is sufficiently large.
In particular, 
when $\kk_1 =0$ and  $\kk_2 = 1$, 
the noise 
is of the form 
$(\dx u) \phi \z$, 
representing velocity fluctuations.

%The stochastic KdV equation appears in various fields of mathematical physics
%(\cite{BKKS, Herman, Gar}), sometimes with a damping term, 
%and 
%the noise of the form $u \phi \z$ (i.e.~$\kk_1 = \kk_2 = 0$)
% corresponds to dissipation,
%while the noise of the form
%$(\dx u) \phi \z$ (i.e.~$\kk_1 = \kk_2 = 0$)
%corresponds to ; see \cite[pp.\,84-85]{BKKS}.

%We also point out that, in view of
%\eqref{K3}, 
%analogous local well-posedness in $H^s_0(\T)$
%for any $s \in \R$
%(and thus regularization by noise)
%holds
%for the following
%SmKdV
%with a multiplicative Young noise
%(which is quadratic in $u$):
%\begin{align*}
%\dt u+  \dx^3 u \cdot \dt w = \dx u^2 + \P_{\ne 0}\dx^{\kk_1}[(\dx^{\kk_2}u)
%(\dx^{\kk_3}u)
% \phi \zeta]
%\end{align*}
%
%
%\noi
%for any $\kk_1, \kk_2, \kk_3 \in \N$, 
%provided that $\rho = \rho(\be,  s, \s, \kk_1, \kk_2)\gg1$
%is sufficiently large.
%Since a required modification of the proof of Proposition \ref{PROP:drive2}, 
%using 
%\eqref{K3} instead of \eqref{YG3e}, 
%is straightforward, we omit details.

\end{remark}

\begin{remark}\rm

In \cite{CGLLO1}, 
the authors also studied 
 the 
modulated Benjamin-Ono equation (BO) on $\T$:
\begin{equation}
\label{BO}
\dt u-    \H\dx^2 u \cdot \dt w=\dx u^2,
\end{equation}

\noi
where $\H$ denotes the
Hilbert transform
with the Fourier multiplier $- \ind_{n\ne 0}\cdot i \sgn (n)$, 
and 
the 
modulated 
intermediate long wave equation (ILW) on $\T$:
\begin{equation}
\label{ILW1}
\dt u-    \Gdl\dx^2 u \cdot \dt w =\dx u^2
\end{equation}

\noi
for  $0 < \dl < \infty$,
where the operator~$\Gdl$ is given 
as the following Fourier multiplier operator:
\begin{align}
\ft{\Gdl f}(n) =
\ft{\Gdl}(n) \ft f(n)
:=
- \ind_{n\ne 0}\cdot i \bigg(\coth(\dl n)  - \frac{1}{\dl n}\bigg) \ft{f}(n), \quad n\in\Z.
\label{GG1}
\end{align}

\noi
See \cite{CLO, CGLLO1, CHLO} for the references therein
on (unmodulated) BO and ILW and their relation to (unmodulated)
KdV.

By adapting the proofs of Theorems \ref{THM:1} and \ref{THM:2}, 
we can prove pathwise local well-posedness
in $H^s_0(\T)$ 
of the following stochastic modulated BO
with a multiplicative Young noise:
\begin{equation}
\label{sbo1}
\dt u-    \H\dx^2 u \cdot \dt w=\dx u^2 +  \P_{\ne 0}[u \phi \zeta]
\end{equation}

\noi
and
the stochastic modulated ILW
with a multiplicative Young noise:
\begin{equation}
\label{SILW1}
\dt u-    \Gdl\dx^2 u \cdot \dt w =\dx u^2
+  \P_{\ne 0}[u \phi \zeta]
\end{equation}

\noi
for any $s \in \R$, provided that $\rho$ is sufficiently large
such that 
\begin{align}
\begin{split}
\textup{(a)} \ \   \rho =  1 \quad \text{and} \quad s > - \tfrac 12
\qquad \text{or}\qquad 
\textup{(b)} \ \  \rho >  1 \quad  \text{and} \quad  s \ge - \tfrac 12 \rho
\end{split}
\label{regBO1}
\end{align}

\noi
and 
\begin{align}
s + \s > -  \frac 23 \rho (2\be -1 )
\qquad \text{and}
\qquad 
s <  \s +   \frac 23 \rho (2\be -1 ).
\label{BOZ}
\end{align}

The conditions in \eqref{regBO1}
allows us to construct the nonlinear Young integral
associated with the nonlinearity $\dx u^2$
for these equations; see
\cite[Propositions 5.1 and 5.2]{CGLLO1}.
The conditions in \eqref{BOZ} come
from the construction of the stochastic term
as a (linear) Young integral.
See
Remark \ref{REM:BO} for details.
In particular, 
the second condition
in \eqref{BOZ}
comes from setting 
 $s_0 = s$ in~\eqref{BOX4}
 (where we also use the first condition in \eqref{BOZ} when $s < 0$).
In view of \eqref{BOX1} and~\eqref{ILX3}, 
it is crucial to impose \eqref{phi2}
as in the KdV case.

\end{remark}

\begin{remark}\label{REM:NLS1}\rm 

In Appendix \ref{SEC:NLS}, 
we consider the following stochastic modulated 
Schr\"odinger equation
on $\T^d$
with a multiplicative Young noise ($\frac 12 < \be < 1$):
\begin{align}
i \dt u+  \Dl  u \cdot \dt w =  (\jb{\nb}^\kk \cj u) \phi \zeta
\label{NLS2}
\end{align}

\noi
for any $\kk \ge 0$, 
and 
prove its pathwise global well-posedness;
see Theorem \ref{THM:NLS}.
This result
holds for any $\kk \gg 1$, 
provided that the modulation function $w$
is sufficiently irregular, 
thus providing another example
of regularization by noise in the case of multiplicative noise.

\end{remark}

\subsection{Main results, Part 2: rough case}
\label{SUBSEC:main2}

In this subsection, we consider the white-in-time case ($\be = \frac 12$).
Namely, we study SmKdV \eqref{skdv2}, 
where the noise $\z$ is given by the space-time white noise~$\xi$.
In this case, 
due to the lack of temporal regularity, 
we need to give a meaning to 
the stochastic term $\PPsi(\uu)$ in \eqref{psi2}
(more precisely, 
$\wt \PPsi(\uu)$ in \eqref{psi3})
as a {\it rough} integral, 
using controlled rough paths introduced by the second author \cite{Gub04}.
See Subsection \ref{SUBSEC:Y4}
for a review on the rough integration theory.

Furthermore, as observed in the Schr\"odinger case in \cite{COZ}, 
in the current rough case, 
roughness in time induces a loss in spatial regularity.
More precisely, 
there is a (spatial) regularity loss in the mapping property of 
the driver $\YY$ in \eqref{sto1x};
see Remark \ref{REM:rough}.
For this reason, we consider the following SmKdV on $\T$
(with small $\eps_0 > 0$):
\begin{align}
\dt u+  \dx^3 u \cdot \dt w = \dx u^2 + \jb{\dx}^{-\eps_0} \P_{\ne 0}[u \phi \xi], 
\label{skdv3}
\end{align}

\noi
where we slightly regularize the noise term $u \phi \xi$.
By writing \eqref{skdv3} 
in the Duhamel formulation
(at the level of the modulated interaction representation
$\uu$ defined in \eqref{int1}), 
we have 
\begin{equation}
\begin{split}
\uu(t) 
& = u_0 + \int_0 ^t \uw(t')^{-1}
\dx\big( (\uw(t') \uu(t'))^2\big)dt'
+ \wt \PPsi(\uu) (t), 
\end{split}
\label{mild1c}
\end{equation}

\noi
where 
the stochastic term  $\wt \PPsi(\uu)$ is given by 
\begin{align}
\wt  \PPsi(\uu) (t) 
 &  = 
  \jb{\dx}^{- \eps_0}
  \P_{\ne 0}
  \int_0^t
\uw(t')^{-1} \big[ (\uw(t')\uu(t') )
\phi dW^{ \frac 12}(t') \big].
\label{psi3}
\end{align}

Given $\rho \ge\frac 12$, $\frac 12 < \g < 1$, and $T> 0$, 
fix a $(\rho,\g)$-irregular function $w$ on $[0, T]$ in the sense of Definition~\ref{DEF:ir}.
Our first goal is to construct the stochastic term $\wt \PPsi(\uu)$ in \eqref{psi3}
as the rough integral $\I^{\wt \YY, \wt \YYb}(\uu)$ associated 
with the following 
operator-valued rough path $(\wt \YY, \wt \YYb) = (\wt \YY^{w, \phi}, \wt \YYb^{w, \phi})$
adapted to SmKdV \eqref{skdv3}, 
where the driver~$\wt \YY$ is given by 
\begin{align}
\begin{split}
\wt \YY_{t, r}(f)
&  = 
\jb{\dx}^{- \eps_0} \YY_{t, r}(f)\\
& = 
\jb{\dx}^{- \eps_0} \P_{\ne 0} \int_r^t
\uw(t')^{-1} \big[ (\uw(t') f)
\phi dW^\frac 12(t') \big]
\end{split}
\label{sto2x}
\end{align}

\noi
with  $\YY$ as in \eqref{sto1x}, 
and 
the second order operator-valued driver 
$\wt \YYb$
is given by 
\begin{align}
\begin{split}
\wt \YYb_{t, r}(f)
& 
 = \wt \YY_{t,r}\circ \wt \YY_{\bullet, r}(f)\\
&  =\jb{\dx}^{- \eps_0} \P_{\ne 0} \int_r^t
\uw(t_1)^{-1} \bigg[ \Big( \uw(t_1) \\
& \hphantom{XXX}
\jb{\dx}^{- \eps_0} \P_{\ne 0} \int_r^{t_1}
\uw(t_2)^{-1} \big[ (\uw(t_2) f)
\phi dW^\frac 12(t_2) \big]\Big)
\phi dW^\frac 12(t_1) \bigg]
\end{split}
\label{sto3x}
\end{align}

\noi
for $0 \le r < t \le T$.
Here,  $\bul$ denotes the variable of integration.
See \eqref{stoconv3} for 
the precise meaning of the iterated stochastic integral.

In the following, we set $\uu_t = \uu(t)$.
With a slight abuse of notation, 
we extend the definition of the driver $\wt \YY$ in \eqref{sto2x}
such that it formally acts on space-time functions by 
\begin{align}
\wt \YY_{t, r}(\uu_\bul)
&  = \jb{\dx}^{- \eps_0}\P_{\ne 0} \int_r^t
\uw(t')^{-1} \big[ (\uw(t') \uu(t'))
\phi dW^\frac 12(t') \big]
\label{sto4x}
\end{align}

\noi
for $0 \le r < t \le T$, 
where  $\bul$ denotes the variable of integration
as in \eqref{sto3x}.
At this point, the definition~\eqref{sto4x}
makes sense
only for functions $\uu$
that are smooth (in both $t$ and $x$). 
With this notation, we can formally rewrite
\eqref{mild1c} as 
\begin{equation}
\uu = u_0 + \I^\XX(\uu) + \wt \YY(\uu_\bul), 
\label{PP1}
\end{equation}

\noi
where 
$\I^\XX(\uu)$ is the nonlinear Young integral 
with the bilinear driver $\XX =\XX^w$ in \eqref{K1x}.
Recall from \cite{CGLLO1} (see Lemmas \ref{LEM:int1} and \ref{LEM:kdv1})
that 
the nonlinear Young integral $\I^\XX(\uu)$ has temporal regularity $\g \in \big(\frac 12 , 1\big)$, 
inherited from the $(\rho, \g)$-irregular modulation function $w$.
In view of~\eqref{PP1}
and this observation, 
we postulate that the increment of $\uu$ is of the form:
\begin{align}
(\updl \uu)_{t, r} = \wt \YY_{t, r}( \uu_r) + R_{t, r}, 
\label{PP2}
\end{align}

\noi
where
$(\updl \uu)_{t, r} = \uu_t - \uu_r$ is as in \eqref{dl1}
and a remainder term $R= R^{\wt \YY, \uu}$ has temporal regularity $\g \in \big( \frac 12, 1\big)$.
Namely, we postulate that $\uu$ is controlled
by the driver $\wt \YY$
with the Gubinelli derivative $\uu' = \uu$;
see Definition \ref{DEF:control}.

\begin{theorem}\label{THM:4}

Given $\rho> 0$, $ \frac 12 <  \g < 1$,   $s, \s \in \R$, 
 $\eps_0>0$, and $T > 0$ 
such that 
\begin{align}
\min\Big(s + \s , 
s + 2\s + \eps_0, 
s\wedge \s + \s + \eps_0\Big) > 0, 
\label{RR1}
\end{align}

\noi
let $w$ be 
 a $(\rho,\g)$-irregular function on $[0, T]$ in the sense of Definition~\ref{DEF:ir}
 and 
 $\phi \in \HS(L^2(\T); H^\s(\T))$, 
 satisfying
 \eqref{phi1}. 
Suppose that  $s_0 \in \R$ satisfies
\begin{align}
\begin{split}
s_0 < \min \Big( 
& s\wedge 0 + \s + \s\wedge 0 + \eps_0, \\
& s+ \s \wedge 0 + \eps_0, \, 
 s\wedge \s +  \s \wedge 0 +2\eps_0
\Big).
\end{split}
\label{RR2}
\end{align}

\noi
Then, given 
 $0 < \al < \g < 1$, 
let 
 $\uu \in \CC^\al([0, T]; H^s(\T))$
 satisfy \eqref{PP2}
 for some remainder term 
 $R = R^{\wt \YY, \uu}
\in  C^{\g}_{2, T}H^s(\T)$, 
where $C^{\g}_{2, T}H^s(\T)$ is as in \eqref{Ho2}.
Then, 
the stochastic term $\wt \PPsi (\uu)$ in~\eqref{psi3} can be constructed as 
the  rough integral $\I^{\wt \YY, \wt \YYb}(\uu)$
associated with the rough driver $(\wt \YY, \wt \YYb)$
such that 
\[\wt  \PPsi(\uu)
= \I^{\wt \YY, \wt \YYb}(\uu)
 \in \CC^{\frac 12 - \eps}([0, T]; H^{s_0}(\T)),\] 

\noi
almost surely.
 As a  consequence, with $u(t) = \uw(t)\uu(t)$, 
 the stochastic convolution $\wt  \Psi(u)(t) = \uw(t)\wt \PPsi(\uu)(t)$, 
defined by 
\begin{align*}
\wt \Psi(u)(t) = 
\jb{\dx}^{- \eps_0}
U^w(t) \int_0^t U^w(t')^{-1}
\big[u (t') \phi dW^\frac 12 (t')\big], 
\end{align*}

\noi
belongs to 
$C([0, T]; H^{s_0}(\T))$, almost surely.

\end{theorem}

In view of 
the  rough integration theory (see Subsection \ref{SUBSEC:Y4}), 
Theorem \ref{THM:4}
follows
once we establish
 the desired regularity 
properties of
the drivers $\wt \YY$ and $\wt \YYb$
(see  Proposition \ref{PROP:drive3}),
and thus we omit details of a proof of Theorem \ref{THM:4}.
 We note that, in the  rough case, 
 the modulation function $w$
 plays {\it no} role
in the construction of the stochastic term 
$\PPsi(\uu)
= \I^{\wt \YY, \wt \YYb}(\uu)$;
see the proof of 
 Proposition \ref{PROP:drive3}, 
 where we study the drivers $\wt \YY$ and $\wt \YYb$.
Under the hypothesis of Theorem \ref{THM:4}, 
the stochastic term $\wt \PPsi(\uu)$
is given 
as the unique limit of 
Riemann-Stieltjes type sums:
\begin{align*}
\begin{split}
\wt \PPsi(\uu)(t)
& = \I^{\wt \YY, \wt \YYb}(\uu)(t)\\
& = \lim_{|\Pi([0,t])|\to 0} 
\sum^{n-1}_{j=0} \Big(\wt \YY_{t_j,t_{j+1}}(\uu_{t_{j+1}})
+ \wt \YYb_{t_j,t_{j+1}}(\uu_{t_{j+1}})\Big), 
\end{split}
\end{align*} 

\noi
 where 
 the limit is 
 understood 
 in the sense of \eqref{sto1z}.

\medskip

With Theorem \ref{THM:4} in hand, 
we can view \eqref{mild1c} (see also \eqref{PP1})
%(at the level of the modulated interaction representation 
%$\uu(t) = \uw(t)^{-1}u(t)$)
as the following nonlinear YDE
with a rough perturbation:
\begin{equation}
\uu = u_0 + \I^\XX(\uu) + \I^{\wt \YY, \wt \YYb}(\uu).
\label{RDE0a}
\end{equation}

\noi
In Proposition \ref{PROP:main2}
(see also Remark \ref{REM:mean0}), 
we establish 
a local well-posedness result
for a nonlinear YDE
with a rough perturbation of the form \eqref{RDE0a}.
As a consequence, we obtain the following
 pathwise local well-posedness result
of SmKdV \eqref{skdv3}
with a multiplicative white-in-time noise.

\begin{theorem}\label{THM:5}

Let $\rho \ge\frac 12$, $ \frac 12 <  \g < 1$,   $s, \s \in \R$, 
 $\eps_0>0$, and $T > 0$
satisfy
 \eqref{reg1}, 
\eqref{RR1}, and
\begin{align}
s < \s + \s\wedge 0 + \eps_0\qquad \text{and}\qquad
\s  > -\eps_0.
\label{RR3}
\end{align}

\noi
Let $w$ be 
 a $(\rho,\g)$-irregular function on $[0, T]$ in the sense of Definition~\ref{DEF:ir}
 and 
 $\phi \in \HS(L^2(\T); H^\s(\T))$, 
 satisfying
 \eqref{phi1}. 
Then, the stochastic modulated KdV equation~\eqref{skdv3}
on  $\T$
with a multiplicative white-in-time noise 
is pathwise locally well-posed in $H^s_0(\T)$
in the sense that 
the modulated interaction representation~$\uu$ of the solution %= \uw(t)^{-1}u(t)$
 almost surely belongs
to $\CC^{\frac 12 - \eps}([0, \tau]; H^s_0(\T))$
and is a unique solution 
 the nonlinear YDE with a rough perturbation~\eqref{RDE0a}
on $[0, \tau]$, 
where 
 $\tau \in [0, T]$ denotes the almost surely positive  time of local existence.

\end{theorem}

We note that the condition \eqref{RR3} comes from 
substituting 
 $s_0 = s$ in \eqref{RR2}.
As in the Young case,  Theorem~\ref{THM:5} follows
as a consequence of 
Proposition~\ref{PROP:main2}
with 
 Lemma \ref{LEM:kdv1} and 
 Proposition \ref{PROP:drive3}
 on the regularity properties
 of the drivers $\XX$ and $(\wt \YY, \wt \YYb)$
 in~\eqref{K1x}, \eqref{sto2x}, and \eqref{sto3x}, respectively,
 and thus we omit details of a proof of Theorem~\ref{THM:5}.

\begin{remark}\rm 

Arguing as in \cite{CLO2}, 
it is possible to show that the solution $u$ constructed in Theorem \ref{THM:5}
agrees with the corresponding Ito solution 
(whose construction has not been written down but follows
easily by combining the argument in \cite{CGLLO1}
and the truncation method as in \cite{DD1, DD2}).
See also \cite[Section 5.1]{FH20}.
\end{remark}

\begin{remark}
\rm

In the current rough case, 
there is no regularization by noise;
see Remark~\ref{REM:rough1}.
In particular, the condition \eqref{phi2} plays no role 
and is not needed
in Theorems~\ref{THM:4} and \ref{THM:5}.

\end{remark}

\begin{remark}\label{REM:rough}\rm

(i)
Consider the rough case ($\be = \frac 12$).
Write $\YY$ in \eqref{sto1x} as 
$\YY = \Yg + \Yb$,
where 
the ``good'' part 
$\Yg$ denotes the contribution to $\YY$
in \eqref{stoconv1} from the case $|n_1|\ges \max (|n|, |n_2|)$, 
and 
the ``bad'' part $\Yb$ is defined by 
$\Yb = \YY - \Yg$.

Let $0 \le r < t\le T$.
Then, a slight modification of the proof of Proposition \ref{PROP:drive3}
with $\eps_0 = 0$
shows that $\Yg_{t, r}$
maps $H^s(\T)$ to $H^s(\T)$ almost surely
(by appropriately modifying 
\eqref{RR1} and~\eqref{RR3}).
However, there is a small derivative loss 
 in the mapping property of the bad part $\Yb$.
This comes from the unavoidable (small)
loss in the random tensor estimate (Lemma~\ref{LEM:RT};
see the factor $\| h_{bcn_A} \|_{\l^2_{bcn_A}}^{\ta}$
in \eqref{BM5}) - such a loss is unavoidable
in applying a Kolmogorov continuity criterion type argument
- 
and the fact that, under $|n_1|\ll \max (|n|, |n_2|)$, 
$\phi_{n_1}$ in~\eqref{stoconv1} is of no help.
This is the reason we introduced the slight regularization $\jb{\dx}^{-\eps_0}$
in~\eqref{skdv3}.

\smallskip

\noi
(ii) By following the approach in \cite{COZ}
developed for the (modulated) Schr\"odinger case, 
it is in fact possible to prove pathwise local well-posedness
of \eqref{skdv3} with $\eps_0 =0$
even in the rough case.
However, this approach involves
an infinite iteration of (part of) the Duhamel formulation,
necessitating
the study  of random operators of arbitrarily high degrees,
which is beyond the scope of this paper
(especially since there is no regularization by noise when $\be = \frac 12$).
We will address this issue in a forthcoming work.

\end{remark}

\subsection{Organization}

In Section
\ref{SEC:2}, 
we introduce basic notations and function spaces, 
including those from \cite{CGLLO1}
which play a crucial role in the remaining part of this paper.
In Section \ref{SEC:RT}, 
 we first go over the basic definitions and properties of multiple stochastic integrals
 with respect to fractional Brownian motions.
In Subsection \ref{SUBSEC:RT},  
we then state the random tensor estimate (Lemma \ref{LEM:RT})
which plays a crucial role
in studying regularity properties of various drivers.
We also recall 
Kolmogorov's continuity criterion
for operator-valued rough paths
(Lemma~\ref{LEM:kolm})
in Subsection~\ref{SUBSEC:Kol}.
In Section~\ref{SEC:Young}, 
we go over the basic construction 
for nonlinear Young  integrals
in Subsections \ref{SUBSEC:Y2}
and present a local well-posedness
result for a nonlinear YDE with a Young perturbation in 
Subsection~\ref{SUBSEC:Y3}.
We then go over the construction
of rough integrals in 
Subsection~\ref{SUBSEC:Y4}
and present a local well-posedness
result for a nonlinear YDE with a rough perturbation in 
Subsection \ref{SUBSEC:Y5}.
The materials presented in Section \ref{SEC:Young}
may be  standard by now, 
but we decided to include them in an accessible manner
for 
readers in dispersive PDEs who may not be familiar
with Young and rough integration theory.
In 
Section \ref{SEC:driver}, 
we establish regularity properties
of various drivers for SmKdV, 
thus establishing Theorems \ref{THM:1}, 
\ref{THM:2}, \ref{THM:4}, and \ref{THM:5}.
In Appendix~\ref{SEC:NLS}, 
we discuss 
the stochastic modulated Schr\"odinger equation~\eqref{NLS2}
with a multiplicative Young noise.

\section{Notations and function spaces} %preliminary tools}
\label{SEC:2}

\subsection{Basic notations}
Let $A\les B$ denote an estimate of the form $A\leq CB$ for some constant $C>0$. We write $A\sim B$ if $A\les B$ and $B\les A$, while $A\ll B$ denotes $A\leq c B$ for some small constant $c> 0$. 
We may write  $\les_{\al}$ and $\sim_{\al}$ to 
emphasize the dependence on an external parameter $\al$.
We use $C>0$ to denote various constants, which may vary line by line, 
and we may write $C_{\al}$
to emphasize 
the dependence on an external parameter $\al$.

Given $a, b \in \R$, we set
$a\vee b = \max(a, b)$ and 
$a\wedge b = \min(a, b)$.

Let $V$ be a normed vector space.
Given $K > 0$, we use $B_K$ to denote
the closed ball in $V$ of radius $K > 0$ centered at the origin.

In expressing the dependence of a function $u$
on the time variable, we often use the short-hand notation
$u_t = u(t)$,  which is standard in probability theory and stochastic analysis.

Given $n_{j_1}, \dots, n_{j_k} \in \Z$, 
we set 
\[n_{j_1 \cdots j_k} = 
n_{j_1}+ \cdots+ n_{j_k}.\]

\noi
For example, 
$n_{123} = n_1 + n_2 + n_3$.

Given dyadic numbers $N, N_1, N_2\ge 1$, 
we use $N_{\max}$, $N_{\med}$, and $N_{\min}$
to denote their decreasing rearrangement:
\begin{align}
N_{\max} \ge N_{\med} \ge N_{\min}.
\label{ord1}
\end{align}

\subsection{Function spaces}
\label{SUBSEC:2.2}

Let $\M = \T^d$ or $\R^d$.
In considering an integral (in time) operator $\I$, 
a priori defined on functions on $\M$, 
we use the notation~$\I(u_\bul)$
to denote its (formal) action
on a space-time function $u$, 
where  $\bul$ denotes the variable of integration.

We  use $\ft f$ 
to denote
the  Fourier transform
of a function $f$ on $\M$, 
defined by 
\begin{align*}
\ft  f(\xi)
%=\int_{\M}f(x)e^{-i \xi \cdot x} d x
= \frac 1 {(2\pi)^d}\int_{\M}f(x)e^{-i \xi\cdot  x} d x, \qquad \xi\in   \M^*, 
\end{align*}

\noi
where $\M^*$
denotes the Pontryagin dual of $\M$:
$\ft \M = \Z^d$ when $\M = \T^d$
and 
$\ft \M = \R^d$ when $\M = \R^d$.
Given $s \in \R$, 
we define the  Sobolev space $H^{s}(\M)$ 
via the norm: 
\begin{align*}
\|  f  \|_{H^{s}(\M)}
& =
\bigg( \int_{\M^*} \jb{\xi}^{2s} |\ft f( \xi)|^2 d\xi\bigg)^\frac 12 ,
\end{align*}

\noi
where 
$d\xi$ denotes the counting measure when $\M = \T^d$.
We also define 
$ H_0^s(\T)$
as in \eqref{hs1}
via the norm:
\begin{align*}
\|  f  \|_{H^{s}_0(\T)}
& =
\bigg( \sum_{n\in \Z_*} \jb{n}^{2s} |\ft f( n)|^2 \bigg)^\frac 12 ,
\end{align*}

\noi
where 
$\Z_*= \Z\setminus \{0\}$.

In studying space-time functions, 
we often use short-hand notations such as
$C_T H^s_x  = C\big([0, T]; H^s(\M))$, etc.
 when there is no ambiguity.

\medskip

Next, we recall the basic definitions
of  function spaces relevant to rough analysis from \cite{CGLLO1}.
Given Banach spaces $V$ and $W$, we use $\L(V; W)$
to denote the Banach space of bounded linear operators from $V$ to $W$.
When $V = W$, we simply set $\L(V) = \L(V;V)$.

%Given $k \in \N$, 
%let $V, V_1, \dots, V_k$  be separable Hilbert spaces.
%We use 
%\begin{align*}
%\cL_k\Big(\bigotimes_{j = 1}^k V_j; V\Big)
%\end{align*}
%
%\noi
% to denote 
%the Banach space of bounded $k$-linear operators 
%on $\bigotimes_{j = 1}^k V_j$ 
%(equipped with the Hilbert tensor norm) %product) 
%with values in $V$.
%When $V_j =  V$ for $j =1, \dots,k$,
%we simply set  
% $\cL_k(V)=\cL_k(V^{\otimes k}; V)$. 
%Moreover, when $k = 1$, 
%we simply set $\cL = \cL_1$.

Let $V$ be a Banach space and $T>0$.
For $n\in\N$, we denote 
\begin{align*}
\Delta_{n, T} = 
\big\{ (t_1, \ldots, t_n) \in [0,T]^n: \ t_i > t_j 
\text{ for } i < j\big\}.
\end{align*}
We denote by $C_{n,T}V$ 
the space of continuous functions 
from $\Delta_{n,T}$ to $V$. When $n=1$,
we may write $C_T V$ for simplicity, 
and equip this space with the supremum norm: 
\begin{align*}
\|f\|_{C_T V} = \|f\|_{L^\infty_T V} = \sup_{0\leq t \leq T} \|f(t)\|_V.
\end{align*}

\noi
We define the coboundary operator 
$\updl: C_{n,T} V  \to C_{n+1,T} V$ 
as follows; 
given  $f\in C_{n,T} V$ 
and $(t_1,\ldots, t_{n+1}) \in \Dl_{{n+1},T}$, 
we set
\begin{align*}
(\updl f)_{t_1,\ldots , t_{n+1}} = 
\sum_{k=1}^{n+1} (-1)^{k} f_{t_1, \ldots, t_{k-1}, t_{k+1}, \ldots, t_{n+1}}.
\end{align*}

\noi
For example, for $f\in C_T V$
and 
$g\in C_{2,T} V$, 
we have 
\begin{align}
\begin{split}
(\updl f )_{t,r} &= f_t - f_r, \\
(\updl g)_{t_1,t_2,t_3} &= 
g_{t_1,t_3} - g_{t_1,t_2} - g_{t_2,t_3}
\end{split}
\label{dl1}
\end{align}

\noi
for $(t,r)\in\Dl_{2,T}$
and $(t_1,t_2,t_3) \in \Dl_{3,T}$.
As noted in \cite{GT10}, 
the sequence 
\begin{align*}
0 \too V \too C_{1, T}V \stackrel{\updl}{\too} C_{2, T}V
\stackrel{\updl}{\too} C_{3, T}V
\stackrel{\updl}{\too} \cdots
\end{align*}

\noi
is exact. 
In particular, we have 
 $\updl\circ\updl =0$ and 
if $f \in C_{n,T} V$ with $\updl f =0$, 
then there exists a $g\in C_{n-1,T}V$ 
such that $f= \updl g$; 
see, for example,  \cite[Lemma 2.1]{GT10}.

Given $0 < \g \le 1$, we denote by $C^\g_T V = C^\g([0, T]; V)$ the space of $\g$-H\"older continuous functions taking values in $V$, endowed  with the seminorm:
\begin{align*}
\|f\|_{C^\g_T V} = \sup_{(t,r)\in \Dl_{2,T}} 
\frac{\|(\updl f)_{t,r}\|_V}{|t-r|^\g}.
\end{align*}

\noi
We then define 
 $\CC^\g_T V = \CC^\g([0, T]; V)$ via the norm:
\begin{align*}
\| f \|_{\CC^\g_TV} = \| f \|_{L^\infty_TV} + \|f\|_{C^\g_T V}.
%\label{Ho2a}
\end{align*}

\noi 
We also introduce the spaces 
$C^\g_{n,T}V$, $n=2,3$, 
equipped with the following H\"older-type norms;
 for $g\in C_{2,T}V$ and $h\in C_{3,T}V$, we set
\begin{align}
\begin{split}
\| g\|_{C^\g_{2,T} V} & 
= \sup_{(t,r) \in \Dl_{2,T}} \frac{\|g_{t,r} \|_{V} }{|t-r|^{\g}}, \\
\| h\|_{C^\g_{3,T} V} & 
= \inf_{0<\al<\g} 
\sup_{(t_1,t_2,t_3) \in \Dl_{3,T}} 
\frac{ \|h_{t_1,t_2,t_3}\|_V}
{|t_1-t_2|^{\al} |t_2-t_3|^{\g-\al}}.
\end{split}
\label{Ho2}
\end{align}

\medskip

Let $V$ and $W$ be Banach spaces.
Given $k \in \N$, 
we  use $\Lip_k(V; W)$ 
to denote the Banach space of 
locally Lipschitz maps $f:V\to W$
with polynomial growth of order $k$
such that 
\begin{align}
\|f\|_{\Lip_k(V; W)} = 
\sup_{x,y\in V} 
\frac{\|f(x)-f(y)\|_W}{\|x-y\|_V \big(1+\|x\|_V+\|y\|_V\big)^{k-1}}
<\infty .
\label{Lip1}
\end{align}

\noi
When $V = W$, 
we simply set $\Lip_k(V)= \Lip_k(V;V)$.
Given an integer $k \ge 2$, 
we say that $f\in\Lip^2_k(V; W)$ 
if 
\smallskip

\begin{itemize}
\item[(i)] $f\in\Lip_k(V; W)$,

\smallskip
\item[(ii)]
 $f$ is Fr\'echet differentiable 
with 
$Df \in \Lip_{k-1}(V;\cL_1(V; W))$.

\end{itemize}

\smallskip

\noi
From \eqref{Ho2} and \eqref{Lip1}, we have 
\begin{align}
\|f \|_{C^\g_{2, T}\Lip_k(V;W)}
= \sup_{(t,r) \in \Dl_{2,T}}
\frac 1{|t-r|^{\g}}
\sup_{x,y\in V} 
\frac{\| f_{t, r}(x)- f_{t, r}(y)\|_W}{\|x-y\|_V \big(1+\|x\|_V+\|y\|_V\big)^{k-1}}.
\label{Ho3}
\end{align}

\noi
In addition, suppose that $V_0 \hookrightarrow V$
is a Banach subspace of $V$.
Given an integer $k \ge 2$, %$k \in \N$, 
we  use $\Lip_k(V, V_0; V_0)\subset  \Lip_k(V)$ 
to denote the Banach space of 
locally Lipschitz maps $f:V\to V_0$
such that 
\begin{align}
\|f\|_{\Lip_k(V, V_0; V_0)}
= 
\sup_{x,y\in V} 
\frac{\|f(x)-f(y)\|_{V_0}}
{G_{V, V_0}(x, y)}
<\infty , 
\label{Lip2}
\end{align}

\noi
where $G_{V, V_0}(x, y)$ is given by 
\begin{align*}
G_{V, V_0}(x, y)
& = \|x-y\|_{V_0} \big(1+\|x\|_{V}+\|y\|_{V}\big)^{k-1}\\
& \quad + \|x-y\|_{V} \big(1+\|x\|_{V}+\|y\|_{V}\big)^{k-2}
 \big(1+\|x\|_{V_0}+\|y\|_{V_0}\big).
\end{align*}

\noi
Then, 
from \eqref{Ho2} and \eqref{Lip2}, we have 
\begin{align*}
\|f \|_{C^\g_{2, T}\Lip_k(V, V_0; V_0)}
= \sup_{(t,r) \in \Dl_{2,T}}
\frac 1{|t-r|^{\g}}
\sup_{x,y\in V} 
\frac{\| f_{t, r}(x)- f_{t, r}(y)\|_{V_0}}{
G_{V, V_0}(x, y)}.
%\label{Ho4}
\end{align*}

\medskip

Given $s \in \R$, $0 < \g < 1$,  $k \in \N$, and $T > 0$, 
we define the space 
$\cX^{s, \g}_k([0, T]\times\M)$ of drivers 
as follows; for $k \ge 2$, we set
\begin{align}
\begin{split}
& \cX^{s, \g}_k ([0, T]\times \M) \\
& \quad = 
\big\{X \in C^\g_{2, T} \Lip_k^2(H^s(\M)):
X(0) = DX[0]= 0 \text{ and }\updl X = 0\big\}, 
\end{split}
\label{X1}
\end{align}

\noi
endowed with the norm
\begin{align}
\|X\|_{\cX^{s, \g}_k([0, T]\times \M)} 
=  \|X\|_{C^\g_{2, T}\Lip_k(H^s(\M))} + \|DX\|_{C^\g_{2, T}\Lip_{k-1}(H^s(\M);\cL_1(H^s(\M)))}.
\label{X2}
\end{align}

\noi
When $k = 1$, we set
\begin{align}
\begin{split}
 \cX^{s, \g}_1 ([0, T]\times \M) 
& =  \big\{ X \in 
C^\g_{2, T} \L(H^s(\M)):
\updl X = 0\big\}.
\end{split}
\label{X1z}
\end{align}

\noi
In \eqref{X1}, 
 $X(0) = 0$ (and $DX[0]= 0$) means $X_{t, r}(0) = 0$ 
(and $DX_{t, r}[0]= 0$, respectively)
for any $(t, r) \in \Dl_{2, T}$, 
where $DX_{t, r}[0]$
denotes the Fr\'echet derivative of $X_{t, r}$ at $u = 0 \in H^s(\M)$.
We may simply write $\cX^{s, \g}_k ([0, T])$ or $\cX^{s, \g}_k(T)$, 
where there is no confusion about 
an underlying spatial domain.

Given $s, s_0 \in \R$, $0 < \g < 1$,  $ k \in \N$, and $T > 0$
with $s_0 > s$, we define the space
$\cX^{s, s_0, \g}_{k}([0, T]\times\M)$ 
of drivers (for establishing nonlinear smoothing) by 
setting 
\begin{align}
& \cX^{s, s_0, \g}_k ([0, T]\times \M) 
= \cX^{s, \g}_k ([0, T]\times \M) \cap C^\g_{2, T} \Lip_k(H^s(\M); H^{s_0}(\M)).
\label{X1x}
\end{align}

\noi
In proving persistence of regularity, we need the following
class of drivers.
Given $s, s_0 \in \R$, $0 < \g < 1$,  
an integer $k \ge 2$,
 and $T > 0$
with $s_0 > s$, we define 
$\cY^{s, s_0, \g}_{k}([0, T]\times\M)$ by 
\begin{align}
\begin{split}
& \cY^{s, s_0, \g}_k ([0, T]\times \M) \\
& \quad = 
\big\{ X \in 
  C^\g_{2, T} \Lip_k(H^s(\M), H^{s_0}(\M);H^{s_0}(\M)): 
  \updl X = 0\big\}.
\end{split}
\label{X2c}
\end{align}

\noi
Note that, when $k = 1$, we have 
\begin{align}
\begin{split}
 \cX^{s, s_0, \g}_1 ([0, T]\times \M) 
&  = \cY^{s, s_0, \g}_1 ([0, T]\times \M) \\
& =  \big\{ X \in 
C^\g_{2, T} \L(H^s(\M); H^{s_0}(\M)):
\updl X = 0\big\}.
\end{split}
\label{X1z1}
\end{align}

\begin{remark}\label{REM:m0}\rm
In studying SmKdV \eqref{skdv2}, 
we restrict our attention to mean-zero functions
belonging to $H^s_0(\T)$ defined in \eqref{hs1}.
For this purpose, we define 
$\cX^{s, \g}_{k, 0} ([0, T]\times \T) $, 
$\cX^{s, s_0, \g}_{k, 0} ([0, T]\times \T) $, 
and 
$\cY^{s, s_0, \g}_{k, 0}([0, T]\times\T)$ 
by replacing $H^s(\M)$ and 
$H^{s_0}(\M)$ in \eqref{X1}, \eqref{X1x}, and \eqref{X2c}
by 
$H^s_0(\T)$ and 
$H^{s_0}_0(\T)$, respectively.
In particular, from \eqref{X1z} and \eqref{X1z1}, we have
\begin{align}
\begin{split}
 \cX^{s, \g}_{1, 0} ([0, T]\times \T) 
& =  \big\{ X \in 
C^\g_{2, T} \L(H^s_0(\T)):
\updl X = 0\big\}, \\
 \cX^{s, s_0, \g}_{1, 0} ([0, T]\times \T) 
 & =  \cY^{s, s_0, \g}_{1, 0} ([0, T]\times \T) \\
& =  \big\{ X \in 
C^\g_{2, T} \L(H_0^s(\T); H_0^{s_0}(\T)):
\updl X = 0\big\}.
\end{split}
\label{X1y}
\end{align}

\end{remark}

\section{Random tensor estimate}
\label{SEC:RT}

In this section, 
 we first go over the basic definitions and properties of multiple stochastic integrals
 with respect to fractional Brownian motions.
In Subsection \ref{SUBSEC:RT},  
we state the random tensor estimate (Lemma \ref{LEM:RT})
which plays a crucial role
in establishing almost sure mapping properties
of (random) drivers
associated with multiplicative noises.
We then recall 
Kolmogorov's continuity criterion
for operator-valued rough paths
(Lemma \ref{LEM:kolm})
in Subsection~\ref{SUBSEC:Kol}.

\subsection{Fractional Brownian motion and multiple stochastic integrals}
\label{SUBSEC:FBM}

In this subsection, we 
briefly go over 
 the basic definitions and properties  of fractional Brownian motions and 
Wiener integrals with  respect to fractional Brownian motions.
See \cite[Chapter 5]{Nualart06} for a further discussion.

\begin{definition}\rm
\label{DEF:fBM}
Let $0 < \be < 1$.
 A (real-valued) fractional Brownian motion $\{B(t)\}_{t\in \R_+}$ with Hurst parameter $\be$ is a centered Gaussian process with covariance given by
\begin{align*}
 \E[B(t_1) B(t_2)] =   \frac12 \Big( t_1^{2\be} + t_2^{2\be} - |t_1-t_2|^{2\be}\Big) .
\end{align*}

\noi
When $\be=\frac12$, this process reduces  to the standard Brownian motion. 
A complex-valued fractional Brownian motion $\{ B(t)\}_{t\in \R_+}$ with Hurst parameter $\be$ is a 
complex-valued centered Gaussian process  such that 
%its real and imaginary parts 
$\{ \sqrt 2 \Re B(t)\}_{t\in \R_+}$ and  $\{\sqrt 2\Im B(t)\}_{t\in \R_+}$
 are independent real-valued fractional Brownian motions with Hurst parameter $\be$
 such that 
 \[ \E\big[ |B(t_1) - B(t_2)|^2\big] = |t_1 - t_2|^{2\be}.\]

\noi
See also 
\cite[Section 5]{OST}
for a discussion on fractional Brownian motions.

\end{definition}

In the following, we restrict our attention to the real-valued setting
but our discussion can be easily adapted to the complex-valued setting.\footnote{Namely, 
by dropping the conditions \eqref{fBM1}, \eqref{fBM2}, 
\eqref{BM3}, and \eqref{BM4}.
See also \cite{OTh1}
for a related discussion in the complex-valued setting.}
Let us first  introduce the following partition of  $\Z^d$:
\begin{align*}
\Z^d = (\Z^d)_+  \cup (\Z^d)_- \cup \{0\}^d ,
\end{align*}

\noi
where 
\begin{equation*}
(\Z^d)_+=\bigcup_{k=0}^{d-1} \Z^k\times \Z_{+}\times \{0\}^{d-k-1}
\qquad \text{and}\qquad 
(\Z^d)_- = -(\Z^d)_+.
\end{equation*}

\noi
We also  set 
\begin{align}
(\Z^d)_{+,0} := (\Z^d)_{+} \cup \{0\}^d.
\label{fBM0}
\end{align}

\noi
Given $0 < \be < 1$, 
let $\{B_n\}_{n\in \Z^d}$ be a family of 
mutually independent complex-valued fractional Brownian motions with Hurst parameter $\be$, conditioned 
that 
\begin{align}
B_{-n} = \cj{B_n}, \qquad n \in \Z^d, 
\label{fBM1}
\end{align}

\noi
which in particular implies that $B_0$ is real-valued.
Then,  $\{B_0\}\cup \{\sqrt 2\Re B_n, \sqrt 2 \Im B_n \}_{n\in (\Z^d)_{+}}$ 
forms a family of  mutually independent real-valued fractional Brownian motions
with the same Hurst parameter $\be$.

In order to define stochastic integrals with respect to these fractional Brownian motions, 
we need the following {\it real} Hilbert space.
For $\frac 12 < \be< 1$,  let 
$\Hs^\be(\R_+)$ be the completion of 
linear combinations of
(real-valued) step functions on $\R_+$ under 
the following norm:\footnote{On the class of functions $f$ on $\R$ with 
$\supp (f) \subset \R_+$, \eqref{BM0} defines a norm, not a semi-norm.}
\begin{align}
\label{BM0}
\begin{split}
\|f\|_{\Hs^\be(\R_+)}^2
&   = \be(2\be-1)
\int_0 ^\infty \int_0^\infty f(t) f(t') |t-t'|^{2\be-2} dtdt' \\
& = C_\be \int_0^\infty f(t) \mathfrak{I}_{2\be - 1}(f) (t) dt \\
&  = C_\be  \int_\R  |\tau|^{1-2\be} |\ft{f}(\tau)|^2 d\tau 
= C_\be \|f\|_{\dot H^{\frac 12 - \be}(\R)}^2
\end{split}
\end{align}

\noi
for a function $f$ supported on $\R_+$, 
where
$\mathfrak{I}_{2\be-1} = |\dt|^{1-2\be}$ denotes the Riesz potential of order $2\be - 1>0$.
Similarly, given $k \in \N$, 
we 
define $\Hs^\be(\R_+^k)$ 
(with the understanding 
that $\R_+^k := (\R_+)^k$)
to be
the completion of 
linear combinations of
products of step functions in  $t_j \in \R_+$, $j = 1, \dots, k$,  under 
\begin{align}
\begin{split}
\|f\|_{\Hs^\be(\R_+^k)}^2
&   = \be^k(2\be-1)^k
\int_{\R_+^{2k}}
f(t_1, \dots, t_k) 
f(t_1', \dots, t_k')
\prod_{j = 1}^k |t_j-t_j'|^{2\be-2} dt_jdt_j' \\
&
= C_\be^k
\bigg\| \prod_{j = 1}^k |\dd_{t_j}|^{\frac 12 - \be}f\bigg\|_{L^2(\R^k_+)}^2
\end{split}
\label{BM0a}
\end{align}

\noi
for a function $f(t_1, \dots, t_k)$ supported on $\R_+^k$. 
When  $\be = \frac12$, we set  
\begin{align}
\Hs^\frac12(\R^k_+) = L^2(\R_+^k).
\label{BM1}
\end{align}

%
%
%Let 
%$\{B_n\}_{n\in \Z^d}$ 
%be the family of independent  complex-valued fractional Brownian motion
%conditioned that 
%
%
%\noi

We say that 
a sequence $f = \{ f_n \}_{n \in \Z^d}$ of complex-valued functions
$f_n$ on $\R_+$
 belongs to $ \l^2(\Z^d;\Hs^\be(\R_+))$, 
if we have
 \begin{align}
 f_{-n} = \cj{f_n}, \quad n \in \Z^d, 
\label{fBM2} 
 \end{align}

\noi
and 
$\Re f_n, \Im f_n \in \Hs^\be(\R_+)$
(note from \eqref{fBM2}  that $\Im f_0 = 0$)
such that 
\begin{align*}
\|f\|_{\l^2(\Z^d;\Hs^\be(\R_+))}
:\! & = \bigg(\sum_{n \in \Z^d} \| f_n\|_{\Hs^\be(\R_+)}^2\bigg)^\frac 12\\
& = \bigg(\sum_{n \in \Z^d} \|\Re  f_n\|_{\Hs^\be(\R_+)}^2
+ \sum_{n \in \Z^d} \|\Im  f_n\|_{\Hs^\be(\R_+)}^2\bigg)^\frac 12 < \infty.
\end{align*}

%\noi
%Given $f\in \l^2(\Z^d;\Hs^\be(\R_+))$
% satisfying 
% \begin{align}
% f_{-n} = \cj{f_n}, \quad n \in \Z^d, 
%\label{fBM2} 
% \end{align}
 
\noi
We then define the Wiener integral  of $f$
 with respect to $\{B_n\}_{n\in\Z^d}$
 by setting
\begin{align}
\begin{split}
I_1[f]  
& =  \sum_{n\in \Z^d}J_n^{(r)}(f_n)
+ \sum_{n\in \Z^d}J_n^{(i)}(f_n)\\
& = \sum_{n\in \Z^d} \int _0^\infty f_n(t) d \Re B_n(t)
+ i 
 \sum_{n\in \Z^d}
\int _0^\infty f_n(t) d \Im B_n(t).
\end{split}
\label{I1}
\end{align}

\noi
Note that $I_1[f]$ is real-valued
in view of the conditions \eqref{fBM1} and \eqref{fBM2}.
In \eqref{I1},  each summand $J_n^{(r)}(f_n)$ 
or $J_n^{(i)}(f_n)$ is  understood as a Wiener integral;
namely, $\{J_0^{(r)}(f_0)\}\cup\{\sqrt 2 J_n^{(r)}(f_n), \sqrt 2 J_n^{(i)}(f_n)\}_{n \in (\Z^d)_{+}}$ is a family of independent mean-zero Gaussian random variables
with variance $\|f_n\|_{\Hs^\be}^2$.
In particular,  
the map $I_1$ is an isometry from 
$\l^2(\Z^d;\Hs^\be(\R_+))$ into $L^2(\Omega, \s(I_1) , \PP)$, where $\s(I_1)$ is the $\s$-algebra generated by the process\footnote{Here, we view $I_1$ as a process
indexed by  $f \in \l^2(\Z^d;\Hs^\be(\R_+))$.}
$I_1= \{I_1[f]: f \in \l^2(\Z^d;\Hs^\be(\R_+))\}$. 
The process $I_1$ is known as an isonormal Gaussian process associated with the Hilbert space 
$\l^2(\Z^d;\Hs^\be(\R_+))$,  satisfying
\begin{align*}
%\label{BM2}
\E\big[I_1[f] I_1[g]\big] = \jb{f,g}_{\l^2_n \Hs_t^\be }
\end{align*}

\noi
for any $f,g\in 
\l^2(\Z^d;\Hs^\be(\R_+))$;
see \cite[Definition 1.1.1]{Nualart06}.

Given $k \in \N \cup\{0\}$, 
we define the $k$th homogeneous Wiener chaos $\HH_k$ 
to be  the closed linear subspace of $L^2(\Omega)$ generated by
\begin{align*}
\big\{ H_k(I_1[f]) : \, f\in \l^2(\Z^d;\Hs^\be(\R_+)),  \, \|f\|_{\l^2_n\Hs^\be_t} = 1   \big\},
\end{align*}

\noi
where $H_k$ denotes the  Hermite polynomial of degree $k$, 
defined via the following generating function:
\begin{equation*}
 e^{tx - \frac{1}{2} t^2} = \sum_{k = 0}^\infty \frac{t^k}{k!} H_k(x).
 \end{equation*}
	
\noi
For readers' convenience, we write out the first few Hermite polynomials:
\begin{align*}
\begin{split}
& H_0(x) = 1, 
\quad 
H_1(x) = x, 
\quad
H_2(x) = x^2 - 1,   
\quad
 H_3(x) = x^3 - 3 x.
\end{split}
\end{align*}

\noi
The spaces $\HH_k$ and $\HH_j$ are orthogonal when $k\neq j$, 
and the real Hilbert space $L^2(\Omega, \s(I_1), \PP)$ admits the following 
Wiener-Ito decomposition: % as an orthogonal sum of the subspaces $\HH_k$:
\begin{align*}
L^2(\Omega, \s(I_1), \PP) = \bigoplus_{k=0}^\infty \HH_k.
\end{align*}

\noi
Note that given any $f\in \l^2(\Z^d;\Hs^\be(\R_+))$, the stochastic integral $I_1[f]$ is an element of the first 
homogeneous Wiener chaos~$\HH_1$.
We now state the Wiener chaos estimate, which follows from the hypercontractivity of the Ornstein-Uhlenbeck semigroup
 due to Nelson \cite{Nelson2};
 see 
 \cite[Theorem~I.22]{Simon}.

\begin{lemma}[Wiener chaos estimate]
\label{LEM:WCE}
Let $k\in\N$.
Then, given any finite $p \ge 1$ and $F \in 
\bigoplus_{j=0}^k \HH_j$, 
%\H_{\le k}$, 
 we have
\begin{align*}
\| F \|_{L^p(\Omega)} \le p^{\frac k 2 } \| F \|_{L^2(\Omega)}.
\end{align*}

\end{lemma}

Lastly, we introduce multiple stochastic integrals with respect to the fractional Brownian motions 
$\{B_n\}_{n\in\Z^d}$ with Hurst parameter $\frac 12  \le \be < 1$, satisfying \eqref{fBM1}.
Fix  an integer $k \ge 2$.
Given  $f\in \l^2( (\Z^d)^{k} ; \Hs^\be(\R_+^{k})  )$, where 
$\Hs^\be(\R_+^{k})$ 
is as in  \eqref{BM0a}
and \eqref{BM1},  
 we define its symmetrization by
\begin{align*}
\Sym(f) (z_1,  \ldots, z_k)  
= \frac{1}{k!} \sum_{\s\in \S_k} f(z_{\s(1)},  \ldots, z_{\s(k)}),
\end{align*}

\noi
where 
$z_j = (t_j, n_j)$
and $\S_k$ denotes the symmetric group on $\{1, \dots, k\}$.
We denote by  
$\big(\l^2( (\Z^d)^{k} ; \Hs^\be(\R_+^{k})  )\big)^{ \Sym}$  the subspace of symmetric functions in 
$\l^2( (\Z^d)^{k} ; \Hs^\be(\R_+^{k})  )$.

%We then denote by 
%$\l^2_{n}( (\Z^d)^{k} ; \Hs^\be(\R_+^{k})  )^{ \Sym}$ the subspace of symmetric functions in 
%$\l^2_{n}( (\Z^d)^{k} ; \Hs^\be(\R_+^{k})  )$.

We now introduce 
 the notion of a multiple Wiener integral $I_k$.

\begin{definition} \rm
Let $k\in\N$.
The $k$th multiple Wiener integral $I_k$ is an isometry
(up to a constant factor; see \eqref{ISO1}) from 
$\big(\l^2( (\Z^d)^{k} ; \Hs^\be(\R_+^{k}))  \big)^\Sym$
into the Wiener chaos $\HH_k$, uniquely determined by
\begin{align*}
I_k\big[ \Sym(h_1^{\otimes k_1} \otimes \cdots \otimes h_n^{\otimes k_n})
\big  ] = \prod_{j=1}^n H_{k_j} (I_1[h_j])
\end{align*}

\noi
for any orthonormal elements $h_1, \ldots, h_n \in 
\l^2(\Z^d;\Hs^\be(\R_+))$ and  $k_1, \dots, k_n \in \N\cup\{0\}$
such that  $k = k_1 + \ldots + k_n$, 
where $H_{k}$ denotes the  Hermite polynomial of degree $k$ and $I_1$ is as in~\eqref{I1}.

\smallskip

\begin{itemize}
\item
We have 
$I_k[f] = I_k [ \Sym(f)]$
for any $f\in
\l^2( (\Z^d)^{k} ; \Hs^\be(\R_+^{k})  )$.

\smallskip
\item
Given $f\in
\l^2( (\Z^d)^{k} ; \Hs^\be(\R_+^{k})  )$
 and $g\in \l^2( (\Z^d)^\l ; \Hs^\be(\R_+^{\l})  )$
for some $k, \l \in \N$, we have
\begin{align}
\E \big[ I_k[f] I_\l[g] \big] = k! \cdot \ind_{k=\l} \cdot \jb{\Sym(f), \Sym(g)}_{\l^2_{n_1, \dots, n_k} \Hs^\be_{t_1, \dots, t_k}}, 
\label{ISO1}
\end{align}

\noi
where
$\l^2_{n_1, \dots, n_k} \Hs^\be_{t_1, \dots, t_k}$
is a short-hand notation for 
$\l^2_{n}( (\Z^d)^{k} ; \Hs^\be(\R_+^{k})  )$.

\smallskip
\item
When $\be=\frac{1}{2}$, the multiple Wiener integrals agree with the iterated Wiener-Ito integrals with respect to 
a family  $\{B_n \}_{n \in \Z^d}$ of mutually independent standard Brownian motions, 
satisfying \eqref{fBM1}.
Furthermore, suppose that $f$ is symmetric.
Then, we have 
\begin{align*}
& I_k[f] 
 = k! \sum_{n_1, \cdots, n_k \in \Z^d} 
%\int_{0}^\infty \int_{0}^{t_1} \cdots \int_{0}^{t_{k-1}} 
\intt_{\Dl_k}
f(t_1, n_1, \dots, t_k, n_k)
 dB_{n_k}(t_k) \cdots dB_{n_1}(t_1),
\end{align*}

\noi
where 
$\Dl_k
= \big\{ (t_1, \ldots, t_k) \in \R_+^k: \ t_i > t_j 
\text{ for } i < j\big\}$.
Here,  the iterated integral on the right-hand side is understood as an iterated 
Ito integral; see \cite[p.\,23]{Nualart06}.

\end{itemize}

\end{definition}

\subsection{Random tensor estimate}
\label{SUBSEC:RT}

In this subsection, we provide the basic definition of tensors and 
state the random tensor estimate (Lemma \ref{LEM:RT}).
See~\cite[Sections 2 and 4]{DNY3}, \cite[Section~4]{Bring}, 
\cite[Appendix C]{OWZ}, 
and \cite[Section 2]{OW3}
for further discussions.

\begin{definition} \label{DEF:tensor} \rm
Let $A$ be a finite index set. We denote by $n_A$ the tuple $ \{n_j \}_{j \in A}$. 
 A tensor $h = h_{n_A}$ is a function: $(\Z^d)^{A} \to \mathbb{C} $ with the input variables $n_A$. Note that the tensor $h$ may also depend on $\o \in \Om$. 
 The support of a tensor $h$ is the set of $n_A$ such that $h_{n_A} \neq 0$. 

Given a finite index set  $A$, 
let $(B, C)$ be a partition of $A$. We define the norms 
 $\| \cdot \|_{n_A}$ and 
$\| \cdot \|_{n_{B} \to n_{C}}$ by 
\[ \| h \|_{n_A}  = \|h\|_{\l^2_{n_A}} = \bigg(\sum_{n_A} |h_{n_A}|^2\bigg)^\frac{1}{2}\]
and
\begin{align*}
  \| h \|^2_{n_{B} \to n_{C}} = \sup \bigg\{ 
\sum_{n_{C}} \Big| \sum_{n_{B}} h_{n_A} f_{n_{B}} \Big|^2 :  \| f \|_{\l^2_{n_{B}}} =1  \bigg\},  
%\label{Z0a}
\end{align*}

\noi
where  we used the short-hand notation $\sum_{n_Z} = \sum_{n_Z \in (\Z^d)^Z}$ for a finite index set $Z$.
By duality, we have  $\| h \|_{n_{B} \to n_{C}} = \| h \|_{n_{C} \to n_{B}} 
= \| \cj h \|_{n_{B} \to n_{C}}$ for any tensor $h = h_{n_A}$. 
If $B = \varnothing$ or $C = \varnothing$,  then we have
$  \| h \|_{n_{B} \to n_{C}} = \| h \|_{n_A}$.
\end{definition}

We  now state the random tensor estimate
for multiple stochastic integrals
from \cite{CLO2, COZ}.
The random tensor estimate 
was first introduced 
in a breakthrough work \cite{DNY3}
by Deng, Nahmod, and Yue 
in the context of random variables;
see also 
\cite{Bring, OW3}.
See also  \cite{BO96} for a precursor of
the random tensor estimate.
In \cite{OWZ}, 
the fifth author with Wang and Zine
extended it 
to treat the case of iterated Wiener-Ito integrals, 
which was further extended
in~\cite{CLO2, COZ}
to the case of 
multiple stochastic integrals
with respect to fractional Brownian motions of Hurst parameter $\be\in(\frac12, 1)$.
In the following, we state a slightly simplified
version of the random tensor estimate from \cite{CLO2},
which is sufficient for our purpose.
See \cite{CLO2}
for a more general statement and its proof
(which follows closely the presentation in \cite{OWZ}, 
corresponding to the $\be = \frac 12$ case).

\begin{lemma}
\label{LEM:RT}

Fix $\frac 12 \le \be < 1$
and let $A$ be a finite index set with $k = |A| \ge 1$. 
Given a  tensor  $h=h_{bc n_A} \in \l^2_{bcn_A}$ 
with  $n_A \in (\Z^d)^{A}$ and $(b,c) \in (\Z^d)^{m}$ for some integer $m\ge2$, 
satisfying
\begin{align}
h_{-b, -c, -n_A} = \cj{h_{bc n_A}}, 
\quad 
(b,c) \in (\Z^d)^{m}, \ 
n_A \in (\Z^d)^{A}, 
\label{BM3}
\end{align}

\noi
 define the random tensor $H = H_{bc}$  by
\begin{align*}
H_{bc} =   \sum_{n_A} I_k \big[ h_{bcn_A} f_{bc} (t_A, n_A) \big]
\end{align*}

\noi
for $f\in \l^\infty_{bc n_A}( (\Z^d)^{k+m} ; \Hs^\be(\R_+^{k})  )$, satisfying
\begin{align}
f_{-b, -c, -n_A} = \cj{f_{bc n_A}}, 
\quad 
(b,c) \in (\Z^d)^{m}, \ 
n_A \in (\Z^d)^{A}, 
\label{BM4}
\end{align}

\noi
where 
$\Hs^\be(\R_+^{k})$ 
is as in  \eqref{BM0a} and \eqref{BM1}
\textup{(}see also \eqref{BM0}\textup{)}
and 
$I_k$ denotes the multiple stochastic integral 
%associated with an isonormal process over $\l^2(\Z^d; \Hs^\g(\R_+))$ 
defined in Subsection~\ref{SUBSEC:FBM}.
Then, given any $\ta > 0$
and finite $p \ge1$, we have
\noi
\begin{align}
\begin{split}
\big\| \| H_{bc} \|_{b \to c} \big\|_{L^p(\Om)}
&
\les p^\frac k2
\|f_{bc}(t_A,n_A)\|_{\l^\infty_{bc n_A}\Hs^\be_{t_A}}
\| h_{bcn_A} \|_{\l^2_{bcn_A}}^{\ta}\\
 &
 \quad \times \Big( \max_{(B,C)} \|    h_{bcn_A} \|_{b n_B \to c n_C} \Big)^{1 - \ta} , 
\end{split}
\label{BM5}
\end{align}

\noi
where the maximum is taken over  all partitions $(B,C)$ of $A$.

\end{lemma}

\subsection{Kolmogorov's continuity criterion}
\label{SUBSEC:Kol}

We conclude this section by stating a version of 
Kolmogorov's continuity criterion
for an operator-valued  rough path, 
namely, a pair $(\XX, \XXb)$ of two-parameter operator-valued stochastic processes, 
satisfying the relation  \eqref{KL1}
(see Definition \ref{DEF:control}\,(ii)).

\begin{lemma}\label{LEM:kolm}
Let $V$ be a Banach space and  $T > 0$.
Let $(\XX, \XXb)
\in  C_{2,T}\L(V)\times C_{2,T}\L(V)$ be a pair of two-parameter stochastic processes, 
satisfying
\begin{align}
\begin{split}
(\updl \XX)_{t_1,t_2,t_3}& =0, \\
(\updl \XXb)_{t_1,t_2,t_3} &  = \XX_{t_1,t_2} \circ \XX_{t_2,t_3}
\end{split}
\label{KL1}
\end{align}

\noi
 for any $(t_1,t_2,t_3)\in\Dl_{3,T}$.
Given a Banach subspace $V_0$ of $V$, 
suppose that there exist $M_1, M_2>0$, $p \ge 1$, and $\frac 1p < \be \le 1$
such that
\begin{align}
\begin{split}
\big\| \|  \XX_{t,r} \|_{\L(V; V_0)} \big\|_{L^p(\Om)} & \le M_1  |t-r|^\be,\\
\big\| \|  \XXb_{t,r} \|_{\L(V; V_0)} \big\|_{L^\frac{p}{2}(\Om)}&  \le M_2  |t-r|^{2\be}
\end{split}
\label{KL2}
\end{align}

\noi
for any $(t, r) \in \Dl_{2, T}$.
Then, for any $0<\al<\be-\frac{1}{p}$, there exists 
a constant $C_{\be,\al, T} >0$, 
independent of $p$,  such that
\begin{align}
\begin{split}
\big\| \|\XX\|_{C^\al_{2,T} \L(V; V_0)} \big\|_{L^p(\Om)} & \le C_{\be,\al, T} M_1,\\
\big\| \|\XXb\|_{C^{2\al}_{2,T} \L(V; V_0)} \big\|_{L^\frac p2(\Om)} & \le C_{\be,\al, T} 
(M_1^2 +M_2).
\end{split}
\label{KL3}
\end{align}

In particular, 
if, given $0 < \be\le 1$,  there exist $M, \ta_1, \ta_2 > 0$ such that 
\begin{align}
\begin{split}
\big\| \|  \XX_{t,r} \|_{\L(V; V_0)} \big\|_{L^p(\Om)} \le M p^{\ta_1} |t-r|^\be,\\
\big\| \|  \XXb_{t,r} \|_{\L(V; V_0)} \big\|_{L^\frac p2(\Om)} \le M p^{\ta_2} |t-r|^{2\be}
\end{split}
\label{KL4}
\end{align}

\noi
for any finite $p > \frac 1\be$
and for any $(t, r) \in \Dl_{2, T}$, 
 then 
 for any $0<\al<\be-\frac{1}{p}$, there exists 
a constant $C_{\be,\al, T} >0$ such that
\begin{align}
\begin{split}
\big\| \|\XX\|_{C^\al_{2,T} \L(V; V_0)} \big\|_{L^p(\Om)} \le C_{\be,\al, T} Mp^{\ta_1}, \\
\big\| \|\XXb\|_{C^{2\al}_{2,T} \L(V; V_0)} \big\|_{L^\frac p2(\Om)} \le C_{\be,\al, T} M
(p^{2\ta_1} + p^{\ta_2})
\end{split}
\label{KL5}
\end{align}

\noi
for any finite $p > \frac 1\be$.
As a consequence, we have
\[ (\XX, \XXb) 
\in C^\al_{2,T} \L(V; V_0) \times 
C^{2\al}_{2,T} \L(V; V_0), 
\]

\noi
almost surely.
\end{lemma}

Since Lemma \ref{LEM:kolm}
follows
from 
a straightforward modification of the proof of 
\cite[Theorem~3.1]{FH20}, 
we omit details.
Furthermore, 
as pointed out on \cite[p.\,41]{FH20}, 
the usual Kolmogorov continuity criterion 
for $\XX$ is contained in the proof of Lemma \ref{LEM:kolm}
by ignoring all considerations related to the second-order
process $\XXb$.
Namely, if the first conditions in \eqref{KL1}
and \eqref{KL2} (and in 
\eqref{KL1}
and \eqref{KL4})
hold, 
then the first bound in \eqref{KL3}
(and in \eqref{KL5}, respectively)
holds.

\section{Nonlinear Young differential equation with a 
perturbation}
\label{SEC:Young}

In this section, we 
go over pathwise study of stochastic perturbations
of the following nonlinear YDE:
\begin{equation}
\label{YDE0}
\uu = u_0 + \I^\XX(\uu)
\end{equation}

\noi
in a general setting, 
where $\XX$ is a given nonlinear driver of suitable regularity
studied in~\cite{CGLLO1};
see \cite[Section~3]{CGLLO1}
for a local well-posedness result of \eqref{YDE0} in a general setting.
See also~\cite{G23} for a recent review on
nonlinear YDEs.

In \cite{CGLLO1}, 
the last four authors with Chouk  constructed 
 a nonlinear Young integral  $\I^\XX(\uu)$ 
with a suitable nonlinear Young driver $\XX$, using
 the sewing lemma (Lemma~\ref{LEM:sew}), 
 which is reviewed 
in Subsection \ref{SUBSEC:Y2}.
In Subsection \ref{SUBSEC:Y3}, 
we then present a  local well-posedness result 
(Proposition \ref{PROP:main})
of the following 
nonlinear YDE with a Young perturbation:
\begin{equation}
\label{YDE1}
\uu = u_0 + \I^\XX(\uu) + \I^\YY(\uu), 
\end{equation}

\noi
where
$\I^\XX(\uu)$
and $\I^\YY(\uu)$
are the (nonlinear and linear, respectively)
Young integrals of~$\uu$
with respect to 
Young drivers $\XX$ and $\YY$, respectively;
see Lemma \ref{LEM:int1}
for the construction of Young integrals.
For our application, 
the Young integral
$\I^\YY(\uu)$ in \eqref{YDE1} corresponds
to 
the stochastic term $\PPsi(\uu)$
in \eqref{psi2} (in the case of SmKdV \eqref{skdv2}), 
representing the effect of the multiplicative noise
$u \phi \z$ (at the level of the modulated interaction representation)
in the fractional-in-time case (with Hurst
parameter $\frac 12 < \be < 1$).

In Subsection \ref{SUBSEC:Y4}, 
we present a brief review
of rough integrals of controlled paths
given a rough driver $(\YY, \YYb)$.
While the material presented 
in Subsection~\ref{SUBSEC:Y4}
is standard by now
(see, for example, \cite{Gub04, FH20}), 
we decided to include the discussion 
for non-experts, in particular those working in dispersive PDEs.
Finally, in Subsection \ref{SUBSEC:Y5}, 
we present a local well-posedness result
(Proposition \ref{PROP:main2})
of the 
following nonlinear YDE with a rough perturbation:
\begin{equation}
\label{YDE1a}
\uu = u_0 + \I^\XX(\uu) + \I^{\YY, \mathbb{Y}}(\uu), 
\end{equation}

\noi
where $\I^\XX(\uu)$ is the nonlinear Young integral of $\uu$
with a nonlinear Young driver $\XX$  and 
$\I^{\YY, \YYb}(\uu)$ is the rough integral of $\uu$
with a rough driver $(\YY, \YYb)$.
For our application, 
the rough integral
%$\I^{\YY, \YYb}(\uu)$ 
corresponds
to 
the stochastic term $\wt \PPsi(\uu)$
in \eqref{psi3} (in the case of SmKdV~\eqref{skdv3}
with the associated drivers $\wt \YY$ and $\wt \YYb$ in \eqref{sto2x}
and \eqref{sto3x}), 
representing the effect of the multiplicative noise
$\jb{\dx}^{-\eps_0} u \phi \xi$
(at the level of the modulated interaction representation)
in the white-in-time case (with Hurst
parameter $\be = \frac 12$).

In Sections~\ref{SEC:driver}
and Appendix~\ref{SEC:NLS}, we
will establish almost sure regularity properties
of relevant drivers
for 
  the stochastic modulated KdV equations, \eqref{skdv2}
and \eqref{skdv3},
and the stochastic modulated Schr\"odinger  equation
 \eqref{NLS2}. 
Then, pathwise local well-posedness
and regularization by noise 
for these equations
(Theorems \ref{THM:2}, \ref{THM:4}, and~\ref{THM:NLS})
follow from 
 Propositions~\ref{PROP:main}
and~\ref{PROP:main2}.

\begin{remark}\label{REM:mean0}
\rm

In Subsections \ref{SUBSEC:Y3} and 
\ref{SUBSEC:Y5}, 
we present local well-posedness
results for \eqref{YDE1}
and~\eqref{YDE1a}, respectively,
for $H^s(\M)$-valued functions, 
$\M = \T^d$ or $\R^d$.
With obvious modifications, 
the corresponding results hold
for  $H^s_0(\T)$-valued functions.
We omit details.

\end{remark}

\subsection{Sewing lemma and nonlinear Young integration}
\label{SUBSEC:Y2}

In this subsection, we briefly go over the construction of 
 a nonlinear Young integral 
 from \cite{CGLLO1}, 
based on the sewing lemma due to the second author \cite{Gub04}.
See also 
 \cite{HK, G23} 
 for the construction of  nonlinear Young integrals
 in a more general setting.
 The following version of the sewing lemma 
 is taken from~\cite{GT10} 
with slight modifications; 
see  
\cite[Proposition~2.3, Corollaries  2.4 and  2.5]{GT10}.
See also \cite[Lemma~4.2]{FH20}.

We set 
$C^{1+}_{n,T}V = \bigcup_{\g>1}C^\g_{n,T}V$.

\begin{lemma}[sewing lemma]
\label{LEM:sew}

Let $V$ be a Banach space and 
fix $T>0$. 
Then,  there exists a unique linear map \textup{(}called the sewing map\textup{)}
$\Lambda:C^{1+}_{3,T} V \cap 
\Ker \updl|_{C_{3,T}V}
\to C^{1+}_{2,T}V$ such that 

\smallskip
\begin{enumerate}
\item[(i)] 
We have 
$\updl \Lambda h = h$
for each  $h\in C_{3,T}V\cap \Ker  \updl|_{C_{3,T} V}$.

\smallskip

\item[(ii)]
 For each $\kk >1$,
the sewing map $\Lambda$ is continuous
from $C^\kk _{3,T}V\cap 
\Ker  \updl|_{C_{3,T} V}$ to 
$C^\kk _{2,T}V$ such that 
\begin{align}
\|\Lambda h \|_{C^\kk _{2,T}V} 
\le \frac{1}{2^\kk - 2}  \| h \|_{C^\kk _{3,T}V} 
\label{sew1}
\end{align}

\noi
for any $h\in C^\kk _{3,T}V$.

\smallskip
\item[(iii)] 
Given any  $g\in C_{2,T}V$ 
with 
$\updl g\in C^{1+} _{3,T}V$, 
 there exists  unique
$f\in C([0, T];V)$  \textup{(}modulo an additive  constant\textup{)} 
such that 
$\updl f = (\Id - \Lambda \updl)g$. 
In addition, 
we have 
\begin{align*}
(\updl f)_{t,r} = \lim_{|\Pi([r,t])|\to 0} 
\sum_{j=0}^{n-1} g_{t_j,t_{j+1}}
%\label{sew2}
\end{align*}

\noi
 for any $(t,r)\in \Dl_{2,T}$,
 where 
 the limit is over any partition
 $\Pi([r,t])$  
 of  $[r,t]$\textup{:} 
\[\Pi ([r,t]) = \{r = t_n < \dots < t_1 <  t_0 = t\}\]
whose mesh size 
$|\Pi([r,t])| = \max_{j} |t_j-t_{j+1}|$ 
tends to $0$.

\end{enumerate}

\end{lemma}

We now review the essential idea
 on the construction of a nonlinear Young integral 
 via the sewing lemma.
Fix a driver  $\XX\in C^\g_{2,T}\Lip_k(V)$, where $C^\g_{2,T}\Lip_k(V)$
is as in \eqref{Ho3},  such that 
\begin{align}
\XX_{t,r}(0)=0
\label{Ja0}
\end{align} 

\noi
for any $(t,r) \in  \Dl_{2,T}$, 
and 
\begin{align}
\label{Ja1}
(\updl \XX)_{t_1,t_2,t_3} = 0
\end{align}

\noi
for any $(t_1,t_2,t_3)\in \Dl_{3,T}$.
In the following, we assume  that the driver $\XX_{t, r}$
is given by an integral operator
over the time interval $[r, t]$;
see, for example, \eqref{K1x}
for the bilinear driver associated with the modulated KdV \eqref{kdv1}.
Given 
$f\in \CC^{\al}([0,T];V)$ for some $0 < \al < 1$, 
%such that $f(0) \in V$, 
%our goal is to construct
the nonlinear Young integral $\I^\XX(f)$ of $f$
is a unique function 
(in a suitable class of functions defined on $[0, T]$)
with $\I^\XX(f)(0) = 0$
whose increment is given by 
\begin{align}
\updl 
\big(\I^\XX(f)\big)_{t, r} 
= \XX_{t, r}(f_\bul)
\label{YQ1}
\end{align}

\noi
for any $(t, r) \in\Dl_{2, T}$.
Here, $f_\bul$ denotes 
the function $f$ evaluated at the variable of integration~$\bul$, satisfying $r \le \bul \le  t$.
As it is, the expression on the right-hand side of \eqref{YQ1}
is not well defined.
By replacing $\bul$ by the left endpoint $r$, 
we have 
\begin{align}
\XX_{t, r}(f_\bul)
= \XX_{t, r}(f_r) + R_{t, r}
\label{YQ2}
\end{align}

\noi
for some two-parameter process $R = R^{\XX, f}$.
Our goal is to find one such error term $R$ with sufficient regularity, 
which will allow us to define the nonlinear Young integral
 $\I^\XX(f)$
 as the unique limit of 
Riemann-Stieltjes type sums; see \eqref{YQ6} and \eqref{YQ7}.
Define  $\Theta$ on $\Dl_{2,T}$
by setting
\begin{align}
\Theta_{t,r} = \XX_{t,r}(f_r), \quad (t,r)\in \Dl_{2,T}.
\label{Ja1a}
\end{align}

\noi
By applying the coboundary operator $\updl$
to both sides of \eqref{YQ2}
together with \eqref{Ja1} and \eqref{Ja1a}, 
any error term $R$ in~\eqref{YQ2} (if it exists) satisfies
\begin{align}
(\updl R)_{t_1, t_2, t_3}
= - (\updl  \Theta)_{t_1, t_2, t_3}
=  \XX_{t_1, t_2}( f_{t_2}) - 
  \XX_{t_1, t_2}( f_{t_3})
\label{YQ3}
\end{align}

\noi
for any $(t_1, t_2, t_3) \in\Dl_{3, T}$.
Then, from  the regularity assumptions on $\XX$
and $f$
with~\eqref{Ho2} and~\eqref{Ho3}, 
we see that 
$\updl R 
 \in C^{\g + \al}_{3,T}V$;  see \cite[(3.27)]{CGLLO1}.
Hence, if $\g + \al > 1$, 
then we can apply the sewing lemma (Lemma~\ref{LEM:sew})
to define an error term $R$ by the relation:
\begin{align}
R = - \Ld \updl  \Theta
\in C^{\g + \al}_{2,T}V,
\label{YQ4} 
\end{align}
where $\Ld$ denotes the sewing map.
Then, 
putting
 \eqref{YQ1},  
 \eqref{YQ2},  
and \eqref{YQ4} together,   
we define the {\it nonlinear Young integral} $\I^\XX(f)$ of $f$
(with respect to the nonlinear Young driver $\XX$)
to be  a unique function 
in $ C^\g([0,T]; V)$
with $\I^\XX(f)(0) = 0$
whose increment is given by 
\begin{align}
\updl (\I^\XX(f))
= (\Id - \Ld \updl )\Ta.
\label{YQ5}
\end{align}

\noi
In this case, $\I^\XX(f)$ 
is given by 
the unique limit of 
Riemann-Stieltjes type sums:
\begin{align}
\I^\XX(f)(t) 
= \lim_{|\Pi([0,t])|\to 0} 
\sum^{n-1}_{j=0} \Theta_{t_j,t_{j+1}}
= \lim_{|\Pi([0,t])|\to 0} 
\sum^{n-1}_{j=0} \XX_{t_j,t_{j+1}}(f_{t_{j+1}}), 
\label{YQ6}
\end{align}

\noi
 where 
the limit is in the sense of Lemma \ref{LEM:sew}\,(iii).
Note that the regularity $\g + \al > 1$
of $R$ defined in \eqref{YQ4}
guarantees that the contribution from the error term $R$ in \eqref{YQ6}
vanishes:
\begin{align}
 \lim_{|\Pi([0,t])|\to 0} 
\sum^{n-1}_{j=0} R_{t_j,t_{j+1}} = 0.
\label{YQ7}
\end{align}

\begin{remark}\label{REM:uniq1}
\rm
We point out that 
an error term $R$ in \eqref{YQ2} is not unique, since
if we find one error term $R$ satisfying \eqref{YQ2}, 
any perturbation of $R$ by the addition of a smooth two-parameter process
satisfies \eqref{YQ2}.
The main point is that, while an error term $R$ in  \eqref{YQ2}
is not unique, the nonlinear Young integral $\I^\XX(f)$ defined in \eqref{YQ6}
is unique thanks to \eqref{YQ7}.
A similar comment applies to 
the construction of a rough integral discussed in Subsection~\ref{SUBSEC:Y4}.
\end{remark}

The following lemma summarizes
the relevant statements
on the construction of nonlinear Young integrals;
see \cite[Lemma~3.4 and Remark~3.6]{CGLLO1}.

\begin{lemma}[nonlinear Young integral]
\label{LEM:int1}
Let $V$ be a Banach space.
Given $0 < \g < 1$, 
 $T>0$, and $k\in \N$,  
let $\XX\in C^\g_{2,T}\Lip_k(V)$, 
satisfying \eqref{Ja0} and \eqref{Ja1}, 
where $C^\g_{2,T}\Lip_k(V)$ is as in~\eqref{Ho3}.
Given 
  $f\in \CC^{\al}([0,T];V)$ for some $0 < \al < 1$, 
%  such that $f(0) \in V$, 
let   $\Theta$  be as in \eqref{Ja1a}.
Suppose that 
 $\kk : =  \g +\al > 1$.
Then,  the following statements hold\textup{:}

\smallskip

\begin{enumerate}

\item[(i)] 
We have $\Theta\in C^\g_{2,T}V$
and
$\updl \Theta \in C^\kk_{3,T}V$.

\smallskip

\item[(ii)]
There exists 
a unique function 
$\I^\XX(f) \in C^\g([0,T]; V)$
with  $\I^\XX(f)(0) =0$, satisfying~\eqref{YQ5}.
Moreover,
we have 
\begin{align}
\|\updl \I^\XX(f) - \Theta\|_{C^\kk_{2,T}  V}
& \leq 
C_1(f)
\|\XX\|_{C^\g_{2,T}\Lip_k(V)},
\notag 
%\label{Ja2}
\\
\|\I^\XX(f)\|_{C^\g_T  V} 
& \leq
C_2(f)
\|\XX\|_{C^\g_{2,T}\Lip_k(V)}, 
\label{Ja3}
\end{align}

\noi
and \eqref{YQ6} holds, 
where $C_1(f)$ and $C_2(f)$ are given by 
\begin{align}
\begin{split}
C_1(f)
&= 
\frac{1}{2^\kk-2}
\big(1+ 2\|f\|_{ L^\infty_T V}\big)^{k-1} \|f\|_{C^\al_T  V}, \\
C_2 (f)
&=  T^\al C_1(f) + 
\big(1 + \|f\|_{ L^\infty_T V} \big)^{k-1}
\|f\|_{L^\infty_T V}.
\end{split}
\label{Ja3x}
\end{align}

\smallskip

\item[(iii)]
\textup{(persistence of regularity).}
Let  $k \ge 2$ and $V_0 \hookrightarrow V$
be a Banach subspace of $V$.
In addition, suppose that 
 $\XX\in C^\g_{2,T}\Lip_k(V, V_0; V_0)$, 
 where $\Lip_k(V, V_0; V_0)$ is as in \eqref{Lip2}, 
 and that 
   $f\in \CC^{\al}([0,T];V_0)$.
Then, 
we have 
$\I^\XX(f)\in C^\g([0,T]; V_0)$
with the bound\textup{:}
\begin{align*}
\|\I^\XX(f)\|_{C^\g_T  V_0} 
& \leq
C_3(f)
\|X\|_{C^\g_{2,T}\Lip_k(V, V_0; V_0)}, 
\end{align*}

\noi
where $C_3(f)$ is  given by 
\begin{align*}
C_3 (f)
&=  
2\big(1 + \|f\|_{ L^\infty_TV}\big)^{k-1} 
\big(1 + \|f\|_{ L^\infty_T V_0} \big)\\
& \quad + 
\frac{T^\al}{2^\kk-2}\Big(
\big(1+ 2\|f\|_{L^\infty_T V} \big)^{k-1} \|f\|_{C^\al_T  V_0}\\
& \hphantom{XXXXXX} + \big(1+ 2\|f\|_{L^\infty_T V} \big)^{k-2} \|f\|_{C^\al_T  V}
\big(1+ 2\|f\|_{L^\infty_T V_0} \big)\Big).
\end{align*}

\smallskip

\item[(iv)]
\textup{(nonlinear smoothing).}
%Let $X$ be as in Lemma \ref{LEM:int1}.
Suppose in addition that  $\XX\in C^\g_{2,T}\Lip_k(V; V_0)$
for some Banach subspace $V_0$ of $V$.
Then, we have 
$\I^\XX(f)\in C^\g([0,T]; V_0)$
with the bound\textup{:}
\begin{align*}
\|\I^\XX(f)\|_{C^\g_T  V_0} 
& \leq
C_2(f)
\|\XX\|_{C^\g_{2,T}\Lip_k(V; V_0)}, 
%\label{Ja300}
\end{align*}

\noi
where $C_2(f)$ is as in \eqref{Ja3x}.

\smallskip

\item[(v)]
\textup{(stability).}
Given $\XX^j \in  C^\g_{2,T}\Lip_k(V)$
with $\XX^j(0) = 0$ and $\updl \XX^j = 0$, $j = 1, 2$, 
we have 
\begin{align}
\|\I^{\XX^1}(f) - \I^{\XX^2}(f)\|_{\CC^\g_T V}  
\le 
C_2(f) 
\|\XX^1 - \XX^2\|_{C^\g_{2,T}\Lip_k(V)}
\label{Jaa4x}
\end{align}

\noi
for 
any  $f \in \CC^{\al}([0,T];V)$, 
where 
$\I^{\XX^j}(f)$ denotes
 the nonlinear Young integral 
 of $f$
 with the driver $\XX^j$, $j = 1, 2$.
In particular, 
if a sequence  
   $\{\XX^N\}_{N \in \N}
\subset C^\g_{2,T}\Lip_k(V)$
with $\XX^N(0) = 0$ and $\updl \XX^N = 0$
 converges to $\XX$ in 
$C^\g_{2,T}\Lip_k(V)$
  as $N \to \infty$,
then 
the corresponding nonlinear Young integral 
 $\I^{\XX^N}(f)$ of $f$
converges to 
$\I^\XX(f)$
 in $\CC^\g([0, T]; V)$
as $N \to \infty$, 
uniformly 
in  
$f \in B_K$,  
where
$B_K$ denotes the  ball 
in $\CC^{\al}([0,T];V)$
of radius $K > 0$ centered at the origin.

\end{enumerate}
\end{lemma}

While the bound \eqref{Jaa4x}
is not explicitly stated in 
 \cite[Lemma~3.4]{CGLLO1}, 
 it easily follows from 
 the proof of 
 \cite[Lemma~3.4\,(iv)]{CGLLO1}.

\subsection{Nonlinear Young differential equation
with a Young perturbation}
\label{SUBSEC:Y3}

In this subsection, we establish local well-posedness
of the nonlinear YDE  with a Young 
perturbation~\eqref{YDE1}:
\begin{equation}
\label{YDE3x}
\uu = u_0 + \I^\XX(\uu) + \I^\YY(\uu), 
\end{equation}

\noi
where $\I^\XX(\uu)$ 
(and $\I^\YY(\uu)$) is the (nonlinear) Young integral of $\uu$
with a nonlinear Young driver $\XX$ 
(and a linear Young driver $\YY$, respectively) 
constructed in Lemma \ref{LEM:int1}.

\begin{proposition}
\label{PROP:main}
Let $\M = \T^d$ or $\R^d$.

\smallskip

\noi
\textup{(i) (local well-posedness).}
Given   $s \in \R$, $\frac 12 < \g_1, \g_2  < 1$,  
an integer $k \ge 2$, and $T > 0$, 
let     $\XX \in \cX^{s, \g_1}_k([0, T]\times \M)$
and 
$\YY \in \cX^{s, \g_2}_1([0, T]\times \M)$, 
where $ \cX^{s, \g}_k([0, T]\times \M)$ is  as  in~\eqref{X1}
and~\eqref{X1z}.
Then, 
 the nonlinear Young differential equation \eqref{YDE3x} with 
a Young perturbation
is locally well-posed in $H^s(\M)$.
More precisely, given  $u_0\in H^s(\M)$, 
there exists
 a unique solution $\uu \in \cC^{\g_1\wedge \g_2} ([0, \tau]; H^s(\M))$
to~\eqref{YDE3x}  
with $\uu|_{t= 0} = u_0$, where the local existence time 
\[\tau = 
\tau\Big(\|u_0\|_{H^s(\M)}, \|\XX\|_{\cX^{s, \g_1}_k(T)},
\|\YY\|_{\cX^{s, \g_2}_1(T)}\Big) \in (0, T\wedge 1]\]

\noi
\noi
is non-increasing in each of the arguments.
In particular, 
the following blowup alternative holds
for \eqref{YDE3x}\textup{;}
let $T_* \in (0, \infty]$ be the maximal time of existence
of the solution~$\uu$ to~\eqref{YDE3x}.
Then, we have either 
\begin{align}
T_* = T \qquad \text{or} \qquad \lim_{t \nearrow T_*} \|\uu(t)\|_{H^s} = \infty, 
\label{BA1}
\end{align}

\noi
where we allow $T$ to take the value $\infty$, 
meaning that 
    $\XX \in \cX^{s, \g_1}_k([0, T_0]\times \M)$
and 
$\YY \in \cX^{s, \g_2}_1([0, T_0]\times \M)$
for any $T_0 > 0$ 
\textup{(}but $\sup_{T_0 > 0 } \| \XX\|_{\cX^{s, \g_1}_k(T)}$
and $\sup_{T_0 > 0 } \| \YY\|_{\cX^{s, \g_2}_k(T)}$
 may be infinite\textup{)}.

\smallskip

\noi
\textup{(ii) (persistence of regularity).}
In addition, suppose that $u_0 \in H^{s_0}(\M)$ for some $s_0 > s$
and that  $\XX \in \cY^{s, s_0, \g_1}_k([0, T]\times \M)$
and $\YY \in \cX^{s, s_0, \g_2}_1([0, T]\times \M)$, 
where $\cY^{s, s_0, \g}_k([0, T]\times \M)$  
and 
 $\cX^{s, s_0, \g}_1([0, T]\times \M)$  
 are as in \eqref{X2c}
and 
\eqref{X1z1}.
Then, 
by possibly making $\tau > 0$ smaller by a multiplicative constant 
\textup{(}in particular $\tau > 0$
depends on the $H^s$-norm of the initial data $u_0$
but not on its $H^{s_0}$-norm\textup{)}, 
we have $\uu \in \cC^{\g_1\wedge \g_2} ([0, \tau]; H^{s_0}(\M))$, 
 where $\uu$ is the solution to~\eqref{YDE3x}
 constructed in Part (i).

\smallskip

\noi
\textup{(iii) (nonlinear smoothing).}
In addition to the hypotheses in Part (i), 
suppose  that  $\XX \in \cX^{s, s_0, \g_1}_k([0, T]\times \M)$
and $\YY \in \cX^{s, s_0, \g_2}_1([0, T]\times \M)$
for some $s_0 > s$.
%where $\cX^{s, s_0, \g}_k([0, T]\times \M)$  is  as in \eqref{X1x}.
Then, we have
 $\uu - u_0 \in \cC^{\g_1\wedge \g_2}([0, \tau]; H^{s_0}(\M))$, 
 where $\uu$ is the solution to \eqref{YDE3x}
 constructed in Part (i).

\smallskip

\noi
\textup{(iv) (convergence).}
In addition to the hypotheses in Part (i), 
suppose  that a sequence $\{\XX^N\}_{N \in \N}\subset 
 \cX^{s, \g_1}_k([0, T]\times \M)$
\textup{(}and $\{\YY^N\}_{N \in \N}\subset 
 \cX^{s, \g_2}_1([0, T]\times \M)$\textup{)}
 converges to $\XX$ in 
 $ \cX^{s, \g_1}_k([0, T]\times \M)$
\textup{(}and 
to  $\YY$ in 
 $ \cX^{s, \g_2}_1([0, T]\times \M)$, respectively\textup{)}
  as $N \to \infty$.
Then, 
by possibly making $\tau > 0$ smaller by a multiplicative constant,  the solution 
$\uu^N$  %\in \CC^\g([0, \tau]; H^s(\M))$
to the following nonlinear Young differential equation with a Young 
perturbation\textup{:}
\begin{equation*}
%\label{YDE2}
\uu^N = u_0 + \I^{\XX^N}(\uu^N)
+ \I^{\YY^N}(\uu^N)
\end{equation*}

\noi
converges to the solution $\uu$ to \eqref{YDE3x}
constructed in Part (i)  
 in $\CC^{\g_1\wedge \g_2}([0, \tau]; H^s(\M))$
as $N \to \infty$.
Moreover, 
given any $K > 0$, 
the rate of convergence of $\uu^N$ to $\uu$ is
uniform in $u_0 \in B_K$,  
where
$B_K$ denotes the  ball 
in $H^s(\M)$
of radius $K > 0$ centered at the origin.

\end{proposition}

\begin{remark} \label{REMLmain2}
\rm

A slight modification of the proof of Proposition \ref{PROP:main}
yields
local well-posedness
of the following general 
nonlinear   YDE  of  the  form:
\begin{equation*}
\label{YDE4}
\uu = u_0 + \sum_{j = 1}^J \I^{\XX_j}(\uu), 
\end{equation*}

\noi
where $\I^{\XX_j}(\uu)$ is the nonlinear Young integral of 
$\uu$ with a $k_j$-linear driver 
$\XX_j\in \cX^{s, \g_j}_{k_j}([0, T] \times \M)$
with $k_1 \ge k_2 \ge \cdots \ge k_J \ge 1$
and $\frac 12 < \g_j < 1$, $j = 1, \dots, J$.
Since the required modification is straightforward, we omit details.

\end{remark}

\begin{proof}[Proof of Proposition \ref{PROP:main}]

Since this proposition follows from a slight modification
of the proof of Proposition~3.8 in \cite{CGLLO1}, 
we only provide  a sketch.

Let $0 < \al < \g_1\wedge \g_2$ such that 
 $ (\g_1\wedge \g_2) + \al>1$. 
Given $0 < \tau \le T\wedge 1$ (to be chosen later), 
 define 
 the map
$\G = \G_{\XX, \YY, u_0}$ on $\cC^\al([0, \tau]; H^s(\M))$
by setting
\begin{align}
\G(\uu) = u_0 + \I^\XX(\uu)+ \I^\YY(\uu), 
\label{YD3}
\end{align}

\noi
where 
$\I^\XX(\uu)$ and $\I^\YY(\uu)$
are the (nonlinear and linear, respectively)
Young integrals with the drivers $\XX$ and $\YY$, respectively, 
constructed via
Lemma \ref{LEM:int1}.

Proceeding as in \cite[(3.53)]{CGLLO1}, we have
\begin{align}
\label{YD5a}
\begin{split}
\|\G(\uu)\|_{\CC^\al_\tau H^s_x}
& \le \|u_0\|_{H^s} +  C\tau^{\g_1 - \al}
\|\XX\|_{\cX^{s, \g_1}_k(T)}
\big(1+ 
\|\uu\|_{\CC^\al_{\tau } H^s_x}\big)^{k-1}
\|\uu\|_{\CC^\al_{\tau } H^s_x}\\
& \quad
+  C\tau^{\g_2 - \al}
\|\YY\|_{\cX^{s, \g_2}_1(T)}
\|\uu\|_{\CC^\al_{\tau } H^s_x}
\end{split}
\end{align}

\noi
for any $0 < \tau \le T\wedge 1$.
Then, by setting $K = 1 \vee 2 \|u_0\|_{H^s} $
and choosing 
\begin{align}
\begin{split}
\tau 
& = 
\tau\Big(K, \|\XX\|_{\cX^{s, \g_1}_k(T)}, \|\YY\|_{\cX^{s, \g_2}_1(T)}\Big) \\
& = 
\tau\Big(\|u_0\|_{H^s}, \|\XX\|_{\cX^{s, \g_1}_k(T)}, \|\YY\|_{\cX^{s, \g_2}_1(T)}\Big) >0
\end{split}
\label{YD4}
\end{align}
sufficiently small, 
we see that 
$\G$ maps the ball $B_K\subset 
\cC^\al([0, \tau]; H^s(\M))$
of radius $K > 0$ 
centered at the origin into itself.

Next, we show that $\G$ is a contraction on $B_K$. 
Proceeding as in \cite[(3.64)]{CGLLO1}, 
we have 
\begin{align}
\begin{split}
& \|\G(\uu^1)- \G(\uu^2)\|_{\cC^\al _{\tau} H^s_x}\\
& \quad \le  
\|\I^\XX(\uu^1)- \I^\XX(\uu^2)\|_{\CC^\al_{\tau}H^s_x}
+ \|\I^\YY(\uu^1)- \I^\YY(\uu^2)\|_{\CC^\al_{\tau}H^s_x}\\
&  
\quad \les 
\tau^{\g_1\wedge \g_2 - \al}
\Big(\|\XX\|_{\cX^{s, \g_1}_k(T)}  (1+K)^{k-1}
+ \|\YY\|_{\cX^{s, \g_2}_1(T)}\Big) \|\uu^1- \uu^2 \|_{\CC^\al _\tau H^s_x}
\end{split}
\label{YD14}
\end{align}

\noi
for $\uu^1, \uu^2 \in B_K\subset 
\cC^\al([0, \tau]; H^s(\M))$.
Hence, 
 by 
 choosing 
$\tau > 0$ of the form \eqref{YD4}
sufficiently small, 
we see that $\G$ is a contraction on $B_K$.
Therefore,  
from the Banach fixed point theorem, 
we conclude that there exists a unique solution 
$\uu \in B_K
\subset \cC^\al([0, \tau]; H^s(\M))$
to \eqref{YDE3x}, 
where 
 the uniqueness a priori holds only in the ball $B_K$
but, by a standard continuity argument, 
the uniqueness  can be extended to hold in the entire space $\CC^\al([0, \tau]; H^s(\M))$
(by possibly making $\tau$ smaller by a multiplicative constant).
Lastly,  by noting from Lemma \ref{LEM:int1}
that the right-hand side of~\eqref{YD3}
belongs to $\CC^{\g_1\wedge \g_2}([0, \tau]; H^s(\M))$, 
we conclude that $\uu \in \CC^{\g_1\wedge \g_2}([0, \tau]; H^s(\M))$.
This proves Part (i).

The claims in Parts (ii), (iii), and (iv) also follow
from straightforward 
modifications
of  the proof of Proposition 3.8 in \cite{CGLLO1}, 
and thus we omit details.
\end{proof}

\subsection{Rough integration}
\label{SUBSEC:Y4}

In Subsection \ref{SUBSEC:Y2}, 
we briefly went over the nonlinear Young integration theory, 
where the sum $\g + \al$ of the regularities of the 
$k$-linear driver $\XX$
and the input function~$f$ is greater than $1$.
When $\g + \al \le 1$,
the Young integration theory breaks down
and we need to resort to rough path theory, 
introduced in a seminal work by Lyons
\cite{Lyons1}, 
and controlled paths, introduced by the second author \cite{Gub04};
 see also 
\cite{Lyons2, FV10, FH20}
for textbook accounts on rough path theory.

In this subsection, 
we restrict our attention to  the {\it linear} case.
See, for example, 
\cite{Gub12, NX, OShao}
for a discussion in the nonlinear case.
In particular, see~\cite{OShao}
for a pedagogical treatment
of the construction of nonlinear rough integrals
and general well-posedness theory
of a nonlinear rough differential equation
with a rough perturbation, which is not treated in the current paper.

As in Subsection \ref{SUBSEC:Y2}, 
fix 
a driver  $\XX\in C^\g_{2,T}\L(V)$, 
satisfying \eqref{Ja0} and \eqref{Ja1}, 
and 
$f\in \CC^{\al}([0,T];V)$ for some $0 < \al < 1$.
As in the Young case, 
 we assume  that the driver $\XX_{t, r}$
is given by an integral operator
over the time interval $[r, t]$.
Our goal here is to define an 
integral $\I(f)$ of $f$
whose increment on $[r, t]$ is given by $\XX_{t, r}(f_\bul)$
under the weaker assumption $\g + \al \le 1$.
The main issue in this case is that 
$\updl \Ta$ in \eqref{YQ3}
does not have sufficient regularity to apply the sewing lemma (Lemma \ref{LEM:sew}).

\begin{definition}\label{DEF:control}\rm
Let $V$ be a Banach space
and fix $T > 0$.

\smallskip

\noi
(i) Given a driver  $\XX\in C^\g_{2,T}\L(V)$, 
satisfying \eqref{Ja0} and \eqref{Ja1}, 
we say that 
$f\in \CC^{\al}([0,T];V)$ for some $0 < \al < 1$
is controlled 
by $\XX$  if there exists $f'\in \CC^{\al}([0,T];V)$
and a remainder term\footnote{In some literature (see, for example,  \cite{FH20}),
the regularity $\be$ of the remainder term $R$ is taken to be $2\al$, 
which is too restrictive for our application.
See also the original paper \cite{Gub04}, 
where only $\be > \al$ is assumed.}
 $R = R^{\XX, f} \in  C^{\be}_{2, T} V$
for some $\al < \be < 1$
such that 
\begin{align}
(\updl f)_{t, r} = \XX_{t, r} (f_r') + R_{t, r}.
\label{RQ1}
\end{align}

\noi
The process $f'$, not uniquely determined in general,  is often called the Gubinelli derivative of $f$
(with respect to the driver $\XX$).
We denote 
the space of such controlled  paths $(f, f')$
by 
$\D^{\al, \be}_{\XX, T}(V) = \D^{\al, \be}_\XX([0, T]; V)$ 
and endow it with the following norm:
\begin{align}
\| (f, f') \|_{\D^{\al, \be}_{\XX, T}(V)} 
= \|f(0) \|_{V} + \|f'(0)\|_V
+ \| f' \|_{C^\al_TV} + \| R \|_{C^\be_{2, T}V},
\label{RQ1a}
\end{align}

\noi
where $R $ is defined by  the relation \eqref{RQ1}.

\medskip

\noi
(ii)
Given $0 < \g \le \frac12 $, 
let $\XX$ be as in Part (i).
We say that a pair $(\XX, \XXb) 
\in C^\g_{2,T}\L(V)\times  C^{2\g}_{2,T}\L(V)$
is a $\g$-H\"older rough path  
if it satisfies  Chen's relation:
\begin{align}
(\updl \XXb)_{t_1,t_2,t_3}  = \XX_{t_1,t_2} \circ \XX_{t_2,t_3}
\label{RQ2}
\end{align}

\noi
for any $(t_1, t_2, t_3) \in\Dl_{3, T}$.

\end{definition}

In the following, 
 we assume that $f$ is controlled by the driver $\XX$
with the Gubinelli derivative 
$f'\in \CC^{\al}([0,T];V)$.
 Recalling  that $\XX$ is linear, 
 it follows
from \eqref{RQ1} that 
\begin{align}
\begin{split}
\XX_{t, r}(f_\bul)
& = \XX_{t, r}(f_r) + 
 \XX_{t, r}\big((\updl f)_{\bul, r}\big)\\
& = \XX_{t, r}(f_r) +  \XXb_{t, r} (f_r') 
+ Q_{t, r} \\
& =:\Taa_{t, r} + Q_{t, r}
\end{split}
\label{RQ3}
\end{align}

\noi
for any $(t, r) \in\Dl_{2, T}$, 
where\footnote{In general,  
$\XX_{t, r}\circ \XX_{\bul, r}$ does not have a canonical pathwise
meaning and thus
the relation $\XXb_{t, r} = \XX_{t, r}\circ \XX_{\bul, r}$
should be read as 
$\XXb_{t, r}$ defining the value of  $\XX_{t, r}\circ \XX_{\bul, r}$.}
$\XXb_{t, r} = \XX_{t, r}\circ \XX_{\bul, r}$
and 
$Q_{t, r} = \XX_{t, r}(R_{\bul, r})$.
As in the Young case, 
our goal is to find an error term $Q$ with sufficient regularity, 
which will allow us to define the rough integral
 $\I^{\XX,  \XXb}(f)$ of $f$
  as the unique limit of 
Riemann-Stieltjes type sums; see \eqref{RQ7} and \eqref{RQ8}.

Suppose that $(\XX, \XXb)$ is 
a $\g$-H\"older rough path
as in Definition \ref{DEF:control}\,(ii).
By applying the coboundary operator $\updl$
to~\eqref{RQ3}
together with \eqref{Ja1} and~\eqref{RQ2}, 
any error term $Q$ in~\eqref{RQ3} (if it exists) satisfies
\begin{align}
(\updl Q)_{t_1, t_2, t_3}
= - (\updl  \Taa)_{t_1, t_2, t_3}
  = \XX_{t_1,t_2} (R_{t_2,t_3}) + \XXb_{t_1,t_2}\big( (\updl f')_{t_2,t_3}\big)
\label{RQ4}
\end{align}

\noi
for any $(t_1, t_2, t_3) \in\Dl_{3, T}$.
Then, from \eqref{Ho2} and the regularity assumptions on $(\XX, \XXb)$, 
$f'$, and $R$, 
we see that 
$\updl Q
 \in C^{(\g + \be)\wedge (2\g + \al)}_{3,T}V$.
Hence, if $(\g + \be)\wedge (2\g + \al) > 1$, 
then we can apply the sewing lemma (Lemma \ref{LEM:sew})
to define an error term $Q$ by the relation:
\begin{align}
Q = - \Ld \updl  \Taa
\in C^{(\g + \be)\wedge (2\g + \al)}_{2,T}V,
\label{RQ5} 
\end{align}
where $\Ld$ denotes the sewing map.
Then, 
from 
 \eqref{RQ3},  
and \eqref{RQ5}, 
we  define 
the {\it rough integral} $\I^{\XX, \XXb}(f)$ of $f$
(with respect to the rough driver $(\XX, \XXb)$)
to be  a unique function 
in $ C^\g([0,T]; V)$
with $\I^{\XX, \XXb}(f)(0) = 0$
whose increment is given by 
\begin{align}
\updl (\I^{\XX, \XXb}(f))
= (\Id - \Ld \updl )\Taa.
\label{RQ6}
\end{align}

\noi
As in the Young case, 
the rough integral $\I^{\XX, \XXb}(f)$ 
is given by 
the unique limit of 
Riemann-Stieltjes type sums:
\begin{align}
\begin{split}
\I^{\XX, \XXb}(f)(t) 
& = \lim_{|\Pi([0,t])|\to 0} 
\sum^{n-1}_{j=0} \Taa_{t_j,t_{j+1}}\\
& = \lim_{|\Pi([0,t])|\to 0} 
\sum^{n-1}_{j=0} \Big(\XX_{t_j,t_{j+1}}(f_{t_{j+1}})
+ \XXb_{t_j,t_{j+1}}(f_{t_{j+1}}')\Big), 
\end{split}
\label{RQ7}
\end{align}

\noi
 where 
the limit is understood in the sense of Lemma \ref{LEM:sew}\,(iii).
Note that the regularity 
$(\g + \be)\wedge (2\g + \al) > 1$
of $Q$ defined in \eqref{RQ5}
guarantees that the contribution from $Q$ in~\eqref{RQ7}
vanishes:
\begin{align}
 \lim_{|\Pi([0,t])|\to 0} 
\sum^{n-1}_{j=0} Q_{t_j,t_{j+1}} = 0.
\label{RQ8}
\end{align} 
 
\noi
As pointed out in the Young case
(Remark \ref{REM:uniq1}), 
the error term $Q$ is not unique, 
but it defines
the rough integral 
$\I^{\XX, \XXb}(f)$  in a unique manner.

\begin{lemma}
[rough integral]
\label{LEM:int2}

Let $V$ be a Banach space.
Given $0 < \g \le \frac 12 $
and  $T>0$, 
let $(\XX, \XXb) 
\in C^\g_{2,T}\L(V)\times  C^{2\g}_{2,T}\L(V)$
be a $\g$-H\"older rough path,   
%\textup{(}see Definition \ref{DEF:control}\textup{)}, 
satisfying \eqref{Ja0}, 
\eqref{Ja1}, and \eqref{RQ2}.
Given $0 < \al < \be < 1$, 
let $(f, f') 
\in \CC^{\al}([0,T];V)
\times 
\CC^{\al}([0,T];V)$
be a controlled  path
in 
$\D^{\al, \be}_{\XX, T}(V)$, 
satisfying
\eqref{RQ1}
for some remainder term
 $R = R^{\XX, f}
\in  C^{\be}_{2, T}V$.
Suppose that 
\begin{align}
\kk : = (\g + \be)\wedge (2\g + \al) > 1.
\label{RQQ2}
\end{align}

\noi
Then,  the following statements hold\textup{:}

\smallskip

\begin{itemize}

\item[(i)] 
Let $\Taa$ be as in \eqref{RQ3}\textup{:}
\begin{align}
\Taa_{t,r}  = \XX_{t,r}(f_r) + \XXb_{t,r} (f'_r)
\label{RQQ1}
\end{align}

\noi
for $(t,r) \in \Dl_{2,T}$.
Then, 
we have $\Taa\in C^\g_{2,T}V$
and
$\updl \Taa \in C^\kk_{3,T}V$.

\smallskip

\item[(ii)]
There exists 
a unique function 
$\I^{\XX, \XXb}(f) \in C^\g([0,T]; V)$
with  $\I^{\XX, \XXb}(f)(0) =0$, satisfying~\eqref{RQ6}, 
such that
\begin{align}
\begin{split}
 \|  \updl \I^{\XX, \XXb}(f) - \Taa\|_{C^\kk_{2,T}  V}
& \les
 T^{\g+ \be -\kk}
\|\XX\|_{C^\g_{2,T}\L(V)}
\| R\|_{C^\be_{2,T}V}\\
& \quad + 
 T^{2\g+ \al-\kk}
\|\XXb\|_{C^{2\g}_{2,T}\L(V)}
\| f'\|_{C^\al_{T}V} 
\end{split}
\label{RQQ3}
\end{align}

\noi
and
\begin{align}
\begin{split}
\|\I^{\XX, \XXb}(f)\|_{C^\g_T  V} 
& \les
\|\XX\|_{C^\g_{2,T}\L(V)}
\|f \|_{L^\infty_TV}\\
& \quad 
+  T^{\g}
\|\XXb\|_{C^{2\g}_{2,T}\L(V)}
\|f' \|_{L^\infty_TV}\\
& \quad 
+  T^{\be}
\|\XX\|_{C^\g_{2,T}\L(V)}
\| R\|_{C^\be_{2,T}V}\\
& \quad + 
 T^{\g+ \al}
\|\XXb\|_{C^{2\g}_{2,T}\L(V)}
\| f'\|_{C^\al_{T}V}.
\end{split}
\label{RQQ4}
\end{align}

\noi
Moreover,  \eqref{RQ7} holds.

\smallskip

\item[(iii)]
\textup{(nonlinear smoothing).}\footnote{While 
the  rough integration considered here is linear in 
the input function $f$, 
it is bilinear in terms of a noise in our application, 
and hence we use the terminology ``nonlinear smoothing''.}
Suppose in addition that 
 $(\XX, \XXb) 
\in C^\g_{2,T}\L(V; V_0)\times  C^{2\g}_{2,T}\L(V;V_0)$
for some Banach subspace $V_0$ of $V$.
Then, we have 
\begin{align}
\begin{split}
\|\I^{\XX, \XXb}(f)\|_{C^\g_T  V_0} 
& \les
\|\XX\|_{C^\g_{2,T}\L(V; V_0)}
\|f \|_{L^\infty_TV}\\
& \quad 
+  T^{\g}
\|\XXb\|_{C^{2\g}_{2,T}\L(V; V_0)}
\|f' \|_{L^\infty_TV}\\
& \quad 
+  T^{\be}
\|\XX\|_{C^\g_{2,T}\L(V; V_0)}
\| R\|_{C^\be_{2,T}V}\\
& \quad + 
 T^{\g+ \al}
\|\XXb\|_{C^{2\g}_{2,T}\L(V; V_0)}
\| f'\|_{C^\al_{T}V}.
\end{split}
\label{RQQ5}
\end{align}

\smallskip

\item[(iv)]
\textup{(stability).}
For $j = 1, 2$, 
let $(\XX^j, \XXb^j) 
\in C^\g_{2,T}\L(V)\times  C^{2\g}_{2,T}\L(V)$
be a $\g$-H\"older rough path  
satisfying \eqref{Ja0}, 
\eqref{Ja1}, and \eqref{RQ2}, 
and 
let $(f^j, (f^j)')\in  
\CC^{\al}([0,T];V)
\times 
\CC^{\al}([0,T];V)$
be a controlled  path
in 
$\D^{\al, \be}_{\XX_j, T}(V)$, 
satisfying
\begin{align*}
(\updl f^j)_{t, r} = \XX^j_{t, r} \big(  (f^j)_r'\big) + R^{\XX^j, f^j}_{t, r}
%\label{RQQ6}
\end{align*}

\noi
for some remainder term
 $R^{\XX^j, f^j}
\in C^{\be}_{2, T}V$.
Then, under the condition \eqref{RQQ2}, we have 
\begin{align}
 \| & \I^{\XX^1, \XXb^1}(f^1)
- \I^{\XX^2, \XXb^2}(f^2)\|_{C^\g_T  V} 
\notag\\
& \les
\|\XX^1- \XX^2\|_{C^\g_{2,T}\L(V)}
\|f^1 \|_{L^\infty_TV}
+ \|\XX^2\|_{C^\g_{2,T}\L(V)}
\|f^1 -f^2\|_{L^\infty_TV}
\notag\\
& \quad 
+  T^{\g}
\Big(
\|\XXb^1- \XXb^2\|_{C^{2\g}_{2,T}\L(V)}
\|(f^1)' \|_{L^\infty_TV}
\notag\\
& \hphantom{XXXXX}
+ \|\XXb^2\|_{C^{2\g}_{2,T}\L(V)}
\|(f^1)'- (f^2)' \|_{L^\infty_TV}\Big)
\notag\\
& \quad 
+  T^{\be}
\Big(
\|\XX^1- \XX^2\|_{C^\g_{2,T}\L(V)}
\| R^{\XX^1, f^1}\|_{C^\be_{2,T}V}
\label{RQQ7}
\\
& \hphantom{XXXXX}
+ 
\|\XX^2\|_{C^\g_{2,T}\L(V)}
\| R^{\XX^1, f^1} - R^{\XX^2, f^2}\|_{C^\be_{2,T}V}\Big)
\notag\\
& \quad + 
 T^{\g+ \al}
\Big(
\|\XXb^1 - \XXb^2\|_{C^{2\g}_{2,T}\L(V)}
\| (f^1)'\|_{C^\al_{T}V}
\notag\\
& \hphantom{XXXXX}
+ \|\XXb^2\|_{C^{2\g}_{2,T}\L(V)}
\| (f^1)' - (f^2)'\|_{C^\al_{T}V}
\Big), 
\notag
\end{align}

\noi
where 
$\I^{\XX^j, \XXb^j}(f^j)$ denotes
 the rough integral 
 of $f^j$
 with the rough driver $(\XX^j, \XXb^j)$, 
 $j = 1, 2$.

\end{itemize}

\end{lemma}

Since we only consider linear rough integrals, 
there is no notion of persistence of regularity in this case.

\begin{proof}[Proof of Lemma \ref{LEM:int2}]
%Since the claims made in this lemma are standard, 
%we will be brief here.
%
%
%
%\noi
(i) 
From \eqref{RQQ1}, we have 
\begin{align}
\|\Taa\|_{C^\g_{2, T} V}
\le 
\|\XX\|_{C^\g_{2,T}\L(V)}
\|f \|_{L^\infty_TV}
+  T^{\g}
\|\XXb\|_{C^{2\g}_{2,T}\L(V)}
\|f' \|_{L^\infty_TV}, 
\label{RZ1}
\end{align}

\noi
while, from 
\eqref{RQ4}, 
we have 
\begin{align}
\begin{split}
\|\updl \Taa\|_{C^\kk_{3, T} V}
& \le
 T^{\g + \be - \kk }
\|\XX\|_{C^\g_{2,T}\L(V)}
\| R\|_{C^\be_{2,T}V}\\
& \quad  + 
 T^{2\g+ \al - \kk}
\|\XXb\|_{C^{2\g}_{2,T}\L(V)}
\| f'\|_{C^\al_{T}V}.
\end{split}
\label{RZ2}
\end{align}

\medskip

\noi
(ii)
The bound \eqref{RQQ3} follows
from \eqref{RQ6}, \eqref{sew1},  %in the sewing lemma (Lemma \ref{LEM:sew}), 
and \eqref{RZ2}, 
while the bound \eqref{RQQ4} follows
from 
\eqref{RQQ3}
and \eqref{RZ1}.
Noting that \eqref{RQ8} follows 
from \eqref{RQ5} with~\eqref{RQQ2}, 
we see that~\eqref{RQ7}
follows from summing
$\updl \big(\I^{\XX, \XXb}(f)\big)
_{t_j, t_{j+1}}
= \Taa_{t_j, t_{j+1}} + 
Q_{t_j, t_{j+1}}$
over $j = 0, \dots, n$
and applying~\eqref{RQ8}.

\medskip

\noi
(iii)
The bound \eqref{RQQ5} follows
from a slight modification of the proof of \eqref{RQQ4}
with the definition \eqref{Ho3}
of $C^\g_{2,T}\L(V; V_0)$.

\medskip

\noi
(iv) 
Set $\Taa^j_{t,r}  = \XX_{t,r}^j(f_r^j) + \XXb_{t,r}^j \big((f^j)'_r\big)$, $j = 1, 2$.
Then, 
from \eqref{RQ6}, we have
\begin{align}
\begin{split}
 \updl  \big( & \I^{\XX^1, \XXb^1}(f^1)\big)_{t, r}
- \updl\big(\I^{\XX^2, \XXb^2}(f^2)\big)_{t, r}\\
& = (\XX_{t,r}^1- \XX_{t,r}^2)(f_r^1)
+ 
\XX_{t,r}^2(f_r^1- f_r^2)\\
& \quad
+ 
(\XXb_{t,r}^1 - \XXb_{t,r}^2) \big((f^1)'_r\big)
+ \XXb_{t,r}^2 \big((f^1)'_r- (f^2)'_r\big)\\
& \quad
- \big(\Ld (\updl  \Taa^1 - \updl  \Taa^2)\big)_{t, r}.
\end{split}
\label{RZ3}
\end{align}

\noi
From \eqref{RQ4}, we have 
\begin{align}
\begin{split}
( & \updl  \Taa^1)_{t_1, t_2, t_3}
- (\updl  \Taa^2)_{t_1, t_2, t_3}\\
& = -
(\XX_{t_1,t_2}^1-\XX_{t_1,t_2}^2) (R^{\XX^1, f^1}_{t_2,t_3})
-
\XX_{t_1,t_2}^2 (R^{\XX^1, f^1}_{t_2,t_3}- R^{\XX^2, f^2}_{t_2,t_3})\\
& \quad 
- (\XXb^1_{t_1,t_2}-\XXb^2_{t_1,t_2})\big( (\updl (f^1)')_{t_2,t_3}\big)
- \XXb^2_{t_1,t_2}\big( (\updl (f^1)')_{t_2,t_3}- (\updl (f^2)')_{t_2,t_3}\big).
\end{split}
\label{RZ4}
\end{align}

\noi
Then, the bound \eqref{RQQ7}
follows from a slight modification of the proof of \eqref{RQQ4}, 
using \eqref{RZ3} and~\eqref{RZ4}.
\end{proof}

\subsection{Nonlinear Young differential equation
with a rough perturbation}
\label{SUBSEC:Y5}

In this subsection, we establish  local well-posedness
of the nonlinear YDE with a rough 
perturbation~\eqref{YDE1a}:
\begin{equation}
\uu = u_0 + \I^\XX(\uu) + \I^{\YY, \YYb}(\uu), 
\label{RDE1}
\end{equation}

\noi
where $\I^\XX(\uu)$ is the nonlinear Young integral of $\uu$
with a $k$-linear Young driver $\XX$ constructed in Lemma \ref{LEM:int1} and 
$\I^{\YY, \YYb}(\uu)$ is the rough integral of $\uu$
with a rough driver $(\YY, \YYb)$
constructed in 
Lemma \ref{LEM:int2}.

\begin{proposition}
\label{PROP:main2}
Let $\M = \T^d$ or $\R^d$.

\smallskip

\noi
\textup{(i) (local well-posedness).}
Given   $s \in \R$, 
\begin{align*}
\frac 13  < \g_2 \le \frac 12 < \g_1  < 1
\quad \text{such that}  \quad \g_1 + \g_2 > 1, 
%\label{REG0}  
\end{align*}

\noi
an integer $k \ge 2$, and $T > 0$, 
let     $\XX \in \cX^{s, \g_1}_k([0, T]\times \M)$
and 
$(\YY, \YYb) \in \cX^{s, \g_2}_1([0, T]\times \M)
\times \cX^{s, 2\g_2}_1([0, T]\times \M)$
be a $\g_2$-H\"older rough path
as in Definition \ref{DEF:control}.
Then, 
 the nonlinear Young differential equation~\eqref{RDE1} with 
 a rough perturbation
is locally well-posed in $H^s(\M)$.
More precisely, given $u_0 \in H^s(\M)$, 
there exists a unique solution $(\uu, \uu)
\in 
\D^{\g_2, \g_1 \wedge (2\g_2)}_{\YY, \tau}(H^s(\M))$
to~\eqref{RDE1}
with $\uu|_{t = 0} = u_0$, 
where 
the local existence time 
\[\tau = 
\tau\Big(\|u_0\|_{H^s(\M)}, \|\XX\|_{\cX^{s, \g_1}_k(T)},
\|\YY\|_{\cX^{s, \g_2}_1(T)},  \|\YYb\|_{\cX^{s, 2\g_2}_1(T)}\Big) \in (0, T\wedge 1]\]

\noi
is non-increasing in each of the arguments.
Here, 
$\D^{\g_2, \g_1 \wedge (2\g_2)}_{\YY, \tau}(H^s(\M))$
is as in Definition~\ref{DEF:control}.
In particular, we have
 $\uu \in \cC^{\g_2} ([0, \tau]; H^{s}(\M))$, 
 and 
the blowup alternative~\eqref{BA1} holds
for~\eqref{RDE1}.

\smallskip

\noi
\textup{(ii) (persistence of regularity).}
In addition, suppose that $u_0 \in H^{s_0}(\M)$ for some $s_0 > s$
and that  $\XX \in \cY^{s, s_0, \g_1}_k([0, T]\times \M)$
and $(\YY, \YYb)  \in \cX^{s, s_0, \g_2}_1([0, T]\times \M)
\times \cX^{s, s_0, 2\g_2}_1([0, T]\times \M)$, 
where $\cY^{s, s_0, \g}_k([0, T]\times \M)$  
and 
 $\cX^{s, s_0, \g}_1([0, T]\times \M)$  
 are as in \eqref{X2c}
and 
\eqref{X1z1}.
Then, 
by possibly making $\tau > 0$ smaller by a multiplicative constant 
\textup{(}in particular $\tau > 0$
depends on the $H^s$-norm of the initial data $u_0$
but not on its $H^{s_0}$-norm\textup{)}, 
we have $\uu \in \cC^{\g_2} ([0, \tau]; H^{s_0}(\M))$, 
 where $\uu$ is the solution to \eqref{RDE1}
 constructed in Part (i).

\smallskip

\noi
\textup{(iii) (nonlinear smoothing)}
In addition to the hypotheses in Part (i), 
suppose  that  $\XX \in \cX^{s, s_0, \g_1}_k([0, T]\times \M)$
and $(\YY, \YYb)  \in \cX^{s, s_0, \g_2}_1([0, T]\times \M)
\times \cX^{s, s_0, 2\g_2}_1([0, T]\times \M)$
for some $s_0 > s$, 
where $\cX^{s, s_0, \g}_k([0, T]\times \M)$  is  as in \eqref{X1x}.
Then, we have
 $\uu - u_0 \in \cC^{\g_2}([0, \tau]; H^{s_0}(\M))$, 
 where $\uu$ is the solution to \eqref{RDE1}
 constructed in Part (i).

\smallskip

\noi
\textup{(iv) (convergence)}
In addition to the hypotheses in Part (i), 
suppose  that a sequence $\{\XX^N\}_{N \in \N}\subset 
 \cX^{s, \g_1}_k([0, T]\times \M)$
\textup{(}and 
a sequence $\{(\YY^N, \YYb^N)\}_{N \in \N}\subset 
 \cX^{s, \g_2}_1([0, T]\times \M)
 \times  \cX^{s,2 \g_2}_1([0, T]\times \M) $
 of $\g_2$-H\"older rough paths\textup{)}
 converges to $\XX$ in 
 $ \cX^{s, \g_1}_k([0, T]\times \M)$
\textup{(}and 
to  $(\YY, \YYb)$ in 
 $ \cX^{s, \g_2}_1([0, T]\times \M)\times \cX^{s, 2\g_2}_1([0, T]\times \M)$, respectively\textup{)}
  as $N \to \infty$.
Given $N \in \N$, let 
 $(\uu^N, \uu^N)
\in 
\D^{\g_2, \g_1 \wedge (2\g_2)}_{\YY^N, \tau}(H^s(\M))$
be a unique solution
to the following nonlinear Young differential equation
with a rough 
perturbation\textup{:}
\begin{equation}
\label{RDE1x}
\uu^N = u_0 + \I^{\XX^N}(\uu^N)
+ \I^{\YY^N, \YYb^N}(\uu^N)
\end{equation}

\noi
constructed in Part (i).  
Then, 
by possibly making $\tau > 0$ smaller by a multiplicative constant,  
$\uu^N$  
converges to the solution $\uu$ to \eqref{RDE1}
constructed in Part (i)  
 in $\CC^{\g_2}([0, \tau]; H^s(\M))$
as $N \to \infty$.
Moreover, 
given any $K > 0$, 
the rate of convergence of $\uu^N$ to $\uu$ is
uniform in $u_0 \in B_K$,
where
$B_K$ denotes the ball 
in $H^s(\M)$
of radius $K > 0$ centered at the origin.

\end{proposition}

\begin{proof}

(i)
Let $0 < \al < \be : = \g_1\wedge (2\g_2) < 1$ such that 
\begin{align}
\al < \g_2, \quad 
 \g_1+\al > 1, 
\quad \text{and}\quad 
 (\g_2 + \be)\wedge (2\g_2 + \al) > 1, 
\label{REG1}
\end{align}

\noi
where the last two conditions are from Lemmas \ref{LEM:int1}
and \ref{LEM:int2}, respectively.
In particular, 
we need $\g_2 > \frac 13$ and $\g_1 + \g_2 > 1$.
	
Fix small $0 < \tau \le T\wedge 1$ (to be chosen later)
and $(u_0, u_0') \in H^s(\M) \times H^s(\M)$
(which needs to  satisfy a compatibility condition as we see later; see \eqref{REG2}).
Denote by $\D^{\al, \be, (u_0, u_0')}_{\YY, \tau}(H^s(\M))$
the subclass of 
$\D^{\al, \be}_{\YY, \tau}(H^s(\M))$
such that 
$(\uu, \uu')|_{t = 0} = (u_0, u_0')$.
Given $(\uu, \uu') \in \D^{\al, \be, (u_0, u_0')}_{\YY, \tau}(H^s(\M))$
satisfying 
\begin{align}
(\updl \uu)_{t, r} = \YY_{t, r}(\uu'_r) + R^\uu_{t, r}
\label{RDE1a}
\end{align}

\noi
for some remainder term 
$R^\uu = R^{\YY, \uu}\in   C^{\be}_{2, \tau}H^s(\M)$, 
define the  process 
 $z = z(\uu, \uu')$ by 
\begin{align}
z
  = \I^{\YY, \YYb}(\uu)
+ \big(u_0 + \I^\XX(\uu)\big), 
\label{RDE2}
\end{align}

\noi
where
$\I^{\YY, \YYb}(\uu)$ is the rough integral of $\uu$
with the rough driver $(\YY, \YYb)$ constructed in Lemma~\ref{LEM:int2}
and   $\I^\XX(\uu)$ is the nonlinear Young integral of $\uu$
with the $k$-linear Young driver $\XX$ constructed in Lemma \ref{LEM:int1}.
Then, 
noting that the Gubinelli derivative of $\I^{\YY, \YYb}(\uu)$
(with respect to $\YY$)
is given by $\uu$, 
it follows from 
 the current regularity assumption
 that $z$ is controlled by $\YY$ with the Gubinelli derivative 
 \begin{align}
 z' = \uu.
 \label{RDE2a}
 \end{align}
 More precisely, we have
\begin{align}
(\updl z)_{t, r}
= \YY_{t, r} (\uu_r) + R^z_{t, r}, \quad (t, r) \in \Dl_{2, \tau}
\label{RDE3} 
\end{align}

\noi
for some remainder term $R^z = R^{\YY, z}$.
Here, 
from \eqref{RDE2}, 
\eqref{RQ6},  \eqref{RQQ1}
(for $(\YY, \YYb)$ instead of $(\XX, \XXb)$), 
\eqref{YQ5}, and \eqref{Ja1a}, 
we see that the remainder term
 $R^z$ is given by 
\begin{align}
\begin{split}
R^z_{t, r} 
& = (\updl \I^{\YY, \YYb}(\uu))_{t, r}
+ (\updl\I^\XX(\uu))_{t, r} - \YY_{t, r} (\uu_r)\\
& = 
\YYb_{t, r} (\uu_r')
- \big(\Ld \updl
\Taa^{\YY, \YYb}(\uu, \uu')\big)_{t, r}
+ \big((\Id - \Ld \updl) \Ta^\XX(\uu)\big)_{t, r},
\end{split}
\label{RDE4}
\end{align}

\noi
where
$\Taa^{\YY, \YYb}(\uu, \uu')$ and 
$\Ta^\XX(\uu)$ are defined by 
\begin{align}
\begin{split}
\Taa^{\YY, \YYb}(\uu, \uu')_{t, r}
& = \YY_{t,r}(\uu_r) + \YYb_{t,r} (\uu'_r), \\
\Ta^\XX(\uu)_{t, r} & = \XX_{t,r}(\uu_r). 
\end{split}
\label{RDE5}
\end{align}

\noi
We now 
define the map $\G = \G_{\XX, \YY, \YYb, u_0}$
on $ \D^{\al,\be, (u_0, u_0')}_{\YY, \tau}(H^s(\M))$
by setting
%Given $(\uu, \uu') \in \D^{\al,\be}_{\XX, T}(V)$, 
\begin{align}
\G(\uu, \uu')
& = (z, z')
= (z, \uu), 
\label{RDE6}
\end{align}

\noi
where $z$ is as in \eqref{RDE2}
(and $R^z$ is as in \eqref{RDE4}).
From \eqref{RDE2} and \eqref{RDE2a}, we have 
\begin{align}
(z, z')|_{t = 0} = (u_0, u_0).
\label{RDE6a}
\end{align}

\noi
In the following, we show that, 
 by choosing $\tau > 0$ sufficiently small, 
$\G$ is a contraction on 
the ball $B_K
\subset \D^{\al,\be, (u_0, u_0')}_{\YY, \tau}(H^s(\M))$
 of radius $K > 0$ centered at the origin.
In particular, 
$\G$ maps 
$ \D^{\al,\be, (u_0, u_0')}_{\YY, \tau}(H^s(\M))$
into itself, which imposes
the following 
 compatibility condition:
\begin{align}
u_0' = u_0
\label{REG2}
\end{align}

\noi
in view of \eqref{RDE6} and \eqref{RDE6a}.

\medskip

\noi
$\bul$ {\bf Step 1:}
We first show boundedness
of $\G$. 
In view of \eqref{RQ1a}, 
we need to estimate the $C^\al_\tau H^s_x$-norm of $z' = \uu$
and the $C^\be_{2, \tau} H^s_x$-norm of $R^z$.
We have 
\begin{align}
\| \uu'\|_{L^\infty_\tau H^s_x} 
& \le \|u_0'\|_{H^s} 
+ \tau^\al \|\uu '\|_{C^\al_\tau H^s_x}.
\label{RDE7}
\end{align}

\noi
Then, from  the controlled structure \eqref{RDE1a} of $\uu$ with \eqref{RDE7}, 
we obtain
\begin{align}
\begin{split}
\| & z'\|_{C^\al_\tau H^s_x}
 = \| \uu\|_{C^\al_\tau H^s_x}\\
&  \le \tau^{\g_2 - \al}\| \YY\|_{\cX^{s, \g_2}_1(T)} \|\uu'\|_{L^\infty_\tau H^s_x}
+ \tau^{\be - \al}\| R^\uu\|_{C^\be_{2, \tau } H^s_x}\\
& \le  \| \YY\|_{\cX^{s, \g_2}_1(T)}\|u_0'\|_{H^s} 
+ \tau^{\g_2} \| \YY\|_{\cX^{s, \g_2}_1(T)}\|\uu '\|_{C^\al_\tau H^s_x}
+ \tau^{\be - \al}\| R^\uu\|_{C^\be_{2, \tau } H^s_x}, 
\end{split}
\label{RDE8}
\end{align}

\noi
since  $\g_2 - \al > 0$
(recall that $0 < \tau \le 1$).

Next, we estimate the $C^\be_{2, \tau}$-norm of $R^z$
in \eqref{RDE4}.
Proceeding as in 
\eqref{RQ4} and \eqref{YQ3}, we have 
\begin{align}
\begin{split}
(\updl \Taa^{\YY, \YYb}(\uu, \uu')\big)_{t_1, t_2, t_3}
& = - \YY_{t_1,t_2} R^\uu_{t_2,t_3} - \YYb_{t_1,t_2} (\updl \uu')_{t_2,t_3}, \\
 \big(\updl  \Ta^{\XX}(\uu) )_{t_1, t_2, t_3}
& = 
-  \XX_{t_1, t_2}(\uu_{t_2})
+   \XX_{t_1, t_2}( \uu_{t_3})
\end{split}
\label{RDE9}
\end{align}

\noi
for any $(t_1, t_2, t_3) \in \Dl_{3, \tau}$.
From \eqref{RDE5},    \eqref{Ja0},    \eqref{X2}, and \eqref{Ho3}, 
we have
\begin{align}
\begin{split}
\| \Ta^\XX(\uu)_{t, r} \|_{H^s}
& = \| \XX_{t, r}(\uu_r)  - \XX_{t, r}(0) \|_{H^s}\\
& \le \|\XX\|_{\cX^{s, \g_1}_k(T)}|t- r|^{\g_1}
\big( 1+ \|\uu_r\|_{H^s}\big)^{k-1} \|\uu_r\|_{H^s}
\end{split}
\label{RDE9aa}
\end{align}

\noi
for any $(t, r) \in \Dl_{2, \tau}$, 
from which we obtain
\begin{align}
\|\Ta^\XX(\uu) \|_{C^{\g_1}_{2, \tau}H^s_x}
& \le \|\XX\|_{\cX^{s, \g_1}_k(T)}
\big(1 + \|\uu\|_{L^\infty_\tau H^s_x}\big)^{k-1} \|\uu\|_{L^\infty_\tau H^s_x}.
\label{RDE9a}
\end{align}

\noi
Similarly, from \eqref{RDE9} with \eqref{X2}, we have 
\begin{align*}
\| (\updl \Ta^\XX(\uu))_{t_1,t_2,t_3}\|_{H^s}
& \leq \|\XX\|_{\cX^{s, \g_1}_k(T)}|t_1 - t_2|^{\g_1}
\\
&\quad \times
\big(1+\|\uu_{t_2}\|_{H^s} + \|\uu_{t_3}\|_{H^s}\big)^{k-1}
\|\uu_{t_2}-\uu_{t_3}\|_{H^s} 
\end{align*}

\noi
for any $(t_1, t_2, t_3) \in \Dl_{3, \tau}$, 
from which we obtain
\begin{align}
\|\updl \Ta^\XX(\uu) \|_{C^{\g_1+\al}_{3, \tau}H^s_x}
& \les \|\XX\|_{\cX^{s, \g_1}_k(T)}
\big(1 + \|\uu\|_{L^\infty_\tau H^s_x}\big)^{k-1} \|\uu\|_{C^\al _\tau H^s_x}.
\label{RDE9b}
\end{align}

\noi
Then, from \eqref{RDE4},  
\eqref{sew1},
\eqref{RDE9}, 
\eqref{RDE9a}, 
and \eqref{RDE9b}
with 
\eqref{Ho2}, 
we have
\begin{align}
\begin{split}
\| R^z\|_{C^\be_{2, \tau}H^s_x}
& \le \tau^{2\g_2 - \be} \|\YYb\|_{\cX^{s, 2\g_2}_1(T)}\|\uu'\|_{L^\infty_TH^s_x}\\
& \quad + C_0\tau^{\g_2}
\|\YY\|_{\cX^{s, \g_2}_1(T)}
\| R^\uu\|_{C^\be_{2, \tau } H^s_x}\\
& \quad  + C_0\tau^{2\g_2 + \al - \be}
\|\YYb\|_{\cX^{s, 2\g_2}_1(T)}\|\uu'\|_{C^\al_\tau H^s_x}\\
& \quad   + 
\tau^{\g_1 - \be}\|\XX\|_{\cX^{s, \g_1}_k(T)}
\big(1 + \|\uu\|_{L^\infty_\tau H^s_x}\big)^{k-1}
\|\uu\|_{L^\infty_\tau H^s_x}\\
& \quad + C_0
 \tau^{\g_1 + \al - \be} \|\XX\|_{\cX^{s, \g_1}_k(T)}
 \big(1 + \|\uu\|_{L^\infty_\tau H^s_x}\big)^{k-1}
 \|\uu\|_{C^\al_\tau H^s_x}, 
\end{split}
\label{RDE10}
\end{align}

\noi
where we used \eqref{REG1} in applying \eqref{sew1}.
Hence, 
by writing
\begin{align}
\| \uu \|_{L^\infty_\tau H^s_x} 
& \le \|u_0 \|_{H^s} + \tau^\al \| \uu\|_{C^\al_\tau H^s_x}, 
\label{RDE8a}
\end{align}

\noi
it follows 
from \eqref{RDE10} with \eqref{RDE7}, 
\eqref{RDE8}, and \eqref{RDE8a} that 
\begin{align}
\begin{split}
  \| R^z\|_{C^\be_{2, \tau}H^s_x}
&    \le 
%\tau^{\g_1 - \be}
C_1  \Big\{ A_{\XX, \YY, \YYb}(u_0, u_0')\\
& 
\hphantom{XXX}
+ \tau^{2\g_2 + \al - \be} 
\|\YYb\|_{\cX^{s, 2\g_2}_1(T)}
 \|\uu '\|_{C^\al_\tau H^s_x}\\
& 
\hphantom{XXX}
 +
 \tau^{\g_2}
\|\YY\|_{\cX^{s, \g_2}_1(T)}
\| R^\uu\|_{C^\be_{2, \tau } H^s_x}\\
& 
\hphantom{XXX}
 + 
 \tau^{\g_1 +k \g_2 + k \al - \be}\|\XX\|_{\cX^{s, \g_1}_k(T)}
 \| \YY\|_{\cX^{s, \g_2}_1(T)}^k 
  \|\uu '\|_{C^\al_\tau H^s_x}^k\\
& 
\hphantom{XXX}
 + 
 \tau^{\g_1  +(k-1) \be}\|\XX\|_{\cX^{s, \g_1}_k(T)}
\| R^\uu\|_{C^\be_{2, \tau } H^s_x}^k\Big\}, 
%\tau^{2\al - \be} \|\YYb\|_{\cX^{s, 2\g_2}_1(T)}\|\uu'\|_{L^\infty_T}
\end{split}
\label{RDE11}
\end{align}

\noi
where $A_{\XX, \YY, \YYb}(u_0, u_0')$ is defined by 
\begin{align}
\begin{split}
A_{\XX, \YY, \YYb}(u_0, u_0')
& = \|\XX\|_{\cX^{s, \g_1}_k(T)}\big(1+ \|u_0\|_{H^s}\big)^{k}
+ \|\YYb\|_{\cX^{s, 2\g_2}_1(T)}\|u_0'\|_{H^s} \\
& \quad 
+ \|\XX\|_{\cX^{s, \g_1}_k(T)}
\|\YY\|_{\cX^{s, \g_2}_1(T)}^k
\big(1 + \|u_0'\|_{H^s}\big)^k.
\end{split}
\label{RDE12}
\end{align}

\noi
Here, we used the fact that 
$\be = \g_1 \wedge (2\g_2)$
and $0 < \tau \le 1$ in treating the terms involving
initial data.
Then, by setting
\begin{align}
K = 1\vee 
2\Big( 2 \|u_0\|_{H^s} + \| \YY\|_{\cX^{s, \g_2}_1(T)}\|u_0'\|_{H^s}  + C_1  A_{\XX, \YY, \YYb}(u_0, u_0')\Big), 
\label{RDE13}
\end{align}

\noi
it follows from \eqref{RDE6}, 
\eqref{RQ1a}, 
\eqref{RDE1a}, 
\eqref{RDE3}, 
 \eqref{RDE6a},
\eqref{RDE8}, 
\eqref{RDE11},
and \eqref{RDE12}
with 
 $0 < \al < \be = \g_1\wedge (2\g_2)$
that by 
 choosing 
\begin{align}
\begin{split}
\tau 
& = 
\tau\Big(K, \|\XX\|_{\cX^{s, \g_1}_k(T)}, 
 \| \YY\|_{\cX^{s, \g_2}_1(T)},
  \| \YYb\|_{\cX^{s, 2\g_2}_1(T)}\Big) \\
& = 
\tau\Big(\|u_0\|_{H^s}, 
\|u_0'\|_{H^s}, 
\|\XX\|_{\cX^{s, \g_1}_k(T)}, 
 \| \YY\|_{\cX^{s, \g_2}_1(T)},
  \| \YYb\|_{\cX^{s, 2\g_2}_1(T)}\Big) > 0
\end{split}
\label{RDE14}
\end{align}
sufficiently small, 
the map $\G$ defined in \eqref{RDE6} maps the ball $B_K\subset 
\D^{\al, \be, (u_0, u_0')}_{\YY, \tau}(H^s(\M))$
of radius $K > 0$ 
centered at the origin into itself.

\medskip

\noi
$\bul$ {\bf Step 2:}
Next, we show that $\G$ is a contraction
on the ball $B_K$.
Fix $\big(\uu^1, (\uu^1)'\big),
\, \big(\uu^2, (\uu^2)'\big)
\in B_K\subset 
\D^{\al, \be, (u_0, u_0')}_{\YY, \tau}(H^s(\M))$, 
satisfying
\begin{align}
 (\updl \uu^j)_{t, r} & = \YY_{t, r}((\uu^j)'_r) + R^{\uu^j}_{t, r}
\label{RDD1} 
\end{align}

\noi
for some remainder term
$R^{\uu^j} = R^{\YY,  \uu^j} \in   C^{\be}_{2, T} H^s(\M)$, $j = 1, 2$.
For simplicity of notation, 
we set $\vv^j = (\uu^j)'$, $j = 1, 2$, in the following.
%First, note from \eqref{RDE8a} that 
%\begin{align}
%\| \uu^j\|_{L^\infty_\tau H^s_x} \les K.
%\label{RDD1a}
%\end{align}

As in \eqref{RDE2}, \eqref{RDE3}, and \eqref{RDE4}, 
we define 
 $z^j = z^j(\uu^j, \vv^j)$ by 
\begin{align*}
z^j
  = \I^{\YY, \YYb}(\uu^j)
+ \big(u_0 + \I^\XX(\uu^j)\big)
%\label{RDD2}
\end{align*}

\noi
 with the Gubinelli derivative $(z^j)' = \uu^j$, 
 satisfying
\begin{align}
(\updl z^j)_{t, r}
= \YY_{t, r} (\uu^j_r) + R^{z^j}_{t, r}
\label{RDD3} 
\end{align}

\noi
for $0 \le r < t \le \tau$.
Here, 
 the remainder term
 $R^{z^j} = R^{\YY, z^j}$ is given by 
\begin{align}
\begin{split}
R^{z^j}_{t, r} 
& = 
\YYb_{t, r} (\vv^j_r)
- \big(\Ld \updl
\Taa^{\YY, \YYb}(\uu^j, \vv^j)\big)_{t, r}
+ \big((\Id - \Ld \updl) \Ta^\XX(\uu^j)\big)_{t, r},
\end{split}
\label{RDD4}
\end{align}

\noi
where
$\Taa^{\YY, \YYb}(\uu^j, \vv^j)$ and 
$\Ta^\XX(\uu^j)$ are as in \eqref{RDE5}.
Then, from \eqref{RDE6}, we have
\begin{align}
\G(\uu^j, \vv^j) = \big(z^j, (z^j)'\big)
= (z^j, \uu^j).
\label{RDD4a}
\end{align}

\noi
In Step 1, we showed that $\G$ maps $B_K$ into itself.
Hence, from 
\eqref{RDE8a}, 
\eqref{RQ1a}, and 
 \eqref{RDD4a}, 
we have 
\begin{align}
\| \uu^j\|_{L^\infty_\tau H^s_x} \les K.
\label{RDD1a}
\end{align}

Proceeding as in \eqref{RDE8} with \eqref{RDD1}  and \eqref{RDE7}
(but for $\vv^1 - \vv^2= (\uu^1)' - (\uu^2)' $ instead of $\uu'$), we have 
\begin{align}
\begin{split}
\| & (z^1)' - (z^2)'\|_{C^\al_\tau H^s_x}
 = \| \uu^1 - \uu^2\|_{C^\al_\tau H^s_x}\\
&  \le \tau^{\g_2 - \al}\| \YY\|_{\cX^{s, \g_2}_1(T)} \|\vv^1 - \vv^2\|_{L^\infty_\tau H^s_x}
+ \tau^{\be - \al}\|  R^{\uu^1} - R^{\uu^2}\|_{C^\be_{2, \tau } H^s_x}\\
& \le  %\| \YY\|_{\cX^{s, \g_2}_1(T)}\|u_0'\|_{H^s} 
 \tau^{\g_2} \| \YY\|_{\cX^{s, \g_2}_1(T)}\|\vv^1 - \vv^2\|_{C^\al_\tau H^s_x}
+ \tau^{\be - \al}\| R^{\uu^1} - R^{\uu^2}\|_{C^\be_{2, \tau } H^s_x}, 
\end{split}
\label{RDD8}
\end{align}

\noi
where we used the fact that $\vv^1(0) = \vv^2(0) = u_0'$
at the last step.

Next, we estimate $R^{z^1 -z^2} = R^{z^1}  - R^{z^2}$, 
where the equality follows form \eqref{RDD3}.
Proceeding as in \eqref{RDE9aa}, 
we have
\begin{align*}
& \| \Ta^\XX(\uu^1)_{t, r}
- \Ta^\XX(\uu^2)_{t, r}
 \|_{H^s}
 = \| \XX_{t, r}(\uu^1_r)  - \XX_{t, r}(\uu^2_r) \|_{H^s}\\
& \quad \le \|\XX\|_{\cX^{s, \g_1}_k(T)}|t- r|^{\g_1}
\big( 1+ \|\uu^1_r\|_{H^s} + \|\uu^2_r\|_{H^s}\big)^{k-1} \|\uu^1_r - \uu^2_r\|_{H^s}
\end{align*}

\noi
for any $(t, r) \in \Dl_{2, \tau}$.
With \eqref{RDD1a}, this  yields
\begin{align}
\|\Ta^\XX(\uu^1) - \Ta^\XX(\uu^2) \|_{C^{\g_1}_{2, \tau}H^s_x}
& \les \|\XX\|_{\cX^{s, \g_1}_k(T)}
K^{k-1} \|\uu^1 - \uu^2\|_{L^\infty_\tau H^s_x}.
\label{RDD9a}
\end{align}

\noi
We also recall from \cite[(3.63)]{CGLLO1}
that 
\begin{align}
\begin{split}
\|\updl \Ta^\XX(\uu^1)
- \updl \Ta^\XX(\uu^2) \|_{C^{\g_1+\al}_{3, \tau}H^s_x}
& \les \|\XX\|_{\cX^{s, \g_1}_k(T)}
K^{k-1} \|\uu^1 - \uu^2\|_{\CC^\al _\tau H^s_x}\\
& \les \|\XX\|_{\cX^{s, \g_1}_k(T)}
K^{k-1} \|\uu^1 - \uu^2\|_{C^\al _\tau H^s_x}, 
\end{split}
\label{RDD9b}
\end{align}

\noi
where we used the fact that $\uu^1(0)  = \uu^2(0)  = u_0$ 
and $0 < \tau \le 1$
 at the second inequality.
Hence, from \eqref{RDD4}, 
\eqref{sew1},
\eqref{RDE9}, 
\eqref{RDD9a}, 
and \eqref{RDD9b}, 
we have
\begin{align}
\begin{split}
\| R^{z^1} - R^{z^2}\|_{C^\be_{2, \tau}H^s_x}
& \les \tau^{2\g_2 - \be} \|\YYb\|_{\cX^{s, 2\g_2}_1(T)}\|\vv^1 - \vv^2\|_{L^\infty_\tau H^s_x}\\
& \quad + \tau^{\g_2}
\|\YY\|_{\cX^{s, \g_2}_1(T)}
\| R^{\uu^1} - R^{\uu^2}\|_{C^\be_{2, \tau } H^s_x}\\
& \quad  +\tau^{2\g_2 + \al - \be}
\|\YYb\|_{\cX^{s, 2\g_2}_1(T)}\|\vv^1 - \vv^2\|_{C^\al_\tau H^s_x}\\
& \quad   + 
\tau^{\g_1 - \be}\|\XX\|_{\cX^{s, \g_1}_k(T)}
K^{k-1}
\|\uu^1 - \uu^2\|_{L^\infty_\tau H^s_x}\\
& \quad + 
 \tau^{\g_1 + \al - \be} \|\XX\|_{\cX^{s, \g_1}_k(T)}
K^{k-1}
 \|\uu^1 - \uu^2\|_{C^\al_\tau H^s_x}
\end{split}
\label{RDD10}
\end{align}

\noi
for any $(\uu^1, \vv^1), (\uu^2, \vv^2) \in B_K$, 
satisfying \eqref{RDD1}.
Finally, 
from \eqref{RDD10} with \eqref{RDE7}, 
 \eqref{RDE8a}, and \eqref{RDD8}, we obtain
\begin{align}
\begin{split}
\| R^{z^1} - R^{z^2}\|_{C^\be_{2, \tau}H^s_x}
&    \les 
\tau^{2\g_2 + \al - \be} \|\YYb\|_{\cX^{s, 2\g_2}_1(T)}\|\vv^1 - \vv^2\|_{C^\al_\tau H^s_x}\\
& \quad + \tau^{\g_2}
\|\YY\|_{\cX^{s, \g_2}_1(T)}
\| R^{\uu^1} - R^{\uu^2}\|_{C^\be_{2, \tau } H^s_x}\\
& \quad 
+ 
 \tau^{\g_1 + \g_2 + \al - \be} \|\XX\|_{\cX^{s, \g_1}_k(T)}
\|\YY\|_{\cX^{s, \g_2}_1(T)}
K^{k-1}
 \|\vv^1 - \vv^2\|_{C^\al_\tau H^s_x}\\
& \quad 
+ 
 \tau^{\g_1} \|\XX\|_{\cX^{s, \g_1}_k(T)}
K^{k-1}
\| R^{\uu^1} - R^{\uu^2}\|_{C^\be_{2, \tau } H^s_x}.
\end{split}
\label{RDD11}
\end{align}

\noi
Therefore, by 
 choosing 
$\tau > 0$ of the form \eqref{RDE14}
sufficiently small, 
we conclude from~\eqref{RDD8} and \eqref{RDD11} with
\eqref{RDD4a} and 
 $0 < \al < \be = \g_1\wedge (2\g_2)$
that $\G$ is a contraction on $B_K
\subset \D^{\al, \be, (u_0, u_0')}_{\YY, \tau}(H^s(\M))$.

By the Banach fixed point theorem with  \eqref{RDE6}, 
there exists a unique fixed point $(\uu, \uu') \in B_K
\subset \D^{\al, \be, (u_0, u_0')}_{\YY, \tau}(H^s(\M))$
such that 
$(\uu, \uu')  = \G(\uu, \uu') = (z, \uu)$, 
where $z$ is as in 
\eqref{RDE2}.
In particular, $\uu$ satisfies
\eqref{RDE1}.
Furthermore,  we have $\uu' = \uu$, 
which is consistent with the compatibility condition
\eqref{REG2}.
Here, the uniqueness a priori holds only in the ball $B_K$
but, by a standard continuity argument, 
we can extend uniqueness to the entire space 
$\D^{\al, \be, (u_0, u_0)}_{\YY, \tau}(H^s(\M))$
(by possibly making $\tau$ smaller by a multiplicative constant).

Lastly, we note that from Lemmas \ref{LEM:int1}
and \ref{LEM:int2}, 
 the right-hand side of~\eqref{RDE1}
belongs to $\CC^{\g_2}([0, \tau]; H^s(\M))$, 
from which we conclude that 
the solution $\uu$ to \eqref{RDE1} 
constructed above belongs to $ \CC^{\g_2}([0, \tau]; H^s(\M))$.

\medskip

\noi
(ii) 
The claim follows from 
Lemma \ref{LEM:int1}\,(iii)
and Lemma \ref{LEM:int2}\,(iii).

\medskip

\noi
(iii) 
The claim follows from 
Lemma \ref{LEM:int1}\,(iv)
and Lemma \ref{LEM:int2}\,(iii).
%\eqref{RDE8} and \eqref{RDE11}

\medskip

\noi
(iv) 
Without loss of generality, assume that 
\begin{align}
\begin{split}
\sup_{N \in \N} \|\XX^N\|_{\cX^{s, \g_1}_k(T)} & \le \|\XX\|_{\cX^{s, \g_1}_k(T)} + \eta
\le L, \\
\sup_{N \in \N}  \| \YY^N\|_{\cX^{s, \g_2}_1(T)}
&  \le  \| \YY\|_{\cX^{s, \g_2}_1(T)} + \eta
\le L,\\
\sup_{N \in \N}   \| \YYb^N\|_{\cX^{s, 2\g_2}_1(T)} & \le   \| \YYb\|_{\cX^{s, 2\g_2}_1(T)}+ \eta
\le L
\end{split}
\label{RDD12x}
\end{align}

\noi
for some $L\ge 0$ and small $\eta > 0$  such that 
we have
$(\uu^N, \uu^N)
\in B_{K+1} \subset \D^{\al, \be, (u_0, u_0)}_{\YY, \tau}(H^s(\M))$
for any $N \in \N$, 
where $K$ is as in \eqref{RDE13}.
Defining $R^{\YY^N, \uu^N}$ by 
\begin{align*}
(\updl \uu^N)_{t, r} = \YY^N_{t, r}((\uu^N)'_r) + R^{\YY^N, \uu^N}_{t, r}
%\label{RDD12y}
\end{align*}

\noi
for $(t, r) \in \Dl_{2, \tau}$, 
it follows from 
\eqref{RQ1a} and the fact that $(\uu^N)' = \uu^N$
that 
\begin{align}
\| u_0 \|_{H^s} + \big\|\uu^N\big\|_{C^\al_\tau H^s_x}
+  \| R^{\YY^N, \uu^N}\|_{C^\be_{2,\tau}H^s_x}\les K , 
\label{RDD12yy}
\end{align}

\noi
uniformly in $N \in \N \cup \{\infty\}$
with the understanding that $\uu^\infty
= \uu$ and $\YY^\infty = \YY$, 
where 
  $R^{\YY, \uu}$ is as in \eqref{RDE1a}.
Then, from \eqref{RDE8a} and   \eqref{RDD12yy} with 
the fact that $0 < \tau \le 1$, we have 
\begin{align}
\|\uu^N\|_{L^\infty_\tau H^s_x} \les 
 K, 
\label{RDD12z}
\end{align}

\noi
uniformly in $N \in \N \cup \{\infty\}$.

From \eqref{RDE1} and \eqref{RDE1x}, we have 
\begin{equation}
\uu - \uu^N = \big(\I^\XX(\uu)- \I^{\XX^N}(\uu^N) \big)+ 
\big(\I^{\YY, \YYb}(\uu) -\I^{\YY^N, \YYb^N}(\uu^N)\big).
\label{RDD12}
\end{equation}

\noi
From 
 \cite[(3.64)]{CGLLO1}
and \eqref{Jaa4x} in 
Lemma \ref{LEM:int1} with \eqref{RDE8a}, 
  $\uu(0) = \uu^N(0) = u_0$, 
the fact that $0 < \tau \le 1$, 
\eqref{RDD12x}, 
and \eqref{RDD12z}, 
we have
\begin{align}
\begin{split}
& \| \I^\XX(\uu)- \I^{\XX^N}(\uu^N)\|_{\CC^\al_\tau H^s_x}
%\les \| \I^\XX(\uu)- \I^{\XX^N}(\uu^N)\|_{C^\al_\tau H^s_x}
\\
&  \quad \le 
\| \I^\XX(\uu) - \I^\XX(\uu^N)\|_{\CC^\al_\tau H^s_x}
+ \| \I^\XX(\uu^N)- \I^{\XX^N}(\uu^N)\|_{\CC^\al_\tau H^s_x}\\
& \quad \le C_{K, L} \,
\tau^{\g_1- \al} 
\Big(
 \| \uu - \uu^N \|_{\CC^\al_\tau H^s_x}
+ 
 \|\XX - \XX^N\|_{\cX^{s, \g_1}_k(T)} \Big).
\end{split}
\label{RDD12a}
\end{align}

%\noi
%where we used the fact that $0 < \tau \le 1$ at  the first inequality.

Recalling that
 $\uu' = \uu$ and 
 $(\uu^N)' = \uu^N$, 
 it follows from 
 \eqref{RQQ7} in Lemma \ref{LEM:int2},  
 \eqref{RDE8a}, 
 \eqref{RDD12x}, 
\eqref{RDD12yy}, and 
\eqref{RDD12z}
 with $\uu(0) = \uu^N(0) = u_0$
 and the fact that $0 < \tau \le 1$
 that 
\begin{align}
\begin{split}
 \| & \I^{\YY, \YYb}(\uu)
- \I^{\YY^N, \YYb^N}(\uu^N)\|_{\CC^\al_\tau H^s_x}\\
& \le
C_{K, L}\Big(  \tau^{\g_2 - \al}
\|\YY- \YY^N\|_{\cX^{s, \g_2}_1(T)}
+ 
 \tau^{\g_2 - \al}
\|\uu -\uu^N\|_{\CC^\al_\tau H^s_x}\\
& \quad 
\hphantom{lXXX}
+ \tau^{2\g_2 - \al}
\|\YYb- \YYb^N\|_{\cX^{s, 2\g_2}_1(T)}
+ \tau^{2\g_2 - \al}
\|\uu -\uu^N\|_{\CC^\al_\tau H^s_x}\\
& \quad 
\hphantom{lXXX}
+   \tau^{\g_2 + \be- \al }
\|\YY - \YY^N\|_{\cX^{s, \g_2}_1(T)}
\\
& \quad 
\hphantom{lXXX}
+ 
\tau^{\g_2 + \be- \al }
\| R^{\YY, \uu} - R^{\YY^N, \uu^N}\|_{C^\be_{2,\tau}H^s_x}\Big).
\end{split}
\label{RDD12b}
\end{align}

\noi
Since $z = \uu = \uu'$, 
we have $\RR^{\YY, \uu} = \RR^{\YY, z}$.
Thus, using
\eqref{RDE4}
for $\RR^{\YY, \uu}$ (and for $\RR^{\YY^N, \uu^N}$ with appropriate modifications),
a slight modification of the computations leading to \eqref{RDE10} 
with~\eqref{RDE8a} 
yields
\begin{align}
\begin{split}
 \| & R^{\YY, \uu} - R^{\YY^N, \uu^N}\|_{C^\be_{2,\tau}H^s_x}\\
& 
\les 
C_{K, L} \Big( \tau^{2\g_2  - \be} \|\YYb- \YYb^N\|_{\cX^{s, 2\g_2}_1(T)}
+ 
\tau^{2\g_2 + \al - \be} 
\|\uu -\uu^N\|_{\CC^\al_\tau H^s_x}\\
& \quad 
\hphantom{lXXX}+ 
\tau^{\g_2}
\|\YY - \YY^N\|_{\cX^{s, \g_2}_1(T)}
+ \tau^{\g_2}
\| R^{\YY, \uu} - R^{\YY^N, \uu^N}\|_{C^\be_{2, \tau } H^s_x}\\
& \quad  \hphantom{lXXX} + 
\tau^{\g_1  - \be}\|\XX - \XX^N\|_{\cX^{s, \g_1}_k(T)}
+  \tau^{\g_1 +\al - \be}
\|\uu- \uu^N\|_{\CC^\al_\tau H^s_x}\Big).
\end{split}
\label{RDD12c}
\end{align}

Hence, putting \eqref{RDD12}, \eqref{RDD12a}, \eqref{RDD12b}, 
and \eqref{RDD12c}
together with \eqref{RDD12x}
and the assumed convergence of $\XX^N$ and $(\YY^N, \YYb^N)$, 
we conclude that, by possibly making $\tau$ smaller by a multiplicative constant, 
$\uu^N$ converges to $\uu$
 in $\CC^{\al}([0, \tau]; H^s(\M))$
 as $N \to \infty$.
Finally, 
by applying
Lemmas \ref{LEM:int1} and \ref{LEM:int2}, 
we see that 
$ \I^{\XX^N}(\uu^N)$
(and $\I^{\YY^N, \YYb^N}(\uu^N)$)
converges
to 
$ \I^{\XX}(\uu)$
(and to $\I^{\YY, \YYb}(\uu)$, respectively)
 in $\CC^{\g_2}([0, \tau]; H^s(\M))$
as $N \to \infty$, 
which yields the claimed convergence.
See the proof of Proposition 3.8\,(iv) in \cite{CGLLO1}
for details of an analogous argument.
\end{proof}

\section{Regularities of the drivers
for the stochastic modulated KdV}
\label{SEC:driver}

In this section, we 
establish Theorems \ref{THM:1} and \ref{THM:2}
for  SmKdV \eqref{skdv2}
with a multiplicative Young noise, 
and 
Theorems  \ref{THM:4} and  \ref{THM:5}
for SmKdV  \eqref{skdv3} with a multiplicative rough noise.
As mentioned in Section \ref{SEC:1}, 
these results follow
from 
Lemma \ref{LEM:int1}, 
Proposition \ref{PROP:main}, 
Lemma \ref{LEM:int2}, 
and Proposition \ref{PROP:main2}
(see also Remark \ref{REM:mean0})
once we establish 
the corresponding 
regularity properties
of the relevant drivers (Lemma \ref{LEM:kdv1}
and Propositions \ref{PROP:drive2} and \ref{PROP:drive3}).

In Subsection \ref{SUBSEC:drive1}, 
we recall the regularity properties
of the bilinear driver $\XX$ defined in~\eqref{K1x}.
In Subsection \ref{SUBSEC:drive2}, 
we consider the Young case ($\frac 12 < \be < 1$) 
and study regularity properties
of the Young driver $\YY$ in~\eqref{sto1x};
see
Proposition \ref{PROP:drive2}.
Finally, in Subsection \ref{SUBSEC:drive3}, 
we consider the rough case ($\be = \frac 12$)
and study the regularity properties
of the drivers $\wt \YY$ and $\wt \YYb$ in \eqref{sto2x}
and \eqref{sto3x}, respectively;
see Proposition \ref{PROP:drive3}.

\subsection{Bilinear  driver
for the modulated KdV}
\label{SUBSEC:drive1}

In this subsection, we recall 
from 
the regularity properties
of the bilinear driver $\XX$ in \eqref{K1x} for the modulated KdV \eqref{kdv1};
see \cite[Subsection~4.1]{CGLLO1}.
By taking the Fourier transform, we have
\begin{align*}
\F\big(\XX_{t,r} (f_1,f_2)\big) (n)
= in \sum_{ \substack{n_1, n_2 \in \Z_*\\n = n_1+n_2}}
 \Phi^w_{t,r}(\Xi_\KDV (\bar n))
\ft f_1(n_1)  \ft f_2(n_2),
\end{align*}

\noi
where $\Phi^w_{t,r}$ is as in \eqref{rho2}
and $\Xi_\KDV (\bar n)$ denotes the resonance function\footnote{Here, we follow
the terminology in \cite{Tao2}.  
We point out that $\Xi_\KDV (\bar n)$ is also called the modulation function (see,
for example,  \cite{KOY})
but we do not use this latter terminology to avoid
confusion with a modulation function $w(t)$.
}
 for KdV given by 
\begin{align}
\begin{split}
\Xi_\KDV (\bar n) &  = \Xi_\KDV (n,n_1,n_2) 
= - n^3+ n_1^3 +n_2^3\\
& = - 3n n_1 n_2, 
\end{split}
\label{K3}
\end{align}

\noi
where the last equality holds
 under $n = n_1+ n_2$.

The following lemma summarizes the 
basic mapping properties
of the bilinear driver $\XX$.
See 
\cite[Proposition 4.1]{CGLLO1}.
See Subsection \ref{SUBSEC:2.2}
for the definition of the function spaces below.
In particular, see Remark \ref{REM:m0}.

\begin{lemma}
\label{LEM:kdv1}

Given $\rho \ge\frac 12$, $\frac 12 < \g < 1$, and $T> 0$, 
let  $w$ be $(\rho,\g)$-irregular on $[0, T]$ in the sense of Definition~\ref{DEF:ir}.

\smallskip

\begin{itemize}
\item[(i)]
Suppose that $\rho \ge \frac 12$ and $s \in \R$
satisfy 
\eqref{reg1}.
Then, 
the driver $\XX$ defined in \eqref{K1x}
belongs to $ \cX^{s, \g}_{2, 0}([0, T]\times \T)$.

\smallskip
\item[(ii)] \textup{(persistence of regularity).}
Suppose that $\rho \ge \frac 12$ and $s \in \R$
satisfy 
\eqref{reg1}.
Then, 
for any $s_0 > s$, 
the driver $\XX$ 
belongs to 
$\cY^{s, s_0, \g}_{2, 0}([0, T]\times \T)$.

\smallskip
\item[(iii)]
\textup{(nonlinear smoothing).}
In addition, suppose that $s_0 > s$ satisfies
\eqref{reg3}.
Then, 
the driver $\XX$ 
belongs to 
$\cX^{s, s_0, \g}_{2, 0}([0, T]\times \T)$.

\smallskip
\item[(iv)]
\textup{(Galerkin approximation).}
Suppose that $\rho \ge \frac 12$ and $s \in \R$
satisfy
\begin{align}
\begin{split}
\textup{(iv.a)} &\ \  \tfrac 12 < \rho \le  \tfrac 34 \quad \text{and} \quad 
s > \tfrac 32 - 3 \rho, \\
\textup{(iv.b)} &\ \  \rho > \tfrac 34\quad  \text{and} \quad  s > - \rho.
\end{split}
\label{reg3a}
\end{align}

\noi
Then, 
the truncated driver $\XX^N$, defined by 
\begin{align*}
\XX_{t,r}^N(f_1,f_2)=\int_r^t \uw(t')^{-1} \P_N \partial_x\big(
( \P_N \uw(t')  f_1 )(\P_N \uw(t') f_2 )\big) dt', 
%\label{XN}
\end{align*}

\noi
 converges
to  $\XX$ 
in  $ \cX^{s, \g}_{2, 0}([0, T]\times \T)$ as $N \to \infty$.
Here, 
 $\P_N$
 denotes  the Dirichlet projector onto the (spatial) frequencies 
$\{|n| \leq N\}$.

\end{itemize}

 \end{lemma}

 \subsection{Young driver
 for the noise}
\label{SUBSEC:drive2}

In this subsection, 
we study regularity properties
of the Young 
driver $\YY$ in~\eqref{sto1x}.
From 
\eqref{W1}, \eqref{phi1}, and \eqref{phi2}, we have 
\begin{align}
\YY_{t, r}(f)
&
 = \sum_{n\in \Z_*} e_n 
 \int_r^t \sum_{n_1, n_2  \in \Z} \ind_{\substack{n = n_{12}\\n_1\ne 0}}
\cdot  e^{- i w(t') (n^3 - n_2^3)}\ft f(n_2) \phi_{n_1}    d B_{n_1}(t'), 
\label{stoconv1}
\end{align}

\noi
where
 $n_{12} = n_1 + n_2$.

\begin{proposition}
\label{PROP:drive2}

Given 
$\rho \ge\frac 12$,
$\frac 12 < \g, \be < 1$, $ \s \in \R$, and $T > 0$, 
let 
$w$ be 
 a $(\rho,\g)$-irregular function on $[0, T]$ in the sense of Definition~\ref{DEF:ir}
 and 
$\phi \in \HS(L^2(\T); H^\s(\T))$, satisfying \eqref{phi1} and \eqref{phi2},
and 
let  $\YY = \YY^{w, \phi}$ be 
the random operator 
 defined in \eqref{stoconv1}.
Suppose that 
$  \g_0 >0$ and $s, s_0 \in \R$ satisfy 
\eqref{YT1}, \eqref{YT2}, and \eqref{YT3}. 
 Then, given any finite $p \ge 1$, we have 
\begin{align}
\big\| \|  \YY_{t,r}  \|_{\LOP(H^s; H^{s_0})} \big\|_{L^p(\Om)} 
 & \les_T  p^\frac 12 
\jb{\|\Phi^w\|_{  \W^{\rho,\g}_T} }
\| \phi\|_{\HS(L^2;H^\s)} 
  |t-r|^{\g_0}
\label{YG1}
\end{align}

\noi
for any   $(t,r) \in \Dl_{2,T}$.
Consequently, we have 
\begin{align}
\big\| \| \YY\|_{C^{\g_0}_{2,T} \LOP(H^s;H^{s_0})} \big\|_{L^p(\Om)} 
\les_T 
p^\frac 12 
\jb{\|\Phi^w\|_{  \W^{\rho,\g}_T} }
 \| \phi\|_{\HS(L^2;H^\s)} .
\label{YG1a}
\end{align}

\noi
In particular, 
there exists a version of $\YY$
such that 
$\YY \in C^{\g_0}_{2,T} \LOP(H^s(\T);H^{s_0}_0(\T))$
and
\begin{align}
\YY \in  \cX^{s, s_0, \g_0}_{1, 0} ([0, T]\times \T) ,
\label{YG1b}
\end{align}

\noi
almost surely.
See Subsection \ref{SUBSEC:2.2}
for the definitions of these function spaces.

\end{proposition}

As mentioned in Subsection \ref{SUBSEC:main1}, 
Theorem \ref{THM:1}
on the construction of the stochastic term~$\PPsi(\uu)$
 in \eqref{psi2}
 as the (linear) Young integral $\I^\YY(\uu)$
follows from 
Lemma \ref{LEM:int1}
and 
Proposition \ref{PROP:drive2}
while 
Theorems \ref{THM:2} 
on local well-posedness of 
 SmKdV \eqref{skdv2}
with a multiplicative Young noise
 follows from 
Proposition \ref{PROP:main}
with 
Lemma \ref{LEM:kdv1} and 
Proposition \ref{PROP:drive2}.
We omit details.

Before proceeding to a proof of 
Proposition \ref{PROP:drive2}, 
we state a remark.
See also Remark~\ref{REM:BO} below
for a discussion
on the Young drivers associated with 
the stochastic modulated BO~\eqref{sbo1}
and 
the stochastic modulated ILW \eqref{SILW1}.

\begin{remark}\label{REM:YG1}
\rm
(i)
We point out that it is crucial to have $\phi_0 = 0$, 
namely $n_1 \ne 0$
in \eqref{stoconv1}.
See~\eqref{YG3d} and \eqref{YG3e}.
Indeed, the contribution to~\eqref{stoconv1}
from $n_1 = 0$ (i.e.~$n = n_2$) is given by 
\begin{align*}
\phi_0
(B_{0}(t) - B_{0}(r))
  \sum_{n\in \Z} 
 \ft f(n)
 e_n 
\end{align*}

\noi
which has the same (spatial) regularity as the input function $f$
and thus there is no regularity gain 
for this term.

%in this case.

\smallskip

\noi
(ii)
As pointed out in Remark \ref{REM:n0}, 
the condition $n \ne 0$ in \eqref{stoconv1}
does not play any role
and Proposition \ref{PROP:drive2}
holds even if the summation over $n \in \Z_*$
in \eqref{stoconv1}
is replaced by $n \in \Z$.

\end{remark}

We now present a proof of
Proposition \ref{PROP:drive2}.

\begin{proof}[Proof of Proposition \ref{PROP:drive2}]
We first note that the bound \eqref{YG1a} follows from \eqref{YG1} and
Kolmogorov's continuity criterion (Lemma \ref{LEM:kolm}).
The claim \eqref{YG1b}
then follows from \eqref{X1y}
with $\updl \YY = 0$
(which follows from \eqref{stoconv1}).
Hence, we focus on proving 
\eqref{YG1} in the following.

Fix  $0 \le r <  t \le T$.
From \eqref{stoconv1} with \eqref{I1}, we have 
\begin{align}
\Ft_x\big(\jb{\nb}^{s_0}\YY_{t,r} \jb{\nb}^{-s}f \big)(n) 
=  \sum_{n_2\in \Z} \ft f (n_2) I_1[ \hf_{nn_2}(n_1) \ff_{nn_2}(t') ] , 
\label{YG2}
\end{align}

\noi
where $\hf_{nn_2}(n_1) $ and $\ff_{nn_2}$ are defined by 
\begin{align}
\begin{split}
\hf_{n n_2}(n_1) & = 
 \ind_{n=n_{12}}
\cdot
%\ind_{n n_1\ne 0}
%\cdot 
\frac{\jb{n}^{s_0}}{\jb{n_2}^{s}} \phi_{n_1} , \\
\ff_{nn_2} (t')  &= 
\ff_{nn_2}^{t, r} (t')= 
 \ind_{[r, t]}(t')\cdot 
%\ind_{n n_1\ne 0}
%\cdot 
 e^{- i w(t') (n^3 - n_2^3)}.
\end{split}
\label{YG3}
\end{align}

\noi
Given dyadic $N, N_1, N_2 \ge1$, we set 
\begin{align}
\begin{split}
\hf^{\bf N}_{n n_2}(n_1) & = 
\hf^{N, N_1, N_2}_{n n_2}(n_1)
= \ind_{E_{\bf N}}
\cdot 
 \hf_{n n_2}(n_1), \\
\ff^{\bf N}_{nn_2} (t')  &= 
\ff^{N, N_1, N_2}_{nn_2} (t')  = 
\ind_{E_{\bf N}}
\cdot 
\ff_{nn_2} (t'),  
\end{split}
\label{YG3a}
\end{align}

\noi
where $E_{\bf N}$ is defined by 
\begin{align}
\begin{split}
E_{\bf N} = \big\{(n, n_1, n_2) \in \Z^3:
\ &  n = n_1 + n_2, 
n n_1 \ne 0,\\
& 
 |n|\sim N, \, |n_j|\sim N_j, \, j = 1, 2\big\}.
\end{split}
\label{EN1}
\end{align}

\noi
Then, given any $\ta > 0$, 
it follows from \eqref{YG2}, \eqref{YG3}, \eqref{YG3a}, 
and the random tensor estimate (Lemma~\ref{LEM:RT})
with \eqref{phi1a}
that 
\begin{align}
\begin{split}
 \big\| &  \| \YY_{t,r} \|_{\LOP(H^s;H^{s_0})} \big\|_{L^p(\Om)}
 = \big\| \| \jb{\nb}^{s_0}\YY_{t,r}\jb{\nb}^{-s} \|_{\LOP(L^2;L^2)} \big\|_{L^p(\Om)}\\
& = \big\| \| I_1[ \hf_{nn_2}(n_1) \ff_{nn_2}(t')  ] \|_{\l^2_{n_2} \to \l^2_n} \big\|_{L^p(\Om)}\\
%&  \big\|  \| I_1[ \hf_{nn_2}(n_1) \ff_{nn_2} (t')  ] \|_{\l^2_{n_2} \to \l^2_n} \big\|_{L^p(\Om)}  \\
& \le \sum_{\substack{N, N_1, N_2 \ge1 \\\text{dyadic}}}  
\big\| \| 
I_1[ \hf^{\bf N}_{nn_2}(n_1) \ff^{\bf N}_{nn_2}(t')] \|_{\l^2_{n_2} \to \l^2_n} \big\|_{L^p(\Om)}\\
& \les 
p^\frac 12 
\sum_{\substack{N, N_1, N_2 \ge1 \\\text{dyadic}}}  
\frac{N_{\max}^{\frac 12 \ta} N^{\ta s_0}}
{N_1^{ \ta\s} N_2^{\ta s}}
\|\phi\|_{\HS(L^2; H^\s)}^\ta
 \|  \ff^{\bf N}_{nn_2} (t')\|_{\l^\infty_{nn_2} \Hs_{t'}^\be}\\
& \hphantom{XXXXX}\times 
 \max\Big( \|  \hf^{\bf N}_{nn_2}(n_1)\|_{n_1n_2 \to n} , 
 \|   \hf^{\bf N}_{nn_2}(n_1) \|_{n_2 \to n n_1} \Big)^{1-\ta}
\end{split}
\label{YG3b}
\end{align}

\noi
for any finite $p \ge 1$, 
where 
 $N_{\max} = \max(N, N_1, N_2)$ and 
 $\Hs^\be(\R_+)$ is as in \eqref{BM0}.

From \eqref{YG3a}, \eqref{YG3}, and 
Cauchy-Schwarz's inequality with \eqref{phi1a}, we have
\begin{align}
\begin{split}
\|  \hf^{\bf N}_{nn_2}(n_1)\|_{n_1n_2 \to n} 
&\sim 
\frac{N^{s_0}}{N_2^s}
\sup_{\|f\|_{\l^2_{n_1 n_2}}
= \|g\|_{\l^2_{n}}  = 1}
\bigg|
\sum_{n, n_1, n_2 \in \Z}
\ind_{E_{\bf N}}
\cdot 
\phi_{n_1} f_{n_1n_2} g_{n}\bigg|\\
&\les
\frac{N^{s_0}}{N_1^\s N_2^s}
\|\phi\|_{\HS(L^2; H^\s)}.
\end{split}
\label{YG3x}
\end{align}

\noi
A similar computation yields
\begin{align}
\begin{split}
\|  \hf^{\bf N}_{nn_2}(n_1)\|_{n_2 \to n n_1} 
&\les
\frac{N^{s_0}}{N_1^\s N_2^s}
\|\phi\|_{\HS(L^2; H^\s)}.
\end{split}
\label{YG3y}
\end{align}

Next,  we  estimate the 
$\Hs^\be_{t'}$-norm of 
$ \ff_{nn_2}^{\bf N}(t')$
appearing in \eqref{YG3b}.
For notational simplicity, 
we drop the superscript $\bf N$
in the following
but it is understood that 
$n = n_1 + n_2$ and $n n_1 \ne 0$ hold; see \eqref{EN1}.
By integration by parts with~\eqref{rho2}, we have 
\begin{align}
\begin{split}
\ft \ff_{nn_2}(\tau) 
& = \frac 1{\sqrt {2\pi}}\int_r^t e^{- i t'\tau }  e^{- i w(t') (n^3 - n_2^3)}dt'\\
&  =  \frac 1{\sqrt {2\pi}}\Phi^w_{t, r}(-n^3 + n_2^3) e^{-i t \tau}
+  \frac {i \tau}{\sqrt {2\pi}} \int_r^t 
\Phi^w_{t', r}(-n^3 + n_2^3)e^{-i t'\tau} dt'.
\end{split}
\label{YG3c}
\end{align}

\noi
Under $n = n_1 + n_2$, we have
\begin{align}
n^3 - n_2^3 = n_1 \big((n_1 + \tfrac 32 n_2)^2 + \tfrac 34 n_2^2\big).
\label{YG3d}
\end{align}

\noi
In particular, for $n_1 \ne 0$, we have 
\begin{align}
|n^3 - n_2^3| \sim \jb{n_1} \jb{n_{\max}}^2 , 
\label{YG3e}
\end{align}

\noi
where $n_{\max} = \max(|n|, |n_1|, |n_2|)$.
Then,  from \eqref{YG3c} and \eqref{YG3e}
with  
\eqref{rho1}, we have
\begin{align}
\begin{split}
|\ft \ff_{nn_2}(\tau) |
& \les
\|\Phi^w\|_{  \W^{\rho,\g}_T} 
\frac{|t - r|^\g}{\jb{n^3 - n_2^3}^\rho}
+ 
\|\Phi^w\|_{  \W^{\rho,\g}_T} |\tau |
\frac{|t - r|^{1+\g}}{\jb{n^3 - n_2^3}^\rho}\\
& \les
\|\Phi^w\|_{  \W^{\rho,\g}_T} 
\frac{|t - r|^\g}{\jb{n_1}^\rho \jb{n_{\max}}^{2\rho}}
+ \|\Phi^w\|_{  \W^{\rho,\g}_T} 
|\tau |
\frac{|t - r|^{1+\g}}{\jb{n_1}^\rho \jb{n_{\max}}^{2\rho}}.
%& =: A_1 + A_2
\end{split}
\label{YG4}
\end{align}

Let $\Hs^\be(\R_+)$ be as in 
\eqref{BM0}.
Then, we have 
\begin{align}
\begin{split}
\|\ff_{nn_2} \|_{\Hs^\be_{t'}}
& \sim \bigg(\int_{|\tau|\le 1} |\tau|^{1-2\be} |\ft \ff_{nn_2}(\tau) |^2d\tau \bigg)^\frac 12 \\
& \quad + \bigg(\int_{|\tau|>  1} |\tau|^{1-2\be} |\ft \ff_{nn_2}(\tau) |^2d\tau \bigg)^\frac 12  \\
& =: G_1 + G_2.
\end{split}
\label{YG5}
\end{align}

\noi
From \eqref{YG4}, we have 
\begin{align}
G_1
\les 
\jb{T}
\|\Phi^w\|_{  \W^{\rho,\g}_T}
\frac{|t - r|^{\g}}{\jb{n_1}^{\rho} \jb{n_{\max}}^{2\rho}},
\label{YG6}
\end{align}

\noi
since $\be < 1$.

Let 
$\ld = \frac 13 (2\be -1) -\eps \in (0, 1)$ 
for some small $\eps > 0$
and 
$b = \ld^{-1}(1- 2\be)$
such that $b < -3$.
Then, 
by H\"older's inequality with 
 \eqref{YG3} and \eqref{YG4}, we have 
\begin{align}
\begin{split}
G_2
& \le \bigg(\int_\R  |\ft \ff_{nn_2}(\tau) |^2d\tau\bigg)^{\frac {1-\ld} 2}
\bigg(\int_{|\tau|>  1} |\tau|^b |\ft \ff_{nn_2}(\tau) |^2d\tau \bigg)^\frac \ld 2\\
& \les 
\jb{T}^\ld
\|\Phi^w\|_{  \W^{\rho,\g}_T}^\ld 
|t- r|^{\frac{1-\ld}2}
\bigg(\frac{|t - r|^{\g}}{\jb{n_1}^{\rho} \jb{n_{\max}}^{2\rho}}\bigg)^{\ld}\\
& \le 
\jb{T}
\big(1\vee \|\Phi^w\|_{  \W^{\rho,\g}_T}\big)
\frac {|t- r|^{\frac 12 + \frac 16 (2\be - 1)(2\g-1) -  \frac 12(2\g-1)\eps}}
{\jb{n_1}^{\frac 13\rho (2 \be - 1-3\eps)} \jb{n_{\max}}^{\frac 23\rho (2 \be - 1 -3\eps)}}.
\end{split}
\label{YG7}
\end{align}

\noi
Hence, from \eqref{YG3a}, \eqref{YG5}, \eqref{YG6}, and \eqref{YG7}
with $\be < 1$,
we obtain
\begin{align}
\begin{split}
\|\ff^{\bf N}_{nn_2}\|_{\Hs^\be_{t'}}
& \le \ind_{E_N} \cdot \|\ff_{nn_2}\|_{\Hs^\be_{t'}}\\
& \les \ind_{E_N} \cdot  
\jb{T} 
\jb{\|\Phi^w\|_{  \W^{\rho,\g}_T} }
\frac {|t- r|^{ \frac 12 + \frac 16 (2\be - 1)(2\g-1) -  \frac 12(2\g-1)\eps}}
{\jb{n_1}^{\frac 13\rho (2 \be - 1-3\eps)} \jb{n_{\max}}^{\frac 23\rho (2 \be - 1 -3\eps)}}, 
\end{split}
\label{YG8}
\end{align}

\noi
uniformly in $n, n_2 \in \Z_*$ and dyadic $N, N_1, N_2 \ge 1$.
Here, we used the fact that 
\[  \frac 12 + \frac 16 (2\be - 1)(2\g-1) < \g.\]

Therefore, putting 
\eqref{YG3b}, \eqref{YG3x}, \eqref{YG3y}, 
and \eqref{YG8} together with \eqref{ord1}, we obtain
\begin{align}
\begin{split}
  \big\|  &  \| \YY_{t,r} \|_{\LOP(H^s;H^{s_0})} \big\|_{L^p(\Om)}\\
& \les 
p^\frac 12 
\jb{T} 
\jb{\|\Phi^w\|_{  \W^{\rho,\g}_T} }
\|\phi\|_{\HS(L^2; H^\s)}
 |t- r|^{\frac 12 + \frac 16 (2\be - 1)(2\g-1) - \frac 12  (2\g-1)\eps}\\
%\sum_{\substack{N, N_1, N_2 \ge1 \\\text{dyadic}}}  
& \quad \times \sum_{\substack{N, N_1, N_2 \ge1 \\\text{dyadic}\\
N_{\max}\sim N_{\med}}}  
N_{\max}^{\frac 12 \ta
- \frac 23\rho (2 \be - 1 -3\eps)}
\frac{N^{s_0}}{N_1^{\s+ \frac 13\rho (2 \be - 1-3\eps)} N_2^{s}}\\
& \les_T 
p^\frac 12 
\jb{\|\Phi^w\|_{  \W^{\rho,\g}_T} }
\|\phi\|_{\HS(L^2; H^\s)}
 |t- r|^{\g_0}, 
\end{split}
\label{YG9}
\end{align}

\noi
provided that 
$\g_0 \le \g$, 
\eqref{YT1}, 
\eqref{YT2},  and \eqref{YT3} 
hold
and that $\ta, \eps > 0$ are sufficiently small.
Here, the second inequality in \eqref{YG9} follows
from separately considering the cases:
 (a)~$N \sim N_1 \ges N_2$, 
(b)~$N \sim N_2 \ges N_1$, 
and 
(c)~$N_1 \sim N_2 \ges N$
(also depending on the sign of the exponent of $N_{\min}$).
This proves~\eqref{YG1}.
\end{proof}

\begin{remark}\label{REM:BO}
\rm

Let us consider the stochastic modulated BO \eqref{sbo1} with a Young noise, 
where the dispersion symbol is given by $\o_\BO(n) = |n|n$.
Let $\YY^\BO$ denote 
the associated Young driver
defined by
\begin{align*}
\YY^\BO_{t, r}(f)
&  = \P_{\ne 0} \int_r^t
\uw_\BO(t')^{-1} \big[ (\uw_\BO(t') f)
\phi dW^\be (t') \big]\\
&
 = \sum_{n\in \Z_*} e_n 
 \int_r^t \sum_{n_1, n_2  \in \Z} \ind_{\substack{n = n_{12}\\n_1 \ne 0}}
\cdot  e^{- i w(t') (|n|n - |n_2|n_2)}\ft f(n_2) \phi_{n_1}    dB_{n_1}(t'), 
%\label{sto1}
\end{align*}

\noi
where 
$\uw_\BO (t)=e^{   w(t)\H\dx^2}  $
denotes the modulated linear propagator for \eqref{BO}.

Let 
$\rho \ge\frac 12$,
$\frac 12 < \g, \be < 1$, and $s, \s \in \R$
such that 
\begin{align}
s + \s > -  \frac 23 \rho (2\be -1 ), 
\label{BOX3}
\end{align}

\noi
which is the first condition in \eqref{BOZ}.
Suppose that 
 $s_0 \in \R$ satisfies
\begin{align}
\begin{split}
s_0 & < \min \bigg(s + \frac 13 \rho(2\be -1)
+ \Big(\s +  \frac 13 \rho(2\be -1)\Big) \wedge 0, \\
& \hphantom{iiXXXX}
s\wedge0 + \s + \frac 23  \rho (2\be -1 )\bigg).
\end{split}
\label{BOX4}
\end{align}

\noi
In this case, instead of \eqref{YG3d}, we have
\begin{align}
\big| |n| n - |n_2| n_2\big|
\sim  \begin{cases}
|n_1| n_{\max}, & \text{if } nn_2 > 0,\\
 n_{\max}^2, & \text{if } nn_2 < 0
\end{cases}
\label{BOX1}
\end{align}

\noi
under $n = n_1 + n_2$.
Then, by repeating the proof of Proposition \ref{PROP:drive2}, 
we obtain
\begin{align}
\begin{split}
  \big\|  &  \| \YY^\BO_{t,r} \|_{\LOP(H^s;H^{s_0})} \big\|_{L^p(\Om)}\\
& \les_T 
p^\frac 12 
\|\phi\|_{\HS(L^2; H^\s)}
\jb{\|\Phi^w\|_{  \W^{\rho,\g}_T}}  |t- r|^{\frac 12 + \frac 16 (2\be - 1)(2\g-1) - \frac 12  (2\g-1)\eps}\\
%\sum_{\substack{N, N_1, N_2 \ge1 \\\text{dyadic}}}  
& \quad \times \sum_{\substack{N, N_1, N_2 \ge1 \\\text{dyadic}\\
N_{\max}\sim N_{\med}}}  
N_{\max}^{\frac 12 \ta
- \frac 13\rho (2 \be - 1 -3\eps)}
\frac{N^{s_0}}{N_1^{\s+ \frac 13\rho (2 \be - 1-3\eps)} N_2^{s}}\\
& \les_T 
p^\frac 12 
\|\phi\|_{\HS(L^2; H^\s)}
\jb{\|\Phi^w\|_{  \W^{\rho,\g}_T}}  |t- r|^{\g_0}, 
\end{split}
\label{BOX2}
\end{align}

\noi
provided that 
\eqref{YT1}, 
\eqref{BOX3}, and \eqref{BOX4}
hold
and that $\ta, \eps > 0$ are sufficiently small.

Next, consider  the stochastic modulated ILW \eqref{SILW1}
with a Young noise.
The dispersion symbol  $\o_\dl(n)$ for ILW \eqref{ILW1}
(= the multiplier for $-i \Gdl \dx^2$ with a fixed depth parameter $\dl > 0$) 
is   given by 
\begin{align}
\o_\dl(n)=
\ind_{n\ne 0}\cdot \bigg(n^2 \coth(\dl n) -\frac{n}{\dl} \bigg).
\label{ILX1}
\end{align}

\noi
See \eqref{GG1}.
By applying 
\cite[(2.7)]{CLOP}
and 
the fact that  $x^{\al +1} \le C_\al( e^{2x} - 1)$ for any $x \ge  0$
and $\al \ge 0$, 
we have 
\begin{align}
n^2 \coth (\dl n) - |n|n = \sgn(n) \cdot \frac {2n^2} {e^{2|\dl n|} - 1}
= O( \dl^{-2}), 
\label{ILX2}
\end{align}

\noi
uniformly in  $n \in \Z_*$.
Hence, under $n = n_1 + n_2$, 
from \eqref{ILX1}, \eqref{ILX2}, \eqref{BOX1}, we have
\begin{align}
\begin{split}
|\o_\dl(n) - \o_\dl(n_2) |
& = |\o_\BO(n) - \o_\BO(n_2) |
+ O(|n_1|\dl^{-1})
+ O(\dl^{-2})\\
& \ges_\dl   \begin{cases}
|n_1| n_{\max}, & \text{if } nn_2 > 0,\\
 n_{\max}^2, & \text{if } nn_2 < 0, 
\end{cases}
\end{split}
\label{ILX3}
\end{align}

\noi
uniformly in $n, n_1, n_2 \in \Z_*$, 
provided that $n_{\max} \gg \dl^{-2}$.
Hence, the bound \eqref{BOX2}
also holds 
for the stochastic modulated ILW \eqref{SILW1}
with a Young noise
(with an implicit constant depending on $\dl > 0$).

\end{remark}

\subsection{Rough driver  for the noise}
\label{SUBSEC:drive3}

In this subsection, 
we study regularity properties
of the drivers $\wt \YY$ and $\wt \YYb$ in 
\eqref{sto2x}
and \eqref{sto3x}, respectively, 
associated with 
SmKdV \eqref{skdv3} forced by a white-in-time noise.
From \eqref{sto3x}
with \eqref{W1} and \eqref{phi1}, we have 
 \begin{align}
\begin{split}
\wt \YYb_{t, r}(f)
&
 = \sum_{n\in \Z_*} e_n 
 \int_r^t
  \int_r^{t_1}
  \sum_{n_1, n_2, n_3  \in \Z} 
\ind_{\substack{n = n_{123}\\ n_{23} \ne 0}}\cdot 
\frac 1 {\jb{n}^{\eps_0}  \jb{n_{23}}^{\eps_0}}\\
& \hphantom{XXXXXXX}
\times  e^{- i w(t_1) (n^3 - n_{23}^3)}
  e^{- i w(t_2) (n_{23}^3 - n_{3}^3)}
 \ft f(n_3) \\
& \hphantom{XXXXXXX} 
\times \phi_{n_1} \phi_{n_2}    dB_{n_2}(t_2) dB_{n_1}(t_1), 
\end{split}
\label{stoconv3}
\end{align}

\noi
where the iterated stochastic integral 
is interpreted as an iterated Wiener-Ito integral.

\begin{proposition}
\label{PROP:drive3}

Given $\rho \ge\frac 12$, $ \frac 12 <  \g < 1$,   $s, \s \in \R$, 
 $\eps_0>0$, and $T > 0$, 
such that 
\begin{align}
s + \s > 0, 
\label{RG0}
\end{align}

\noi
let $w$ be 
 a $(\rho,\g)$-irregular function on $[0, T]$ in the sense of Definition~\ref{DEF:ir}
 and 
 $\phi \in \HS(L^2(\T); H^\s(\T))$, 
  satisfying \eqref{phi1}.

\smallskip

\noi
\textup{(i)}
Suppose that  $s_0 \in \R$ satisfies
\begin{align}
s_0 < \min \Big( s\wedge 0 + \s + \eps_0, \, 
s+ \s \wedge 0 + \eps_0
\Big).
\label{RG0a}
\end{align}

\noi
Let  $\wt \YY = \wt \YY^{w, \phi}$ be 
the random operator 
 defined in \eqref{sto2x}.
 Then, given any finite $p \ge 1$, we have 
\begin{align}
\big\| \| \wt  \YY_{t,r}  \|_{\LOP(H^s; H^{s_0})} \big\|_{L^p(\Om)} 
 & \les  p^\frac 12 \| \phi\|_{\HS(L^2;H^\s)} 
  |t-r|^\frac 12 
\label{RG1}
\end{align}

\noi
for any   $(t,r) \in \Dl_{2,T}$.
Consequently, we have 
\begin{align*}
\big\| \| \wt \YY\|_{C^{\g_0}_{2,T} \LOP(H^s;H^{s_0})} \big\|_{L^p(\Om)} \les 
p^\frac 12 \| \phi\|_{\HS(L^2;H^\s)} 
\end{align*}

\noi
for 
any finite $p \ge 1$ and $0 < \g_0 < \frac 12$.
In particular, 
there exists a version of $\wt \YY$
such that 
$\wt \YY \in C^{\g_0}_{2,T} \LOP(H^s(\T);H_0^{s_0}(\T))$ and 
\begin{align*}
\wt \YY \in  \cX^{s, s_0, \g_0}_{1, 0} ([0, T]\times \T) ,
%\label{RG2x}
\end{align*}

\noi
almost surely.
See Subsection \ref{SUBSEC:2.2}
for the definitions of these function spaces.

\smallskip

\noi
\textup{(ii)}
In addition to \eqref{RG0}, suppose that 
\begin{align}
\begin{split}
s + 2\s + \eps_0 & > 0,\\
s\wedge \s + \s + \eps_0 & > 0.
\end{split}
\label{RG2a}
\end{align}

\noi
Suppose that  $s_0 \in \R$ satisfies
\begin{align}
\begin{split}
s_0 < \min \Big(  %s + 2\s + \eps_0, \, 
s\wedge 0 + \s + \s \wedge 0 + \eps_0, 
 s\wedge \s  + \s \wedge 0 + 2\eps_0
\Big).
\end{split}
\label{RG2b}
\end{align}

\noi
Let  $\wt \YYb = \wt \YYb^{w, \phi}$ be 
the random operator 
 defined in \eqref{stoconv3}.
 Then, given any finite $p \ge 1$, we have 
\begin{align}
\big\| \|  \wt \YYb_{t,r}  \|_{\LOP(H^{s}; H^{s_0})} \big\|_{L^p(\Om)} & \les  p  \| \phi \|_{\HS(L^2;H^\s)}^2  
 |t-r|
\label{RG3}
\end{align}

\noi
for any   $(t,r) \in \Dl_{2,T}$.
Consequently, we have 
\begin{align}
\big\| \| \wt \YYb \|_{C^{\g_1}_{2,T} \LOP(H^s;H^{s_0})} \big\|_{L^p(\Om)} \les p  \| \phi \|_{{\HS} (L^2; H^\s)}^2.
\label{RG4}
\end{align}

\noi
for 
any finite $p \ge 1$ and $0 < \g_1 < 1$.
In particular, 
there exists a version of $\wt \YYb$
such that 
$\wt \YYb \in C^{\g_1}_{2,T} \LOP(H^s(\T);H^{s_0}_0(\T))$
and
\begin{align*}
\wt \YYb \in  \cX^{s, s_0, \g_1}_{1, 0} ([0, T]\times \T) ,
%\label{RG2y}
\end{align*}

\noi
almost surely.

\smallskip

\noi
\textup{(iii)} 
The pair $(\wt \YY, \wt \YYb)$
satisfies
Chen's relation \eqref{RQ2}{\rm :}
\begin{align}
(\updl \wt \YYb)_{t_1,t_2,t_3} & = \wt \YY_{t_1,t_2} \circ \wt \YY_{t_2,t_3}
\label{RG4a}
\end{align}

\noi
for any $(t_1, t_2, t_3) \in \Dl_{3, T}$, almost surely.

\end{proposition}

As mentioned in Subsection \ref{SUBSEC:main2}, 
Theorem \ref{THM:4}
on the construction of the stochastic term~$\wt \PPsi(\uu)$
 in \eqref{psi3}
 as the rough integral $\I^{\wt \YY, \wt \YYb}(\uu)$
follows from 
Lemma~\ref{LEM:int2}
and 
Proposition~\ref{PROP:drive3}, 
while 
Theorems \ref{THM:5} 
on local well-posedness of 
for  SmKdV \eqref{skdv3}
with a multiplicative rough noise
 follows from 
Proposition \ref{PROP:main2}
with 
Lemma \ref{LEM:kdv1} and 
Proposition \ref{PROP:drive3}.
We omit details.

\begin{proof} [Proof of Proposition \ref{PROP:drive3}]
(i) 
Arguing as in the proof of 
Proposition \ref{PROP:drive2}, 
it suffices to prove the bound \eqref{RG1}.

Fix  $0 \le r <  t \le T$.
Given any $\ta > 0$, 
proceeding as in \eqref{YG3b}
with 
\eqref{sto2x}, 
\eqref{stoconv1}, 
and 
 the random tensor estimate (Lemma~\ref{LEM:RT}), 
we have 
\begin{align}
\begin{split}
&   \big\|    \| \wt \YY_{t,r} \|_{\LOP(H^s;H^{s_0})} \big\|_{L^p(\Om)}
 =   \big\|   \| \jb{\nb}^{s_0}\wt \YY_{t,r}\jb{\nb}^{-s} \|_{\LOP(L^2;L^2)} \big\|_{L^p(\Om)}\\
& \quad \les 
p^\frac 12 
\sum_{\substack{N, N_1, N_2 \ge1 \\\text{dyadic}}}  
N^{-\eps_0} 
\frac{N_{\max}^{\frac 12 \ta} N^{\ta s_0}}{N_1^{\ta \s}N_2^{\ta s}}
 \|\phi\|_{\HS(L^2; H^\s)}^\ta \|  \ff^{\bf N}_{nn_2} (t')\|_{\l^\infty_{nn_2} L^2_{t'}}
\\
& \hphantom{XXXX}\times 
 \max\Big( \|  \hf^{\bf N}_{nn_2}(n_1)\|_{n_1n_2 \to n} , 
 \|   \hf^{\bf N}_{nn_2}(n_1) \|_{n_2 \to n n_1} \Big)^{1-\ta}
\end{split}
\label{RG5}
\end{align}

\noi
for any finite $p \ge 1$, 
where 
$\hf^{\bf N}$ 
and $\ff^{\bf N}$
are as in \eqref{YG3a} (see also \eqref{YG3}).
From \eqref{YG3}, we have 
\begin{align}
\begin{split}
\|\ff^{\bf N}_{nn_2}\|_{L^2_{t'}}
& \sim |t- r|^\frac 12, 
\end{split}
\label{RG7}
\end{align}

\noi
uniformly in $n, n_2 \in \Z$ and dyadic $N, N_1, N_2 \ge1$.
Hence, from \eqref{RG5}, 
 \eqref{YG3x}, \eqref{YG3y}, 
and \eqref{RG7} with \eqref{ord1}, 
we obtain
\begin{align}
\begin{split}
 \big\|   \| \wt \YY_{t,r} \|_{\LOP(H^s;H^{s_0})} \big\|_{L^p(\Om)}
& \les 
p^\frac 12 
\sum_{\substack{N, N_1, N_2 \ge1 \\\text{dyadic}\\
N_{\max}\sim N_{\med}}}  
 N_{\max}^{\frac 12 \ta} 
\frac{N^{s_0-\eps_0}}{N_1^{\s} N_2^{s}}
\|\phi\|_{\HS(L^2; H^\s)}|t- r|^\frac 12\\
& \les 
p^\frac 12 
\|\phi\|_{\HS(L^2; H^\s)}|t- r|^\frac 12, 
\end{split}
\label{RG8}
\end{align}

\noi
provided that 
\eqref{RG0} and \eqref{RG0a}
hold
and that $\ta > 0$ is sufficiently small.
Here, the second inequality in \eqref{RG8} follows
from separately considering the cases:
 (a)~$N \sim N_1 \ges N_2$, 
(b)~$N \sim N_2 \ges N_1$, 
and 
(c)~$N_1 \sim N_2 \ges N$
(also depending on the sign of the exponent of $N_{\min}$).
This proves \eqref{RG1} as desired.

\medskip

\noi
(ii)
As in Part (i), 
it suffices to prove the bound \eqref{RG3}, 
since the bound \eqref{RG4} follows
from~\eqref{RG3} and Lemma \ref{LEM:kolm}.
\\
\indent
Fix  $0 \le r <  t \le T$.
From \eqref{stoconv3}, we have 
\begin{align}
\Ft_x\big( \jb{\nabla}^{s_0} \wt \YYb_{t,r} \jb{\nabla}^{-s} f\big) (n) 
= \sum_{n_3 \in \Z} \ft f(n_3) I_2 [ \wt{\hf}_{nn_3}(n_A) \wt{\ff}_{nn_3}(t_A,n_A)  ], 
\label{RGG1}
\end{align}

\noi
where
$I_2$ denotes the iterated Wiener-Ito integral (see Subsection \ref{SUBSEC:FBM}).
Here,  
with  $A = \{1, 2\}$, 
$\wt \hf_{nn_3}(n_A) $ and $\wt \ff_{nn_3}(t_A, n_A)$ are defined by 
\begin{align}
\begin{split}
\wt \hf_{n n_3}(n_A) & = 
 \ind_{n=n_{123}}
\cdot \ind_{nn_{23} \ne 0}
\cdot 
\frac {\jb{n}^{s_0 - \eps_0}} {\jb{n_3}^s  \jb{n_{23}}^{\eps_0}}
\phi_{n_1}
\phi_{n_2}, \\
\wt \ff_{nn_3} (t_A, n_A)  &= 
\wt \ff_{nn_3}^{t, r} (t_A, n_A)= 
\ind_{r \le t_2 \le t_1\le t}
\cdot 
 e^{- i w(t_1) (n^3 - n_{23}^3)}
  e^{- i w(t_2) (n_{23}^3 - n_{3}^3)}.
\end{split}
\label{RGG2}
\end{align}

\noi
Given dyadic $N, N_1, N_2, N_3, N_{23} \ge1$, we set 
\begin{align}
\begin{split}
\wt \hf^{\bf N}_{n n_3}(n_A) & = 
\wt\hf^{N, N_1, N_2, N_3, N_{23}}_{n n_3}(n_A)
= \ind_{\wt E_{\bf N}}
\cdot 
\wt \hf_{n n_3}(n_A), \\
\wt\ff^{\bf N}_{nn_3} (t_A, n_A)  &= 
\wt \ff^{N, N_1, N_2, N_3, N_{23}}_{nn_3} (t_A, n_A)  = 
\ind_{\wt E_{\bf N}}
\cdot 
\wt \ff_{nn_3} (t_A, n_A),  
\end{split}
\label{RGG3}
\end{align}

\noi
where $\wt E_{\bf N}$ is defined by 
\begin{align*}
\wt E_{\bf N} = \big\{(n, n_1, n_2, n_3) \in \Z^4:
\ & n = n_{123}, \, nn_{23} \ne 0, \, |n|\sim N, \\
&  |n_j|\sim N_j, \, j = 1, 2, 3, \, |n_{23}|\sim N_{23}\big\}.
\end{align*}

\noi
Under $n = n_1 + n_{23}$ 
and $n_{23} = n_2 + n_3$, we have 
\begin{align}
\begin{split}
\max(N, N_1, N_{23})
& \sim \med(N, N_1, N_{23}),\\
%\quad 
%\text{and}\quad 
\max(N_{23}, N_2, N_{3})
& \sim \med(N_{23}, N_2, N_{3}), 
\end{split}
\label{EN3}
\end{align}

\noi
where 
$\max(A, B, C)$ (and 
$\med(A, B, C)$) denotes
the largest element (and the second largest element, respectively)  in $A$, $B$, and $C$.
Then, given any $\ta > 0$, 
it follows from \eqref{RGG1}, \eqref{RGG2}, \eqref{RGG3}, 
and the random tensor estimate (Lemma~\ref{LEM:RT})
that 
\begin{align}
\begin{split}
 \big\| &  \| \wt \YYb_{t,r} \|_{\LOP(H^s;H^{s_0})} \big\|_{L^p(\Om)}
 = \big\| \| \jb{\nb}^{s_0}\wt \YYb_{t,r}\jb{\nb}^{-s} \|_{\LOP(L^2;L^2)} \big\|_{L^p(\Om)}\\
& = \big\| \| I_2[ \wt \hf_{nn_3}(n_A) \wt \ff_{nn_3}(t_A, n_A)  ] \|_{\l^2_{n_3} \to \l^2_n} \big\|_{L^p(\Om)}\\
%&  \big\|  \| I_1[ \hf_{nn_2}(n_1) \ff_{nn_2} (t')  ] \|_{\l^2_{n_2} \to \l^2_n} \big\|_{L^p(\Om)}  \\
& \le \sum_{\substack{N, N_1, N_2, N_3,  N_{23} \ge1 \\\text{dyadic}}}  
\big\| \| 
I_2[ \wt \hf^{\bf N}_{nn_3}(n_A) \wt \ff^{\bf N}_{nn_3}(t_A, n_A)] \|_{\l^2_{n_3} \to \l^2_n} \big\|_{L^p(\Om)}\\
& \les 
p
\sum_{\substack{N, N_1, N_2, N_3, N_{23} \ge1 \\\text{dyadic}\\\eqref{EN3}}}  
\frac{N_{\max}^{\frac 12 \ta} N^{\ta (s_0-\eps_0)}}
{N_1^{ \ta\s} N_2^{\ta \s} N_3^{\ta s} N_{23}^{\ta\eps_0}}
\|\phi\|_{\HS(L^2; H^\s)}^{2\ta}\\
& \hphantom{XXXXX}\times 
 \| \wt  \ff^{\bf N}_{nn_3} (t_A, n_A)\|_{\l^\infty_{nn_3n_A}L^2_{t_A} }
 \Big(\max_{(B, C)} \| \wt  \hf^{\bf N}_{nn_3}(n_A)\|_{n_3n_B \to nn_C} \Big)^{1-\ta}
\end{split}
\label{RGG4}
\end{align}

\noi
for any finite $p \ge 1$, 
where 
 $N_{\max} = \max(N, N_1, N_2, N_3)$ and 
$(B, C)$ is a partition of $A = \{1, 2\}$.
Namely, the maximum 
on the last factor in \eqref{RGG4}
is taken over 
\[n_1 n_2 n_3  \to n , \qquad  
n_1 n_3 \to n n_2, 
\qquad n_2 n_3 \to nn_1, 
\qquad \text{and}\qquad 
n_3 \to nn_1 n_2. \]

From \eqref{RGG3} with \eqref{RGG2}, we have
\begin{align}
 \| \wt  \ff^{\bf N}_{nn_3} (t_A, n_A)\|_{\l^\infty_{nn_3n_A}L^2_{t_A} }
 \les |t-r|, 
\label{RGG5}
\end{align}

\noi
uniformly in 
$n, n_3 \in \Z$ and 
 dyadic $N, N_1, N_2, N_3 \ge1$.

From \eqref{RGG2}, \eqref{RGG3}, and 
Cauchy-Schwarz's inequality, we have
\begin{align*}
& \|  \wt \hf^{\bf N}_{nn_3}(n_A)\|_{n_1n_2n_3 \to n} \\
&\quad \sim 
\frac{N^{s_0-\eps_0}}{N_3^s  N_{23}^{\eps_0}}
\sup_{\|f\|_{\l^2_{n_1 n_2 n_3}}
= \|g\|_{\l^2_{n}}  = 1}
\bigg|
\sum_{n, n_1, n_2, n_3 \in \Z}
\ind_{\wt E_{\bf N}}
\phi_{n_1}\phi_{n_2} f_{n_1 n_2 n_3} g_{n}\bigg|\\
&\quad \les
\frac{N^{s_0-\eps_0}}{N_1^\s N_2^\s N_3^s N_{23}^{\eps_0}}
\|\phi\|_{\HS(L^2; H^\s)}^2.
\end{align*}

\noi
A similar computation yields
\begin{align}
\max_{(B, C)} \| \wt  \hf^{\bf N}_{nn_3}(n_A)\|_{n_3n_B \to nn_C} 
&\les
\frac{N^{s_0-\eps_0}}{N_1^\s N_2^\s N_3^s N_{23}^{\eps_0}}
\|\phi\|_{\HS(L^2; H^\s)}^2.
\label{RGG6}
\end{align}

Hence,
from 
\eqref{RGG4}, \eqref{RGG5}, 
and \eqref{RGG6}, 
we obtain
\begin{align}
\begin{split}
 \big\|   \| \wt \YYb_{t,r} \|_{\LOP(H^s;H^{s_0})} \big\|_{L^p(\Om)}
& \les 
p
\sum_{\substack{N, N_1, N_2, N_3, N_{23} \ge1 \\\text{dyadic}\\\eqref{EN3}}}  
\frac{N_{\max}^{\frac 12 \ta} N^{s_0-\eps_0}}{N_1^\s N_2^\s N_3^{s} N_{23}^{\eps_0}}
\|\phi\|_{\HS(L^2; H^\s)}^2|t-r|\\
& \les 
p
\|\phi\|_{\HS(L^2; H^\s)}^2|t-r|, 
\end{split}
\label{RGG7}
\end{align}

\noi
provided that 
\eqref{RG0}, \eqref{RG2a}, and \eqref{RG2b} hold
and 
that $\ta > 0$ is sufficiently small.
Here, the second inequality in \eqref{RGG7} follows
from separately considering the cases
(recall  $N_{\max} = \max(N, N_1, N_2, N_3)$):

\smallskip
\begin{itemize}
\item[(a)] 
$N\sim N_1 \ges N_{23}$.

\smallskip

\begin{itemize}
\item[(a.i)] 
$N \sim N_1 \sim N_{\max}$. 
In this case, \eqref{RGG7} holds if 
\begin{align}
s_0 <  s\wedge 0 + \s +  \s\wedge 0 + \eps _0 .
\label{RGY1}
\end{align}

\smallskip

\item[(a.ii)]
 $N_2 \sim N_3\sim N_{\max} \gg  N \sim N_1$.
In this case, \eqref{RGG7} holds if 
\begin{align}
s+\s > 0
\qquad \text{and}\qquad
s_0 < s+ 2\s + \eps _0.
\label{RGY2}
\end{align}

\end{itemize}

\smallskip
\item[(b)] 
$N\sim N_{23} \ges N_1$.

\smallskip
\begin{itemize}
\item[(b.i)] 
$N\sim N_{23} \sim N_{\max}$. 
In this case, we have $N_2 \vee N_3 \sim N_{\max}$, 
and \eqref{RGG7} holds if 
\begin{align}
s_0 < s\wedge \s  + \s \wedge 0 + (s\vee \s) \wedge 0+ 2\eps_0
= s\wedge \s  + \s \wedge 0 +  2\eps_0.
\label{RGY3}
\end{align}

\noi
Here, we used the fact that 
\begin{align}
\min \big( s + \s \wedge 0, s\wedge 0 + \s\big)
=  s\wedge \s  + (s\vee \s) \wedge 0
\label{RGY4}
\end{align}

\noi
and
$(s\vee \s) \wedge 0 = 0$ under \eqref{RG0}.

\smallskip
\item[(b.ii)]
 $N_2 \sim N_3\sim N_{\max} \gg N\sim N_{23}$.
In this case, \eqref{RGG7} holds if 
\begin{align}
s+\s > 0
\qquad \text{and}\qquad 
s_0 < s + \s + \s\wedge 0 + 2\eps_0.
\label{RGY5}
\end{align}
\end{itemize}

\smallskip
\item[(c)] 
$N_1\sim N_{23} \ges N$.

\smallskip
\begin{itemize}
\item[(c.i)] 
$N_1\sim N_{23}  \sim N_{\max}$.
In this case, we have $N_2 \vee N_3 \sim N_{\max}$, 
and \eqref{RGG7} holds if 
\begin{align}
s\wedge \s + \s + (s\vee \s) \wedge 0+ \eps_0 > 0 \vee (s_0 - \eps_0).
\label{RGY6}
\end{align}
%$s\wedge \s + \s + \eps_0 >0$

\noi
Here,  we used the identity \eqref{RGY4} once again.

\smallskip
\item[(c.ii)]
 $N_2 \sim N_3\sim N_{\max} \gg N_1\sim N_{23}$.
In this case, 
\eqref{RGG7} holds if 
\begin{align}
s + \s > 0
\qquad \text{and}
\qquad  
s + 2\s + \eps_0 > 0\vee (s_0 - \eps_0). 
\label{RGY7}
\end{align}

\end{itemize}

\end{itemize}

\noi
Putting \eqref{RGY1}, \eqref{RGY2}, \eqref{RGY3}, 
\eqref{RGY5}, \eqref{RGY6}, and \eqref{RGY7}, 
we obtain the conditions \eqref{RG0}, \eqref{RG2a} and~\eqref{RG2b}.

\medskip

\noi
(iii)
The identity \eqref{RG4a}
follows from 
a direct computation, using \eqref{sto2x}, \eqref{stoconv1},  and \eqref{stoconv3}.
We omit details.
\end{proof}

\begin{remark}\label{REM:rough1}\rm
In the rough case, namely $\be= \frac 12$, 
the Hilbert space $\H^\frac 12 (\R_+^k)$
defined in~\eqref{BM0a}
reduces to $L^2(\R_+^k)$.
In this case, as seen in the proof of Proposition \ref{PROP:drive3}
presented above, 
the modulated dispersion plays no role
and thus there is no regularization effect from
the modulated dispersion.
Compare 
\eqref{RG7} in the rough case
with 
\eqref{YG8}
in  the Young case.

We also note that,  
while 
the condition
 \eqref{phi2}: $\phi_0 = 0$
was essential in establishing regularization by noise
in the Young case (see Remark \ref{REM:YG1}), 
it does not play any role in the rough case.

\end{remark}

\appendix 
\section{Stochastic modulated Schr\"odinger equation
with a Young noise}
\label{SEC:NLS}

In this appendix, 
we consider the following stochastic modulated Schr\"odinger equation 
on~$\T^d$ with a multiplicative Young noise ($\frac 12 < \be < 1$):
\begin{align}
\begin{cases}
i \dt u+  \Dl u \cdot \dt w =  (\jb{\nb}^\kk \cj u) \phi \zeta\\
u|_{t = 0} = u_0
\end{cases}
\label{NLS2x}
\end{align}

\noi
for $\kk \ge0$, 
where 
$\phi \in \HS(L^2(\T^d); H^\s(\T^d))$, satisfying
\eqref{phi1}. 
Here, the noise $\z$ is given by $\z = \dt W^\be$, 
where 
 $W^\be$ as in \eqref{W1} with $\frac 12 < \be < 1$
(without the condition \eqref{W2}
in the current complex-valued setting).
Our main goal in this appendix is to prove
pathwise global well-posedness of~\eqref{NLS2x}.

We say that $u$ is a solution to \eqref{NLS2x}
if $u$ satisfies the following
 Duhamel formulation:
\begin{align}
\begin{split}
u(t) 
& = \sw(t) u_0 + 
\sw(t) \int_0^t \sw(t')^{-1}
\big[(\jb{\nb}^\kk \cj{u (t')} )\phi dW^\be(t')\big]\\
& =:
 \sw(t) u_0 + 
\Psi(u)(t), 
\end{split}
\label{ZZ1}
\end{align}

\noi
where 
$\sw (t)=e^{i   w(t)\Dl}  $
denotes the modulated linear propagator for 
the modulated Schr\"odinger equation.
At the level of the modulated interaction representation $\uu(t) = \sw(t)^{-1} u(t)$, 
\eqref{ZZ1} reads as 
\begin{equation}
\uu = u_0 +   \PPsi(\uu), 
\label{ZZ2}
\end{equation}

\noi
where the stochastic term $\PPsi(\uu)$ is given by 
\begin{align}
\begin{split}
 \PPsi(\uu) (t) 
 &  =   \int_0^t
\sw(t')^{-1} \big[ (\jb{\nb}^\kk\cj{\sw(t')\uu(t')} )
\phi dW^\be(t') \big]\\
& = \sum_{n\in \Z^d} e_n 
 \int_0^t \sum_{n_1, n_2  \in \Z^d} \ind_{n = n_1 - n_2}\\
& \hphantom{XXXXX}
\times  e^{ i w(t') (|n|^2 + |n_2|^2)} \jb{n_2}^\kk \cj{\ft \uu(t', n_2)} \phi_{n_1}    dB_{n_1}(t').
\end{split}
\label{ZZ3}
\end{align}

As in the case of SmKdV, our first goal is to construct 
$\PPsi(\uu)$ in \eqref{ZZ3} as the Young integral $\I^\YY(\uu)$
with the driver 
 $\YY$ given by 
\begin{align}
\begin{split}
\YY_{t, r}(f)
&  =  \int_r^t
\sw(t')^{-1} \big[ (\cj{\jb{\nb}^\kk \sw(t') f})
\phi dW^\be(t') \big]\\
&
 = \sum_{n\in \Z^d} e_n 
 \int_r^t \sum_{n_1, n_2  \in \Z^d} \ind_{n = n_1 - n_2}\\
& \hphantom{XXXXX}
\times  e^{ i w(t') (|n|^2 + |n_2|^2)} \jb{n_2}^\kk \cj{\ft f(n_2)} \phi_{n_1}    dB_{n_1}(t').
\end{split}
\label{sto2}
\end{align}

\noi
Then, \eqref{ZZ2} reduces
to the following YDE:
\begin{equation}
\label{NLS3}
\uu = u_0 +   \I^\YY(\uu)
\end{equation}

\noi
whose local well-posedness follows
from Proposition \ref{PROP:main} (with $\XX = 0$).

We now state our main result on \eqref{NLS2x}.

\begin{theorem}
\label{THM:NLS}
Let 
$d\ge 1$ and $\kk \ge 0$.
Given 
$\rho > 0$, 
$\frac 12 < \g, \be < 1$, and $ \s \in \R$, 
let 
$w$ be 
 a $(\rho,\g)$-irregular function on $\R_+$ in the sense of Definition~\ref{DEF:ir}
 and 
$\phi \in \HS(L^2(\T^d); H^\s(\T^d))$, satisfying
\eqref{phi1}.

\medskip

\noi
{\rm (i)}
Suppose that 
$  \g_0 > 0 $ and $s, s_0 \in \R$ satisfy~\eqref{YT1}, 
\begin{align}
s + \s > \kk -  \frac 23  \rho (2\be -1 ), 
\label{SG00}
\end{align}

\noi
and
\begin{align}
\begin{split}
s_0 & < \min\Big( (s-\kk)\wedge 0 + \s  + \frac 23 \rho(2\be -1), \\
& \hphantom{XXXXi} 
(s-\kk)+ \s \wedge 0 + \frac 23 \rho(2\be -1)\Big).
\end{split}
\label{SG0a}
\end{align}

\noi
Let $T > 0$.
Then, given any $\uu \in \CC^\al([0, T]; H^s(\T^d))$, 
where
 $0 < \al < 1$ satisfies
$\g_0 + \al > 1$, 
the stochastic term $\PPsi (\uu)$ in \eqref{ZZ3} can be constructed as 
the Young integral
$\I^\YY(\uu)$ 
associated with the driver $\YY = \YY^{w, \phi}$ in \eqref{sto2}
such that 
\[\PPsi(\uu) = \I^\YY(\uu)
 \in \CC^{\g_0}([0, T]; H^{s_0}(\T^d)),\] 

\noi
almost surely.
 As a  consequence, with $u(t) = \sw(t)\uu(t)$, 
 the stochastic convolution $\Psi(u)(t) = \sw(t)\PPsi(\uu)(t)$
 in~\eqref{ZZ1}
belongs to 
$C([0, T]; H^{s_0}(\T^d))$, almost surely.

\medskip

\noi
{\rm (ii)
 (global well-posedness).}
Suppose that
the parameters satisfy
\eqref{SG00} and
\begin{align}
\max\Big( s\vee \kk - \s, \, \kk- \s \wedge 0 \Big)
< \frac 23 \rho(2\be - 1).
\label{SG0b}
\end{align}

\noi
Then, the stochastic modulated Schr\"odinger equation \eqref{NLS2x}
on  $\T^d$
with a multiplicative fractional-in-time noise 
with the Hurst parameter $\be \in \big(\frac 12 , 1\big)$
is pathwise globally well-posed in $H^s(\T^d)$.
Moreover, the modulated interaction representation~$\uu$
of the solution almost surely belongs
to $\CC^{ \g_0}(\R_+; H^s(\T^d))$
for any $ \g_0 \in \big(\frac 12 ,  \g\big] $ satisfying \eqref{YT1}.

%
%\medskip
%
%\noi
%\textup{(iii) (nonlinear smoothing).}
%In addition, 
%suppose that $s_0 > s$ satisfies \eqref{reg3}
%and \eqref{YT3}.
%Given $u_0 \in H^s(\T^d)$, 
%let $u
%\in C(\R_+; H^s(\T^d))$ be the  global-in-time solution 
%to the 
%stochastic modulated Schr\"odinger equation \eqref{NLS2x}
%on $\T^d$
%with $u|_{t = 0} = u_0$ constructed in Part (ii).
%Then, we have 
%\begin{align*}
% u - e^{iw(t) \Dl}u_0 \in C(\R_+; H^{s_0}(\T^d)).
%\end{align*}
%
%
%
%

\end{theorem}

Fix $\s \in \R$.
It follows from 
\eqref{SG00}, 
\eqref{SG0a},  and 
\eqref{SG0b} that 
Theorem \ref{THM:NLS}\,(i)
and (ii) 
holds
for {\it any} $s, s_0 \in \R$ and $\kk \ge 0$, 
provided that $\rho  \gg1$
is sufficiently large, 
exhibiting regularization by noise
for the stochastic modulated Schr\"odinger equation \eqref{NLS2x}.

As in the SmKdV case, 
Theorem \ref{THM:NLS}\,(i) 
 follows 
from Lemma \ref{LEM:int1}, 
once we establish
the corresponding  regularity properties
of the driver $\YY$  in \eqref{sto2} (see Proposition~\ref{PROP:drive4}), 
Also, pathwise 
local well-posedness in $H^s(\T^d)$ 
of \eqref{NLS2x} 
%Theorem \ref{THM:NLS}\,(ii) 
% on a local existence time 
follows from 
Proposition \ref{PROP:main}
(with $\XX = 0$).
We note that the condition \eqref{SG0b}
comes from setting $s_0 = s$ in \eqref{SG0a}.
In Subsection~\ref{SUBSEC:NLS1}, 
we study the driver $\YY$ in \eqref{sto2}.
In Subsection 
\ref{SUBSEC:NLS2},
we prove pathwise global well-posedness of~\eqref{NLS2x}
(Theorem \ref{THM:NLS}\,(ii)).

\begin{remark}\rm

(i)
As pointed out in Remark \ref{REM:deriv}, 
given any $s \in \R$, 
a straightforward modification of the proof of Theorem~\ref{THM:NLS}
yields pathwise global well-posedness 
in $H^s(\T^d)$ of the following stochastic modulated Schr\"odinger equation
on $\T^d$
with a multiplicative Young noise ($\frac 12 < \be < 1$):
\begin{align*}
i \dt u+  \Dl  u \cdot \dt w =  \jb{\nb}^{\kk_1}\big[(\jb{\nb}^{\kk_2} \cj u) \phi \zeta\big].
%\label{NLS2a}
\end{align*}

\noi
for any $\kk_1, \kk_2 \in \R $, 
provided that $\rho = \rho(\be,  s, \s, \kk_1, \kk_2)\gg1$
is sufficiently large.
We omit details.

\smallskip

\noi
(ii) In \eqref{NLS2x}, 
it is crucial that we have $\cj u$ in the noise.
If we instead have $u$ in the noise, 
then there is no regularization by noise
as in Theorem \ref{THM:NLS}, 
since, in this case, 
 $|n|^2 + |n_2|^2$ 
in~\eqref{sto2}
would be replaced by 
$|n|^2 - |n_2|^2$ 
which 
can be small or vanish.
Compare this with \eqref{SG3c} and \eqref{SG4}.

\end{remark}

\subsection{Regularity of the driver}
\label{SUBSEC:NLS1}

\begin{proposition}
\label{PROP:drive4}

Let 
$d\ge 1$ and $\kk \ge 0$.
Given 
$\rho > 0$, 
$\frac 12 < \g, \be < 1$,  $ \s \in \R$, and $T > 0$, 
let 
$w$ be 
 a $(\rho,\g)$-irregular function on $[0, T]$ in the sense of Definition~\ref{DEF:ir}
 and 
$\phi \in \HS(L^2(\T^d); H^\s(\T^d))$, satisfying
\eqref{phi1}, and 
let  $\YY = \YY^{w, \phi}$ be 
the random operator 
 defined in~\eqref{sto2}.
Suppose that 
$0 <  \g_0 \le  \g $ and $s, s_0 \in \R$ satisfy~\eqref{YT1}, 
\eqref{SG00}, and \eqref{SG0a}.
 Then, given any finite $p \ge 1$, we have 
\begin{align}
\big\| \|  \YY_{t,r}  \|_{\LOP(H^s; H^{s_0})} \big\|_{L^p(\Om)} 
 & \les  p^\frac 12 
 \|\Phi^w\|_{  \W^{\rho,\g}_T}\| \phi\|_{\HS(L^2;H^\s)} 
  |t-r|^{\g_0}
\label{SG1}
\end{align}

\noi
for any   $(t,r) \in \Dl_{2,T}$.
Consequently, we have 
\begin{align*}
\big\| \| \YY\|_{C^{\al}_{2,T} \LOP(H^s;H^{s_0})} \big\|_{L^p(\Om)} \les 
p^\frac 12 
 \|\Phi^w\|_{  \W^{\rho,\g}_T} \| \phi\|_{\HS(L^2;H^\s)} .
%\label{SG1a}
\end{align*}

\noi
In particular, 
there exists a version of $\YY$
such that 
$\YY \in C^{\g_0}_{2,T} \LOP(H^s(\T^d);H^{s_0}(\T^d))$
and
\begin{align*}
\YY \in  \cX^{s, s_0, \g_0}_{1, 0} ([0, T]\times \T^d) ,
\end{align*}

\noi
almost surely.
See Subsection \ref{SUBSEC:2.2}
for the definitions of these function spaces.

\end{proposition}

\begin{proof}

%We closely follow the proof of Proposition \ref{PROP:drive2}
%and focus on proving \eqref{SG1}.

As in the proof of Proposition \ref{PROP:drive2}, 
it suffices to prove \eqref{SG1}.
We closely follow the argument 
in the proof of Proposition \ref{PROP:drive2}.

Fix  $0 \le r <  t \le T$.
From \eqref{sto2}, we have 
\begin{align}
\Ft_x\big(\jb{\nb}^{s_0}\YY_{t,r} \jb{\nb}^{-s}f \big)(n) 
=  \sum_{n_2\in \Z} \cj{\ft f (n_2)} I_1[ \hf_{nn_2}(n_1) \ff_{nn_2}(t') ] , 
\label{SG2}
\end{align}

\noi
where $\hf_{nn_2}(n_1) $ and $\ff_{nn_2}$ are defined by 
\begin{align}
\begin{split}
\hf_{n n_2}(n_1) & = 
 \ind_{n=n_1 - n_2}
\cdot
\frac{\jb{n}^{s_0}}{\jb{n_2}^{s-\kk}} \phi_{n_1} , \\
\ff_{nn_2} (t')  &= 
\ff_{nn_2}^{t, r} (t')= 
 \ind_{[r, t]}(t')\cdot 
 e^{ i w(t') (|n|^2 + |n_2|^2)}.
\end{split}
\label{SG3}
\end{align}

\noi
Given dyadic $N, N_1, N_2 \ge1$, we set 
\begin{align}
\begin{split}
\hf^{\bf N}_{n n_2}(n_1) & = 
\hf^{N, N_1, N_2}_{n n_2}(n_1)
= \ind_{F_{\bf N}}
\cdot 
 \hf_{n n_2}(n_1), \\
\ff^{\bf N}_{nn_2} (t')  &= 
\ff^{N, N_1, N_2}_{nn_2} (t')  = 
\ind_{F_{\bf N}}
\cdot 
\ff_{nn_2} (t'),  
\end{split}
\label{SG3a}
\end{align}

\noi
where $F_{\bf N}$ is defined by 
\begin{align}
F_{\bf N} = \big\{(n, n_1, n_2) \in (\Z^d)^3:
n = n_1 - n_2, \, |n|\sim N, \, |n_j|\sim N_j, \, j = 1, 2 \big\}.
\label{SEN1}
\end{align}

\noi
Then, 
proceeding as in \eqref{YG3b}, 
given any $\ta > 0$, 
it follows from \eqref{SG2}, \eqref{SG3}, \eqref{SG3a}, 
and the random tensor estimate (Lemma~\ref{LEM:RT})
that 
\begin{align}
\begin{split}
 \big\| &  \| \YY_{t,r} \|_{\LOP(H^s;H^{s_0})} \big\|_{L^p(\Om)}\\
& \le \sum_{\substack{N, N_1, N_2 \ge1 \\\text{dyadic}}}  
\big\| \| 
I_1[ \hf^{\bf N}_{nn_2}(n_1) \ff^{\bf N}_{nn_2}(t')] \|_{\l^2_{n_2} \to \l^2_n} \big\|_{L^p(\Om)}\\
& \les 
p^\frac 12 
\sum_{\substack{N, N_1, N_2 \ge1 \\\text{dyadic}}}  
\frac{N_{\max}^{\frac d2 \ta} N^{\ta s_0}}
{N_1^{ \ta\s} N_2^{\ta (s-\kk)}}
\|\phi\|_{\HS(L^2; H^\s)}^\ta
 \|  \ff^{\bf N}_{nn_2} (t')\|_{\l^\infty_{nn_2} \Hs_{t'}^\be}\\
& \hphantom{XXXXX}\times 
 \max\Big( \|  \hf^{\bf N}_{nn_2}(n_1)\|_{n_1n_2 \to n} , 
 \|   \hf^{\bf N}_{nn_2}(n_1) \|_{n_2 \to n n_1} \Big)^{1-\ta}
\end{split}
\label{SG3b}
\end{align}

\noi
for any finite $p \ge 1$, 
where 
 $N_{\max} = \max(N, N_1, N_2)$ and 
 $\Hs^\be(\R_+)$ is as in \eqref{BM0}.

Proceeding as in \eqref{YG3x}
with \eqref{SG3a}, 
\eqref{SG3}, and 
Cauchy-Schwarz's inequality with \eqref{phi1a}, we have
\begin{align}
\begin{split}
 \|  \hf^{\bf N}_{nn_2}(n_1)\|_{n_1n_2 \to n} + 
\|  \hf^{\bf N}_{nn_2}(n_1)\|_{n_2 \to n n_1} 
&\les
\frac{N^{s_0}}{N_1^\s N_2^{s-\kk}}
\|\phi\|_{\HS(L^2; H^\s)}.
\end{split}
\label{SG3y}
\end{align}

Next,  we  estimate the 
$\Hs^\be_{t'}$-norm of 
$ \ff_{nn_2}(t')$
appearing in \eqref{SG3b}.
For notational simplicity, 
we drop the superscript $\bf N$
in the following but it is understood that $n = n_1 - n_2$ holds;
see~\eqref{SEN1}.
By integration by parts with~\eqref{rho2}, we have 
\begin{align}
\begin{split}
\ft \ff_{nn_2}(\tau) 
& =\frac 1{\sqrt {2\pi}} \int_r^t e^{-  i t'\tau }  e^{ i w(t') (|n|^2 + |n_2|^2)}dt'\\
&  = \frac 1{\sqrt {2\pi}}\Phi^w_{t, r}(|n|^2 + |n_2|^2) e^{-i t \tau}
+ 
\frac {i \tau}{\sqrt {2\pi}}
 \int_r^t 
\Phi^w_{t', r}(|n|^2 + |n_2|^2)e^{-i t'\tau} dt'.
\end{split}
\label{SG3c}
\end{align}

\noi
Then,  from \eqref{SG3c} 
with  
\eqref{rho1}, we have
\begin{align}
\begin{split}
|\ft \ff_{nn_2}(\tau) |
& \les
\|\Phi^w\|_{  \W^{\rho,\g}_T} 
\frac{|t - r|^\g}{ \jb{n_{\max}}^{2\rho}}
+ \|\Phi^w\|_{  \W^{\rho,\g}_T} 
|\tau |
\frac{|t - r|^{1+\g}}{ \jb{n_{\max}}^{2\rho}}.
\end{split}
\label{SG4}
\end{align}

\noi
under $n = n_1 - n_2$, 
where $n_{\max} = \max(|n|, |n_1|, |n_2|)$.

Let $G_1$ and $G_2$ be as in \eqref{YG5}.
From \eqref{SG4}, we have 
\begin{align}
G_1
\les 
\jb{T}
\|\Phi^w\|_{  \W^{\rho,\g}_T}
\frac{|t - r|^{\g}}{\jb{n_{\max}}^{2\rho}},
\label{SG6}
\end{align}

\noi
since $\be < 1$.
Proceeding as in \eqref{YG7}
with 
$\ld = \frac 13 (2\be -1) -\eps \in (0, 1)$ 
for some small $\eps > 0$
and 
$b = \ld^{-1}(1- 2\be)$
such that $b < -3$, 
and applying \eqref{SG3} and \eqref{SG4}, we have 
\begin{align}
\begin{split}
G_2 
& \le \bigg(\int_\R  |\ft \ff_{nn_2}(\tau) |^2d\tau\bigg)^{\frac{1-\ld} {2}}
\bigg(\int_{|\tau|>  1} |\tau|^b |\ft \ff_{nn_2}(\tau) |^2d\tau \bigg)^\frac \ld 2\\
& \les \jb{T}^\ld \|\Phi^w\|_{  \W^{\rho,\g}_T}^\ld 
|t- r|^{\frac{1-\ld}{2}}
\bigg(\frac{|t - r|^{\g}}{\jb{n_{\max}}^{2\rho}}\bigg)^{\ld}\\
& \les \jb{T} \jb{ \|\Phi^w\|_{  \W^{\rho,\g}_T}}
\frac {|t- r|^{\frac 12  + \frac 16 (2\be - 1)(2\g-1) -  \frac 12 (2\g-1)\eps}}
{ \jb{n_{\max}}^{\frac 23\rho (2 \be - 1 -3\eps)}}.
\end{split}
\label{SG7}
\end{align}

\noi
Hence, from \eqref{SG3a}, \eqref{YG5}, \eqref{SG6}, and \eqref{SG7}
with $\be < 1$,
we obtain
\begin{align}
\begin{split}
\|\ff^{\bf N}_{nn_2}\|_{\Hs^\be_{t'}}
& \le \ind_{F_N} \cdot \|\ff_{nn_2}\|_{\Hs^\be_{t'}}\\
& \les \ind_{F_N} \cdot \jb{T}\jb{ \|\Phi^w\|_{  \W^{\rho,\g}_T}}
\frac {|t- r|^{\frac 12  + \frac 16 (2\be - 1)(2\g-1) -  \frac 12 (2\g-1)\eps}}
{\jb{n_{\max}}^{\frac 23\rho (2 \be - 1 -3\eps)}}, 
\end{split}
\label{SG8}
\end{align}

\noi
uniformly in $n, n_2 \in \Z^d$ and dyadic $N, N_1, N_2 \ge 1$.

Therefore, putting 
\eqref{SG3b}, \eqref{SG3y}, 
and \eqref{SG8} together with \eqref{ord1}, we obtain
\begin{align}
\begin{split}
  \big\|  &  \| \YY_{t,r} \|_{\LOP(H^s;H^{s_0})} \big\|_{L^p(\Om)}\\
& \les 
p^\frac 12 
\jb{T}
\|\phi\|_{\HS(L^2; H^\s)}
\jb{\|\Phi^w\|_{  \W^{\rho,\g}_T}  }
|t- r|^{\frac 12 + \frac 16 (2\be - 1)(2\g-1) - \frac 12  (2\g-1)\eps}\\
%\sum_{\substack{N, N_1, N_2 \ge1 \\\text{dyadic}}}  
& \quad \times \sum_{\substack{N, N_1, N_2 \ge1 \\\text{dyadic}\\
N_{\max}\sim N_{\med}}}  
N_{\max}^{\frac d2 \ta
- \frac 23\rho (2 \be - 1 -3\eps)}
\frac{N^{s_0}}{N_1^{\s} N_2^{s-\kk}}\\
& \les 
p^\frac 12
\jb{T} 
\|\phi\|_{\HS(L^2; H^\s)}
\jb{\|\Phi^w\|_{  \W^{\rho,\g}_T}}  |t- r|^{\g_0}, 
\end{split}
\label{SG9}
\end{align}

\noi
provided that 
$\g_0 \le \g$, 
\eqref{YT1},
\eqref{SG00}, 
  and \eqref{SG0a} 
hold
and that $\ta, \eps > 0$ are sufficiently small.
As in \eqref{YG9},   the second inequality in \eqref{SG9} follows
from separately considering the cases:
 (a)~$N \sim N_1 \ges N_2$, 
(b)~$N \sim N_2 \ges N_1$, 
and 
(c)~$N_1 \sim N_2 \ges N$
(also depending on the sign of the exponent of $N_{\min}$).
This proves~\eqref{SG1}.
\end{proof}

\subsection{Global well-posedness}
\label{SUBSEC:NLS2}

In this subsection, we prove global well-posedness of \eqref{NLS2}
(Theorem \ref{THM:NLS}\,(ii)).

Fix a target time $T\gg 1$.
Our goal is to extend the local-in-time solution $\uu$
to the YDE~\eqref{NLS3}, 
constructed by Proposition \ref{PROP:main} with Proposition \ref{PROP:drive4},  
onto the entire interval $[0, T]$.
For simplicity, we assume that $\|u_0\|_{H^s} \ge 1$.
For $ \g_0 \in \big(\frac 12 ,  \g\big]$, 
let $0 < \al < \g_0$ such that $\al + \g_0 >1$.
Then, 
from~\eqref{YD5a} and~\eqref{YD14} in the proof of Proposition~\ref{PROP:main}
(with $\XX = 0$), we have 
\begin{align}
\begin{split}
\|\G(\uu)\|_{\CC^\al_\tau H^s_x}
& \le \|u_0\|_{H^s} 
+  C\tau^{\g_0 - \al}
\|\YY\|_{\cX^{s, \g_0}_1(T)}
\|\uu\|_{\CC^\al_{\tau } H^s_x}, \\
 \|\G(\uu^1)- \G(\uu^2)\|_{\cC^\al _{\tau} H^s_x}
&  \les 
\tau^{\g_0 - \al}
\|\YY\|_{\cX^{s, \g_0}_1(T)}\|\uu^1- \uu^2 \|_{\CC^\al _\tau H^s_x}.
\end{split}
\label{SG10}
\end{align}

\noi
for $0 < \al < \g_0$.
In carrying out a contraction argument in 
$\cC^\al([0, \tau]; H^s(\T^d))$, 
we see from~\eqref{SG10}
that the local existence time $\tau = \tau \big(\|\YY\|_{\cX^{s, \g_0}_1(T)}\big) >0$
can be chosen to be  independent  of initial data $u_0$.
%Moreover, it follows from the local theory that on each time interval $[(j-1)\tau, j\tau]$ of local existence, 
%the size of the solution at most doubles.
Hence, by  iteratively applying the contraction argument 
on each time interval $[j\tau, (j+1)\tau]\cap [0, T]$, $j = 0,1,  \dots, \big[\frac T\tau\big] $, 
where $[x]$ denotes the integer part of $x \in \R$, 
we can extend the solution $\uu$ to \eqref{NLS3}
onto the entire interval $[0, T]$.
The solution~$\uu$ thus constructed belongs to 
$\cC^\al([0, T]; H^s(\T^d))$.
By applying the second bound of \eqref{Ja3}
in Lemma \ref{LEM:int1} to \eqref{NLS3}, 
we then conclude that 
$\uu \in \cC^{\g_0}([0, T]; H^s(\T^d))$.
Since the choice of $T\gg1$ was arbitrary, 
this proves global well-posedness of \eqref{NLS3}
(and hence of~\eqref{NLS2x}).
This concludes the proof of  Theorem \ref{THM:NLS}.

\begin{ackno}\rm
A.C.~received support from CNRS-INSMI through a grant ``PEPS Jeunes chercheurs et jeunes chercheuses 2025''.
M.G.~was supported by 
the UKRI Frontier Research Grant 
(grant no.~EP/Z534328/1 ``Stochastic analysis of quantum fields"). 
G.L. was supported by the NSFC (grant no.~12501181).
T.O.~was supported by the European Research Council (grant no.~864138 ``SingStochDispDyn")
and  by the EPSRC 
Mathematical Sciences
Small Grant  (grant no.~EP/Y033507/1).
T.O. also 
acknowledges support from  
the NSFC (grant no.~W2531005).

\end{ackno}

%\medskip
%
%\noi
%{\bf Conflict of  interest.}
%The authors have no conflict of interest  to declare that are relevant to the content of this article.
%
%
%\medskip
%
%\noi
%{\bf Data availability statement.}
%This manuscript has no associated data.

\end{document}